\documentclass[english,12pt]{article}

\usepackage[T1]{fontenc}
\usepackage[latin9]{inputenc}
\usepackage{geometry}
\usepackage{babel}
\usepackage{float}
\usepackage{mathrsfs}
\usepackage{mathtools}
\usepackage{amsmath}
\usepackage{amsthm}
\usepackage{amssymb}
\usepackage{setspace}
\usepackage[unicode=true,pdfusetitle,
 bookmarks=true,bookmarksnumbered=false,bookmarksopen=false,
 breaklinks=false,pdfborder={0 0 0},pdfborderstyle={},backref=false,colorlinks=false]
 {hyperref}

\makeatletter
\numberwithin{equation}{section}
\numberwithin{figure}{section}
\theoremstyle{plain}
\newtheorem{thm}{\protect\theoremname}[section]
\theoremstyle{definition}
\newtheorem{defn}[thm]{\protect\definitionname}
\newtheorem{summ}[thm]{\protect{Summary}}
\newtheorem{rem}[thm]{\protect\remarkname}
\theoremstyle{plain}
\newtheorem{prop}[thm]{\protect\propositionname}
\theoremstyle{definition}
\newtheorem{example}[thm]{\protect\examplename}
\theoremstyle{plain}
\newtheorem{cor}[thm]{\protect\corollaryname}
\theoremstyle{plain}
\newtheorem{lem}[thm]{\protect\lemmaname}

\@ifundefined{date}{}{\date{}}

\usepackage{tikz-cd}
\usepackage{quiver}
\usepackage{tikz}
\usepackage{tikzit}

\tikzstyle{nc}=[fill=white, draw=black, shape=circle, minimum size=0.40cm]
\tikzstyle{nr}=[fill=white, draw=black, shape=rectangle, minimum width=0.60cm, minimum height=0.60cm]
\tikzstyle{nrg}=[fill={rgb,255: red,199; green,199; blue,199}, draw=black, shape=rectangle]
\tikzstyle{sr}=[fill=white, draw=black, shape=rectangle, minimum width=0.60cm, minimum height=0.40cm, inner sep=0pt]
\tikzstyle{lsr}=[fill=white, draw=black, shape=rectangle, minimum width=0.60cm, minimum height=0.40cm, inner sep=0pt, rotate=30]
\tikzstyle{rsr}=[fill=white, draw=black, shape=rectangle, minimum width=0.60cm, minimum height=0.40cm, inner sep=0pt, rotate=-30]
\tikzstyle{rrrsr}=[fill=white, draw=black, shape=rectangle, minimum width=0.60cm, minimum height=0.40cm, inner sep=0pt, rotate=-120]
\tikzstyle{ivsr}=[fill=none, draw=none, shape=rectangle, minimum width=0.60cm, minimum height=0.40cm, inner sep=0pt]
\tikzstyle{ivlsr}=[fill=none, draw=none, shape=rectangle, minimum width=0.60cm, minimum height=0.40cm, inner sep=0pt, rotate=30]
\tikzstyle{ivrsr}=[fill=none, draw=none, shape=rectangle, minimum width=0.60cm, minimum height=0.40cm, inner sep=0pt, rotate=-30]
\tikzstyle{ivrrrsr}=[fill=none, draw=none, shape=rectangle, minimum width=0.60cm, minimum height=0.40cm, inner sep=0pt, rotate=-120]
\tikzstyle{mr}=[fill=white, draw=black, shape=rectangle, minimum width=1.00cm, minimum height=0.50cm, inner sep=0pt]
\tikzstyle{product}=[fill={rgb,255: red,255; green,5; blue,80}, draw=black, shape=circle, inner sep=0pt, minimum size=0.15cm]
\tikzstyle{unit}=[fill={rgb,255: red,255; green,166; blue,217}, draw=black, shape=circle, inner sep=0pt, minimum size=0.15cm]
\tikzstyle{coproduct}=[fill={rgb,255: red,2; green,145; blue,255}, draw=black, shape=circle, inner sep=0pt, minimum size=0.15cm]
\tikzstyle{counit}=[fill={rgb,255: red,165; green,219; blue,255}, draw=black, shape=circle, inner sep=0pt, minimum size=0.15cm]
\tikzstyle{antipode}=[fill=white, draw=black, shape=circle, inner sep=0pt, minimum size=0.15cm]
\tikzstyle{region}=[fill={rgb,255: red,191; green,191; blue,191}, draw={rgb,255: red,191; green,191; blue,191}, shape=circle, minimum size=0.80cm]
\tikzstyle{boundarydisc}=[fill={rgb,255: red,128; green,128; blue,128}, draw=black, shape=circle, minimum size=0.80cm]
\tikzstyle{identity}=[fill=black, draw=black, shape=circle, inner sep=0pt, minimum size=0.15cm]
\tikzstyle{S}=[fill=white, draw=black, shape=rectangle, minimum width=0.15cm, minimum height=0.15cm]
\tikzstyle{S-}=[fill={rgb,255: red,199; green,199; blue,199}, draw=black, shape=rectangle, minimum width=0.15cm, minimum height=0.15cm]
\tikzstyle{dot}=[fill=white, draw=black, shape=circle, inner sep=0pt, minimum size=0.07cm]
\tikzstyle{bdot}=[fill=black, draw=black, shape=circle, inner sep=0pt, minimum size=0.07cm]
\tikzstyle{m}=[fill={rgb,255: red,255; green,5; blue,80}, draw={rgb,255: red,255; green,5; blue,80}, shape=circle, inner sep=0pt, minimum size=0.14cm, tikzit shape=circle]
\tikzstyle{mod}=[fill={rgb,255: red,172; green,229; blue,232}, draw={rgb,255: red,14; green,115; blue,177}, shape=circle]
\tikzstyle{smod}=[fill={rgb,255: red,172; green,229; blue,232}, draw={rgb,255: red,14; green,115; blue,177}, shape=circle, inner sep=0pt, minimum size=0.15cm]
\tikzstyle{smod1}=[fill={rgb,255: red,150; green,255; blue,138}, draw={rgb,255: red,5; green,130; blue,3}, shape=circle, inner sep=0pt, minimum size=0.15cm]
\tikzstyle{smod2}=[fill={rgb,255: red,255; green,202; blue,239}, draw={rgb,255: red,126; green,0; blue,124}, shape=circle, inner sep=0pt, minimum size=0.15cm]
\tikzstyle{smod3}=[fill={rgb,255: red,255; green,213; blue,164}, draw={rgb,255: red,121; green,57; blue,2}, shape=circle, inner sep=0pt, minimum size=0.15cm]
\tikzstyle{mid-nc}=[fill=white, draw=black, shape=circle, inner sep=0pt, minimum size=0.25cm]

\tikzstyle{di}=[draw=black, ->]
\tikzstyle{da}=[-, dashed]
\tikzstyle{da-di}=[dashed, ->]
\tikzstyle{th}=[-, thick]
\tikzstyle{th-di}=[->, thick]
\tikzstyle{th-da}=[-, dashed, thick]
\tikzstyle{gr}=[-, draw={rgb,255: red,0; green,186; blue,143}, thick]
\tikzstyle{gr-di}=[draw={rgb,255: red,0; green,186; blue,143}, ->, thick]
\tikzstyle{gr-da}=[-, dashed, draw={rgb,255: red,0; green,186; blue,143}, thick]
\tikzstyle{pu}=[-, draw={rgb,255: red,124; green,95; blue,239}, thick]
\tikzstyle{pu-da}=[-, dashed, draw={rgb,255: red,124; green,95; blue,239}, thick]
\tikzstyle{re}=[-, draw={rgb,255: red,255; green,5; blue,80}, thick]
\tikzstyle{re-di}=[->, draw={rgb,255: red,255; green,5; blue,80}, thick]
\tikzstyle{re-da}=[-, draw={rgb,255: red,255; green,5; blue,80}, dashed, thick]
\tikzstyle{bl}=[-, draw={rgb,255: red,14; green,115; blue,177}, thick]
\tikzstyle{bl-di}=[draw={rgb,255: red,14; green,115; blue,177}, ->, thick]
\tikzstyle{bl-da}=[-, draw={rgb,255: red,14; green,115; blue,177}, dashed, thick]
\tikzstyle{lgr}=[-, draw={rgb,255: red,188; green,220; blue,156}, thick]
\tikzstyle{lgr-da}=[-, draw={rgb,255: red,188; green,220; blue,156}, thick, dashed]
\tikzstyle{lbl}=[-, draw={rgb,255: red,168; green,207; blue,245}, thick]
\tikzstyle{lbl-da}=[-, draw={rgb,255: red,168; green,207; blue,245}, thick, dashed]
\tikzstyle{pi}=[-, draw={rgb,255: red,208; green,160; blue,160}, thick]
\tikzstyle{pi-da}=[-, draw={rgb,255: red,208; green,160; blue,160}, thick, dashed]
\tikzstyle{sh}=[-, draw={rgb,255: red,171; green,171; blue,171}]
\tikzstyle{sh-da}=[-, draw={rgb,255: red,171; green,171; blue,171}, dashed]
\tikzstyle{fi-gr}=[-, fill={rgb,255: red,194; green,228; blue,162}, draw={rgb,255: red,194; green,228; blue,162}]
\tikzstyle{fi-bl}=[-, fill={rgb,255: red,175; green,215; blue,255}, draw={rgb,255: red,175; green,215; blue,255}]
\tikzstyle{fi-pi}=[-, fill={rgb,255: red,246; green,188; blue,188}, draw={rgb,255: red,246; green,188; blue,188}]
\tikzstyle{fi-pu}=[-, fill={rgb,255: red,230; green,203; blue,246}, draw={rgb,255: red,230; green,203; blue,246}]
\tikzstyle{fi-ye}=[-, fill={rgb,255: red,255; green,255; blue,155}, draw={rgb,255: red,255; green,255; blue,155}]
\tikzstyle{fi-or}=[-, fill={rgb,255: red,255; green,197; blue,51}, draw={rgb,255: red,255; green,197; blue,51}]
\tikzstyle{fi-dg}=[-, fill={rgb,255: red,152; green,170; blue,139}, draw={rgb,255: red,152; green,170; blue,139}]
\tikzstyle{fi-db}=[-, fill={rgb,255: red,111; green,147; blue,183}, draw={rgb,255: red,111; green,147; blue,183}]
\tikzstyle{fi-sh}=[-, fill={rgb,255: red,222; green,222; blue,222}, draw={rgb,255: red,222; green,222; blue,222}]
\tikzstyle{fi-dsh}=[-, fill={rgb,255: red,204; green,204; blue,204}, draw={rgb,255: red,204; green,204; blue,204}]
\tikzstyle{vt}=[-, very thick]
\tikzstyle{vt-di}=[->, very thick]
\tikzstyle{vt-da}=[-, dashed, very thick]
\tikzstyle{vt-gr}=[-, draw={rgb,255: red,0; green,186; blue,143}, very thick]
\tikzstyle{vt-gr-di}=[draw={rgb,255: red,0; green,186; blue,143}, ->, very thick]
\tikzstyle{vt-gr-da}=[-, dashed, draw={rgb,255: red,0; green,186; blue,143}, very thick]
\tikzstyle{vt-pu}=[-, draw={rgb,255: red,124; green,95; blue,239}, very thick]
\tikzstyle{vt-pu-da}=[-, dashed, draw={rgb,255: red,124; green,95; blue,239}, very thick]
\tikzstyle{vt-re}=[-, draw={rgb,255: red,255; green,5; blue,80}, very thick]
\tikzstyle{vt-re-di}=[->, draw={rgb,255: red,255; green,5; blue,80}, very thick]
\tikzstyle{vt-re-da}=[-, draw={rgb,255: red,255; green,5; blue,80}, dashed, very thick]
\tikzstyle{vt-bl}=[-, draw={rgb,255: red,14; green,115; blue,177}, very thick]
\tikzstyle{vt-bl-di}=[draw={rgb,255: red,14; green,115; blue,177}, ->, very thick]
\tikzstyle{vt-bl-da}=[-, draw={rgb,255: red,14; green,115; blue,177}, dashed, very thick]
\tikzstyle{vt-bl-da-di}=[draw={rgb,255: red,14; green,115; blue,177}, ->, very thick, dashed]
\tikzstyle{vt-lgr}=[-, draw={rgb,255: red,188; green,220; blue,156}, very thick]
\tikzstyle{vt-lgr-di}=[draw={rgb,255: red,188; green,220; blue,156}, ->, very thick]
\tikzstyle{vt-lgr-da}=[-, draw={rgb,255: red,188; green,220; blue,156}, very thick, dashed]
\tikzstyle{vt-lbl}=[-, draw={rgb,255: red,168; green,207; blue,245}, very thick]
\tikzstyle{vt-lbl-di}=[->, draw={rgb,255: red,168; green,207; blue,245}, very thick]
\tikzstyle{vt-lbl-da}=[-, draw={rgb,255: red,168; green,207; blue,245}, very thick, dashed]
\tikzstyle{vt-lpp}=[-, draw={rgb,255: red,223; green,199; blue,240}, very thick]
\tikzstyle{vt-pi}=[-, draw={rgb,255: red,236; green,151; blue,151}, very thick]
\tikzstyle{vt-pi-da}=[-, draw={rgb,255: red,236; green,151; blue,151}, very thick, dashed]
\tikzstyle{vt-ye}=[-, draw={rgb,255: red,241; green,241; blue,115}, very thick]
\tikzstyle{vt-or}=[-, draw={rgb,255: red,245; green,150; blue,32}, very thick]
\tikzstyle{vt-sh}=[-, fill=none, draw={rgb,255: red,171; green,171; blue,171}, very thick]
\tikzstyle{vt-sh-di}=[draw={rgb,255: red,171; green,171; blue,171}, very thick, ->]
\tikzstyle{fi-bx}=[-, fill={rgb,255: red,172; green,229; blue,232}, draw={rgb,255: red,14; green,115; blue,177}, thick]
\tikzstyle{ut-gr}=[-, draw={rgb,255: red,0; green,186; blue,143}, line width=3pt]
\tikzstyle{ut-gr-di}=[draw={rgb,255: red,0; green,186; blue,143}, ->, line width=3pt]
\tikzstyle{ut-ig}=[-, draw={rgb,255: red,255; green,5; blue,80}, line width=3pt]
\tikzstyle{ut-ig-di}=[draw={rgb,255: red,255; green,5; blue,80}, line width=3pt, ->]
\tikzstyle{fi-lbl}=[-, fill={rgb,255: red,213; green,237; blue,255}, draw={rgb,255: red,213; green,237; blue,255}]
\tikzstyle{fi-vlbl}=[-, fill={rgb,255: red,229; green,251; blue,255}, draw={rgb,255: red,229; green,251; blue,255}]
\tikzstyle{op-sh}=[-, draw={rgb,255: red,113; green,113; blue,113}, fill={rgb,255: red,113; green,113; blue,113}, opacity=0.5]

\usepackage{lmodern}
\usepackage[T1]{fontenc}
\usepackage{fancyhdr}
\usepackage{bbm}
\usepackage[export]{adjustbox}
    \usepackage{graphicx}
\usepackage[hang,flushmargin]{footmisc}
\allowdisplaybreaks

\makeatletter

             \newcommand\Cite[2] {\cite[#1]{#2}}

\begin{document}

\makeatletter
\newcommand{\pgets}{}
\DeclareRobustCommand{\pgets}{\mathrel{\mathpalette\p@to@gets\gets}}
\newcommand{\p@to@gets}[2]{%
  \ooalign{\hidewidth$\m@th#1\mapstochar\mkern5mu$\hidewidth\cr$\m@th#1\to$\cr}%
}
\makeatother

\makeatletter
\newcommand{\Pgets}{}
\DeclareRobustCommand{\Pgets}{\mathrel{\mathpalette\P@to@gets\gets}}
\newcommand{\P@to@gets}[2]{%
\ooalign{\hidewidth$\m@th#1\mapstochar\mkern3mu$\hidewidth\cr$\m@th#1\longrightarrow$\cr}%
}
\makeatother

\usetikzlibrary{decorations,arrows}
\usetikzlibrary{decorations.markings}
\usetikzlibrary{decorations.pathmorphing}
\usetikzlibrary{matrix,arrows}

\providecommand{\corollaryname}{Corollary}
\providecommand{\definitionname}{Definition}
\providecommand{\examplename}{Example}
\providecommand{\lemmaname}{Lemma}
\providecommand{\propositionname}{Proposition}
\providecommand{\remarkname}{Remark}
\providecommand{\theoremname}{Theorem}

\newcommand\Subsection[1] {\addtocounter{thm}{1}
	\def\thesubsection{\arabic{section}.\arabic{thm}}\subsection{#1}}


\global\long\def\-{\text{-}}%
\global\long\def\lact{\triangleright}%
\global\long\def\ract{\triangleleft}%
\global\long\def\can{\mathrm{can}}%
\global\long\def\cat{\mathrm{Cat}}%
\global\long\def\ccat{\mathcal Cat}%
\global\long\def\ca{\mathcal{A}}%
\global\long\def\cb{\mathcal{B}}%
\global\long\def\cbc{\mathcal{BC}}%
\global\long\def\cc{\mathcal{C}}%
\global\long\def\calb{\GCal_{\cb}}%
\global\long\def\pcalb{\GCal_{\cb'}}%
\global\long\def\conncalb{\GCal_{\cb}^{\mathrm{conn}}}%
\global\long\def\pconncalb{\GCal_{\cb'}^{\mathrm{conn}}}%
\global\long\def\hcalb{\mathsf{\widehat{\GCal}}_{\cb}}%
\global\long\def\phcalb{\mathsf{\widehat{\GCal}}_{\cb'}}%
\global\long\def\connhcalb{\mathsf{\widehat{\GCal}}_{\cb}^{\mathrm{conn}}}%
\global\long\def\pconnhcalb{\mathsf{\widehat{\GCal}}_{\cb'}^{\mathrm{conn}}}%
\global\long\def\cyl{\mathrm{Cyl}}%
\global\long\def\ocyl{\mathrm{Cyl}^{\circ}}%
\global\long\def\dblcc{\cc^{\mathrm{rev}}\boxtimes\cc}%
\global\long\def\dbl{\mathcal{D}\mathrm{bl}}%
\global\long\def\cd{\mathcal{D}}%
\global\long\def\ce{\mathcal{E}}%
\global\long\def\cg{\mathcal{G}}%
\global\long\def\ch{\mathcal{H}}%
\global\long\def\ci{\mathcal{I}}%
\global\long\def\cm{\mathcal{M}}%
\global\long\def\cn{\mathcal{N}}%
\global\long\def\cu{\mathcal{U}}%
\global\long\def\cv{\mathcal{V}}%
\global\long\def\cw{\mathcal{W}}%
\global\long\def\cz{\mathcal{Z}}%
\global\long\def\czc{\mathcal{Z}(\mathcal{C})}%
\global\long\def\cfrc{\mathcal{F}r(\mathcal{C})}%
\global\long\def\corc{\mathrm{Cor}_{\cc}}%
\global\long\def\opcorc{\mathsf{Cor}_{\cc}}%
\global\long\def\ucorc{\mathrm{UCor}_{\cc}}%
\global\long\def\crlla{\mathsf{Crll}_{\cb}}%
\global\long\def\pcrlla{\mathsf{Crll}_{\cb'}}%
\global\long\def\conncrlla{\mathsf{Crll}_{\cb}^{\mathrm{conn}}}%
\global\long\def\pconncrlla{\mathsf{Crll}_{\cb'}^{\mathrm{conn}}}%
\global\long\def\dcftc{\mathbb{C}\mathrm{ft}_{\cc}}%
\global\long\def\da{\mathbb{A}}%
\global\long\def\db{\mathbb{B}}%
\global\long\def\dc{\mathbb{C}}%
\global\long\def\dd{\mathbb{D}}%
\global\long\def\de{\mathbb{E}}%
\global\long\def\df{\mathbb{F}}%
\global\long\def\bbord{\mathcal{B}\mathrm{ord}_{2,\mathrm{o/c}}^{\mathrm{or}}}%
\global\long\def\bbordc{\mathcal{B}\mathrm{ord}_{2}^{\mathrm{or}}}%
\global\long\def\dbord{\mathbb{B}\mathrm{ord}_{2,\mathrm{o/c}}^{\mathrm{or}}}%
\global\long\def\bprof{\mathcal{P}\mathrm{rof}_{\mathbbm{k}}}%
\global\long\def\dprof{\mathbb{P}\mathrm{rof}_{\mathbbm{k}}}%
\global\long\def\bl{\mathrm{Bl}}%
\global\long\def\blc{\mathrm{Bl}_{\cc}}%
\global\long\def\dblc{\mathbb{B}\mathrm{l}_{\cc}}%
\global\long\def\dim{\mathrm{dim}}%
\global\long\def\ev{\mathrm{ev}}%
\global\long\def\coev{\mathrm{coev}}%
\global\long\def\fk{\mathbbm{k}}%
\global\long\def\grph{\varGamma}%
\global\long\def\ogrph{\mathring{\varGamma}}%
\global\long\def\gr{\mathcal{G}\mathrm{raphs}}%
\global\long\def\grb{\gr_{\cb}}%
\global\long\def\eend{\mathrm{End}}%
\global\long\def\hom{\mathrm{Hom}}%
\global\long\def\ihom{\underline{\hom}}%
\global\long\def\id{\mathrm{id}}%
\global\long\def\iso{\stackrel{\cong}{\longrightarrow}}%
\global\long\def\eqv{\stackrel{\simeq}{\longrightarrow}}%
\global\long\def\kar{\mathrm{Kar}}%
\global\long\def\lim{\mathrm{lim}\,}%
\global\long\def\colim{\mathrm{colim}\,}%
\global\long\def\mcg{\mathrm{Map}}%
\global\long\def\hmcg{\widehat{\mathrm{Map}}}%
\global\long\def\bimod{\text{-}\mathrm{mod}\text{-}}%
\global\long\def\bimodc{\text{-}\mathrm{mod}^{\mathcal{C}}\text{-}}%
\global\long\def\rmod{\mathrm{mod}\text{-}}%
\global\long\def\lmodc{\text{-}\mathrm{mod}^{\mathcal{C}}}%
\global\long\def\rmodc{\mathrm{mod}^{\mathcal{C}}\text{-}}%
\global\long\def\Mod{\text{-}\mathcal{M}\mathrm{od}}%
\global\long\def\biMod{\text{-}\mathcal{M}\mathrm{od}\text{-}}%
\global\long\def\clMod{\mathcal{C}\text{-}\mathcal{M}\mathrm{od}}%
\global\long\def\clModtr{\mathcal{C}\text{-}\mathcal{M}\mathrm{od}^{\mathrm{tr}}}%
\global\long\def\mver{\underline{\mu}_{\mathrm{ver}}}%
\global\long\def\lmhor{\underline{\mu}_{\mathrm{hor}}^{\mathrm{l}}}%
\global\long\def\rmhor{\underline{\mu}_{\mathrm{hor}}^{\mathrm{r}}}%
\global\long\def\msa{\mathsf{a}}%
\global\long\def\msb{\mathsf{b}}%
\global\long\def\msc{\mathsf{c}}%
\global\long\def\msA{\mathsf{A}}%
\global\long\def\msB{\mathsf{B}}%
\global\long\def\msC{\mathsf{C}}%
\global\long\def\omsA{\mathsf{A}^{\circ}}%
\global\long\def\omsB{\mathsf{B}^{\circ}}%
\global\long\def\omsC{\mathsf{C}^{\circ}}%
\global\long\def\nat{\mathrm{Nat}}%
\global\long\def\inat{\underline{\mathrm{Nat}}}%
\global\long\def\op{\mathbin{\mathrm{op}}}%
\global\long\def\co{\mathbin{\mathrm{co}}}%
\global\long\def\coop{\mathbin{\mathrm{co\,op}}}%
\global\long\def\ob{\mathrm{obj}\,}%
\global\long\def\mor{\mathrm{mor}\,}%
\global\long\def\nmor{\-\mathrm{mor}\,}%
\global\long\def\ot{\mathbin{\otimes}}%
\global\long\def\ota{\mathbin{\otimes_{A}}}%
\global\long\def\otb{\mathbin{\otimes_{B}}}%
\global\long\def\otk{\mathbin{\otimes_{\fk}}}%
\global\long\def\pcirc{\diamond}%
\global\long\def\pto{\pgets}%
\global\long\def\Pto{\Pgets}%
\global\long\def\rin{\mathrm{in}}%
\global\long\def\rout{\mathrm{out}}%
\global\long\def\rl{\mathrm{l}}%
\global\long\def\rr{\mathrm{r}}%
\global\long\def\set{\mathrm{Set}}%
\global\long\def\surf{\varSigma}%
\global\long\def\snb{\mathrm{SN}_{\cb}}%
\global\long\def\csnb{\mathcal{S}\mathrm{N}_{\cb}}%
\global\long\def\csnbc{\mathcal{S}\mathrm{N}_{\cbc}}%
\global\long\def\csnbp{\mathcal{S}\mathrm{N}_{\cb'}}%
\global\long\def\dsnb{\mathbb{S}\mathrm{N}_{\cb}}%
\global\long\def\osnb{\mathrm{SN}_{\cb}^{\circ}}%
\global\long\def\cosnb{\mathcal{S}\mathrm{N}_{\cb}^{\circ}}%
\global\long\def\posnb{\mathrm{SN}_{\cb'}^{\circ}}%
\global\long\def\dosnb{\mathbb{S}\mathrm{N}_{\cb}^{\circ}}%
\global\long\def\osnfrc{\mathrm{SN}_{\cfrc}^{\circ}}%
\global\long\def\cosnfrc{\mathcal{S}\mathrm{N}_{\cfrc}^{\circ}}%
\global\long\def\dosnfrc{\mathbb{S}\mathrm{N}_{\cfrc}^{\circ}}%
\global\long\def\snc{\mathrm{SN}_{\cc}}%
\global\long\def\csnc{\mathcal{S}\mathrm{N}_{\cc}}%
\global\long\def\dsnc{\mathbb{S}\mathrm{N}_{\cc}}%
\global\long\def\osnc{\mathrm{SN}_{\cc}^{\circ}}%
\global\long\def\pop{\surf_{\mathrm{p.o.p.}}}%
\global\long\def\tr{\mathrm{tr}}%
\global\long\def\tu{\mathbbm{1}}%
\global\long\def\vct{\mathrm{Vect}_{\fk}}%
\global\long\def\vctfd{\mathrm{Vect}_{\fk}^{\mathrm{fd}}}%
\global\long\def\opsnc{\mathsf{SN}_{\cc}}%
\global\long\def\opws{\mathsf{WS}_{\cc}}%
\global\long\def\ophzc{\mathsf{Hom}_{\czc}}%
\global\long\def\oph{\mathsf{Hom}}%
\global\long\def\ws{\mathcal{S}}%
\global\long\def\cws{\widetilde{\mathcal{S}}}%
\global\long\def\uws{\breve{\mathcal{S}}}%
\global\long\def\wsf{\mathfrak{S}}%
\global\long\def\sws{\surf_{\ws}}%
\global\long\def\tw{\mathrm{Tw}}%
\global\long\def\ver{\mathrm{ver}}%
\global\long\def\hor{\mathrm{hor}}%
\global\long\def\vla{\xrightarrow{\hspace*{1.2cm}}}%


\def\be            {\begin{equation}}
\def\bearl         {\begin{array}{l}}
\def\boti          {\,{\boxtimes}\,}
\def\cdo           {\,{\cdot}\,}
\def\cir           {\,{\circ}\,}
\def\clc           {{\mathrm c}}
\def\Colon         {:\quad}
\def\Coloneqq      {\,{\coloneqq}\,}
\def\complex       {{\mathbbm C}} 
\def\Dots          {\,{\dots}\,}
\def\Dotsb         {\,{\dotsb}\,}
\def\ee            {\end{equation}}
\def\eear          {\end{array}}
\def\eq            {\,{=}\,}
\def\Enumerate     {\def\leftmargini{1.84em}~\\[-1.42em]\begin{enumerate}}
\newcommand\equ[1] {\stackrel{\eqref{#1}}=}
\def\Gama          {\grph} 
\def\GAMA          {[\Gama]}
\def\GCal          {\mathsf{GCal}}
\def\hoc           {\hspace*{0.1pt}}     
\def\hoco          {\star}  
\def\hocO          {\,{\hoco}\,}
\def\iN            {\,{\in}\,}    
\def\Itemize       {\def\leftmargini{2.23em}~\\[-1.35em]\begin{itemize}\addtolength\itemsep{-6pt}}
\def\Itemizej      {\def\leftmargini{1.33em}~\\[-1.35em]\begin{itemize}\addtolength\itemsep{-6pt}}
\def\Otb           {\,{\otb}\,}
\def\oti           {\,{\otimes}\,}
\def\Otk           {\,{\otk}\,}
\def\ptO           {\,{\pto}\,}
\def\PtO           {\,{\Pto}\,}
\newcommand\rarr[1]{\xrightarrow{~#1~}}
\newcommand\Rarr[1]{\,{\xrightarrow{\,#1\,}}\,}
\def\Rightarro     {\,{\xRightarrow{\,~\,}}\,}
\def\sss           {\scriptscriptstyle }
\def\Times         {\,{\times}\,}
\def\To            {\,{\to}\,}
\def\Vee           {{}^{\vee\!}}
\newcommand\void[1] {}


\begin{flushright}  
   {\sf ZMP-HH/23-1}\\
   {\sf Hamburger$\;$Beitr\"age$\;$zur$\;$Mathematik$\;$Nr.$\;$937}\\[2mm]
   ~ 
\end{flushright}

\vskip 2.0em

\begin{center}
{\bf \Large String-net models for pivotal bicategories}

\vskip 15mm

{\large \  \ J\"urgen Fuchs\,$^{\,a},~\quad$ Christoph Schweigert\,$^{\,b}~\quad$
and $\quad$ Yang Yang\,$^{\,c,d}$ }

\vskip 12mm

 \it$^a$
 Teoretisk fysik, \ Karlstads Universitet\\
 Universitetsgatan 21, \ S\,--\,651\,88\, Karlstad
 \\[9pt]
 \it$^b$
 Fachbereich Mathematik, \ Universit\"at Hamburg\\
 Bereich Algebra und Zahlentheorie\\
 Bundesstra\ss e 55, \ D\,--\,20\,146\, Hamburg
 \\[9pt]
 \it$^c$
Mathematisches Forschungsinstitut Oberwolfach\\
Schwarzwaldstra\ss e 9-11, \ D\,--\,77\,709\, Oberwolfach-Walke
 \\[9pt]
 \it$^d$
Erwin Schr\"odinger International Institute for Mathematics and Physics \\
Boltzmanngasse 9, \ A\,--\,1090\, Wien

\end{center}

\vskip 4.2em

\noindent{\sc Abstract}\\[3pt]
We develop a string-net construction of a modular functor whose algebraic input
is a pivotal bicategory; this extends the standard construction based on a spherical
fusion category. An essential ingredient in our construction is a graphical calculus 
for pivotal bicategories, which
we express in terms of a category of colored corollas. The globalization of this
calculus to oriented surfaces yields the bicategorical string-net spaces as colimits.
We show that every rigid separable Frobenius functor between strictly pivotal
bicategories induces linear maps between the corresponding bicategorical string-net
spaces that are compatible with the mapping class group actions and with sewing.
Our results are inspired by and have applications to the description of
correlators in two-dimensional conformal field theories.

\newpage
\tableofcontents{}
\newpage

\section{Introduction}

For several decades, skein theoretic constructions have played a prominent role 
in quantum topology, see e.g.\ \cite{bhmv,walk4} for early work. Such constructions
are appealing for several reasons: They exist, for suitable algebraic input data,
in various dimensions, and they provide a rather direct relation between these
algebraic data and geometry, whereby they afford in particular geometric actions 
of mapping class groups. Moreover, skein theoretic constructions can typically be
extended to higher codimensions, so that they can provide extended topological 
field theories. For instance, they allow one to construct interesting categories
by evaluating the algebraic input data on manifolds of codimension 2, see e.g.\
\cite{hoek} for categories associated to one-manifolds (cylinder categories)
and e.g.\ \cite{bebJ,kiTha} for categories associated to two-manifolds (elliptic
Drinfeld centers etc.). Finally, there are close links to factorization homology
\cite{cooke2}.
More recently, in the form of string-net models, skein theoretic constructions 
have turned out to provide a particularly direct description of modular functors
which form the basis for the construction of correlators of two-dimensional
conformal field theories \cite{scYa,traub,fusY}.

In the present paper we study string-net constructions that associate vector spaces
to oriented surfaces and categories to oriented one-manifolds. Traditionally,
following \cite{leWe}, in such constructions a spherical fusion category $\cc$ is 
taken as the algebraic input -- one considers finite embedded graphs whose edges and 
vertices are labeled by objects and morphisms of $\cc$, respectively.
Here we take instead a more general algebraic datum as input: a pivotal bicategory
$\cb$. This is a three-layered structure and, as a consequence, we now have labels
not only for edges and vertices of an embedded graph, but also for the connected 
components of the complement of the graph. Given the categorical and geometric 
dimensions involved, the idea to use a pivotal bicategory is indeed most natural,
involved, and it dates back at least to \Cite{App.\,C.2}{moWa}. What we achieve in
the present paper is a precise construction based on pivotal bicategories in the 
sense of e.g.\ \cite{fgjs}. In the assumptions about our input bicategories we are
parsimonious, the essential requirement being that they are linear over a field $\fk$;
in particular, for most of our considerations we do not need to impose semisimplicity.

\medskip

The basic idea of string nets is to globalize a graphical calculus that exists for
some standard canvas to general manifolds with the same tangential structure as the
standard canvas. Accordingly, we start in Section \ref{chap:GCalc} by developing a
graphical calculus for pivotal bicategories, with the oriented disk as a standard canvas.
First, in Section \ref{sec:sdbicat}, we summarize the ordinary string diagram calculus
for general bicategories, for which the canvas is the standard square in the complex
plane (regarded as 2-framed, with vector fields parallel to the two coordinate axes)
and the coupons for 2-morphisms are rectangles. If the bicategory $\cb$ is endowed
with the additional structure of a (strictly) pivotal bicategory, a different
graphical calculus can be derived from the ordinary string diagram calculus. This 
calculus has the standard oriented disk as its canvas and treats in- and outputs 
of a 2-morphism on the same footing. Thereby the coupons can be chosen to be circles;
as we explain in Section \ref{sec:pccp}, the space of colors associated with such 
a circular coupon can be defined as a limit over a contractible groupoid.

To be suited for string-net models, we need a concise formalization of this graphical
calculus. This is provided in Sections \ref{sec:pccp} and \ref{sec:ufgcpb}
in terms of a symmetric monoidal functor $\calb$ from a category $\crlla$ of partially
$\cb$-colored corollas to vector spaces, as described in Summary \ref{sum:calb}.
In Section \ref{sec:prstCrll} we prove structure theorems that are instrumental
later on: we show that (non-trivial) morphisms in the category $\crlla$ can be
decomposed into a finite disjoint union of partial compositions of morphisms of a 
few specific types (Proposition \ref{prop:CoroDecom}) and establish a similar 
presentation of the subcategory $\conncrlla$ of $\crlla$ that only contains those 
morphisms of $\crlla$ all of whose underlying graphs are connected in the disks 
they are embedded in (Corollary \ref{cor:ConnCoroDecom}).
 
These presentations of $\crlla$ and $\conncrlla$ allow us to investigate the 
functoriality of the graphical calculus encoded in the functor $\calb$ in its 
dependence on the bicategory $\cb$; this is developed in Section \ref{sec:RSF}. It is
worth stressing that considering pseudofunctors between bicategories $\cb$ and 
$\cb'$ does not provide sufficient insight; instead, we examine Frobenius monoidal 
functors, adapted to bicategories. Such functors first arose in the context of 
linearly distributive categories. These are monoidal categories with two tensor
products, which makes it natural to consider functors with monoidal constraints 
that are lax for one of the tensor products and oplax for the other. The lax and 
oplax structures are not required to be inverse to each other (which in the linearly
distributive context would not make sense). Instead, they have to satisfy a 
compatibility condition which implies in particular that the image of the terminal 
bicategory with one object, one 1-morphism and one 2-morphism under a Frobenius functor
gives a 1-endomorphism with the structure of a Frobenius algebra.
In the present paper, we are interested in general bicategories, rather than in the
special case of monoidal categories; the notion of a Frobenius functor can be easily
extended to these. We thus introduce the notion of conjugation of a
2-morphism by a functor equipped with lax and oplax structures
(Definition \ref{def:Fconj}). We show that conjugation by a
rigid Frobenius functor between two strictly pivotal bicategories $\cb$ and $\cb'$
canonically induces a monoidal natural transformation between the graphical calculi
for $\cb$ and $\cb'$ when restricted to $\conncrlla$ and $\pconncrlla$, respectively
(Theorem \ref{thm:RSFconj}), and that this canonically extends to a
monoidal natural transformation between the full graphical calculi if
$F$ is even a rigid pseudofunctor (Corollary \ref{cor:RPconj}).

The two types of bicategories we are particularly interested in are both based on a
pivotal category $\cc$: its delooping $\cbc$, and the bicategory $\cfrc$ which has
simple special symmetric Frobenius algebras internal to $\cc$ as objects, bimodules
over a pair of such algebras as 1-morphisms, and bimodule morphisms as 2-morphisms.
The obvious forgetful functor $\cu$ from $\cfrc$ to $\cbc$ is canonically a rigid 
separable Frobenius functor. In applications a crucial fact, explained in Example 
\ref{exa:UFrCtoBC}, is that the lax and oplax structures of $\cu$ provide an 
idempotent which exhibits the tensor product over a
special symmetric Frobenius algebra as a retract.

\medskip

In Section \ref{sec:SNbicat} we globalize the graphical calculus and construct
modular functors through a string-net construction that uses a strictly pivotal 
bicategory $\cb$ as an input. In Section \ref{sec:BSNspaces} we define the (bare)
string-net space $\osnb(\surf,\msb)$ assigned to a surface $\surf$ and a 
$\cb$-boundary datum $\msb$ on $\surf$. Similarly as in the conventional string-net 
construction, this is based on the idea to utilize the graphical calculus on disks
to impose local relations on a vector space freely generated by all $\cb$-colored
graphs on $\surf$ whose boundary datum is given by $\msb$. As usual, the space 
$\osnb(\surf,\msb)$ carries a geometric action of the mapping class group of $\surf$.
In Section \ref{sec:SN-colimit} we characterize $\osnb(\surf,\msb)$ as a colimit 
over a functor $\ce_{\cb}^{\surf,\msb}$ from a category of partially $\cb$-colored 
graphs to vector spaces (Theorem \ref{thm:SN-colim}). This insight lends itself
to future generalizations of the string-net construction, e.g.\ replacing linear
categories by dg-categories and colimits by homotopy colimits.
 
In Section \ref{sec:functoriality} we show that the string-net construction is 
functorial with respect to rigid pseudofunctors. Recall that, by Corollary 
\ref{cor:RPconj}, this class of functors preserves the entire graphical calculus. 
Concretely, we show that for any surface $\surf$ and $\cb$-boundary datum $\msb$ on
$\surf$, a rigid pseudofunctor $F$ between strictly pivotal bicategories $\cb$ and
$\cb'$ gives rise to a canonical $\mcg(\surf)$-intertwiner between the string-net
spaces $\osnb(\surf,\msb)$ and $\posnb(\surf,\msb')$, where the $\cb'$-boundary datum
$\msb'$ is obtained from $\msb$ by a change of coloring induced by $F$ (Theorem 
\ref{thm:osnb-posnb}).
Having now at our disposal string-net spaces on non-trivial manifolds, we can
introduce, in Sections \ref{sec:oCylS} and \ref{sec:oCylI}, the cylinder categories 
over circles and intervals, respectively. For the latter we impose the additional 
requirement that the strictly pivotal bicategory $\cb$ must be pointed, i.e.\ be 
endowed with a distinguished object.  This allows us to find the cylinder categories 
for intervals as the endomorphism category of the distinguished object 
(Proposition \ref{prop:ocylIEnd}), and to obtain the central result of Section 
\ref{sec:func-embed}, which states that the assignment of cylinder categories 
is functorial under embeddings of 1-manifolds (Proposition \ref{prop:cyl-funct}).

As a standard feature of skein theoretic constructions, the cylinder categories
have, in general, no reason to be idempotent complete or even abelian; we examine 
their idempotent completion in Section \ref{sec:kar}. The general bicategorical
string-net construction shares characteristic features of the special case based
on a spherical fusion category. In particular, as we show in Section \ref{sec:fact}
(Theorem \ref{thm:osnfct}), it obeys factorization, or excision, which we formulate
in terms of coends. Finally, in Section \ref{sec:open-closed} we can combine all 
these results to show, in Theorem \ref{thm:thm} that the string-net construction 
based on a pointed strictly pivotal bicategory $\cb$ provides an open-closed modular
functor $\cosnb$, i.e.\ a symmetric monoidal pseudofunctor
  $$
  \cosnb\Colon \bbord\rarr~\bprof
  $$
from the symmetric monoidal bicategory of open-closed bordisms to the symmetric 
monoidal bicategory of $\fk$-linear profunctors. The same applies to the functor
$\csnb$ furnished by idempotent-completed string nets.

 \medskip

In the final Section \ref{sec:SNC} we provide an application of bicategorical string 
nets to correlators of two-dimensional conformal field theories. This application 
makes use of the fact that if the input bicategory $\cb$ is the delooping of a 
spherical fusion category $\cc$, then the open-closed modular functor $\csnb$ extends 
the standard Turaev-Viro functor. In \cite{fusY} a construction of correlators via 
$\cc$-colored string nets has been achieved. This construction can be paraphrased as 
follows: A topological world sheet $\ws$ with physical boundaries and topological 
defect lines for a rational CFT whose chiral data are encoded in a modular fusion
category $\cc$ is naturally an $\cfrc$-colored graph $\cws$ on its underlying 
surface $\surf$; via the canonical quotient map $\mathrm q(\surf,\msb) \colon 
\fk\mathsf G_{\cfrc}(\surf,\msb) \,{\twoheadrightarrow}\,\osnfrc(\surf,\msb)$ it thus
tautologically determines a vector $[\cws]$ in the string-net 
space based on the bicategory $\cfrc$. On the other hand, in \cite{fusY} we 
associated to any world sheet $\ws$ an explicit string net and a corresponding vector
  $$ 
  \corc(\ws) \,\in \snc(\sws,\df_{\partial\surf}(\msb_{\ws}))
  $$ 
and showed that this assignment gives a consistent system of correlators 
\Cite{Thm.\,3.28}{fusY}. The assignment of correlators to world sheets with fixed 
underlying surface $\surf$ and boundary datum $\msb$ is encoded in a linear map 
$\corc(\surf,\msb)\colon\fk\mathsf{G}_{\cfrc}(\surf,\msb) \Rarr~ \snc(\surf,
\df_{\partial\surf}(\msb))$. The main result of Section \ref{sec:SNC} is the proof of
Theorem \ref{thm:universal}, which states that this linear map factorizes
over the string-net spaces based on $\cfrc$: For every compact oriented surface 
$\surf$ and $\cfrc$-boundary datum $\msb$ there is a unique $\mcg(\surf)$-intertwiner 
$\ucorc(\surf,\msb)\colon \osnfrc(\surf,\msb) \,{\xrightarrow{~~}} %
        $\linebreak[0]$
\snc(\surf,\df_{\partial\surf}(\msb))$ such that
  $$
  \corc(\surf,\msb) = \ucorc(\surf,\msb)\cir \mathrm q(\surf,\msb) \,.
  $$
This establishes a new conceptual role of defect data: they constitute
the input for a string-net construction that
provides a systematic home for relations between different correlators.

This may be seen as a special case of the following general statement (Theorem 
\ref{thm:B2B'}): Every rigid separable Frobenius functor between strictly pivotal 
bicategories induces linear maps between the corresponding bicategorical string-net 
spaces that are compatible with the mapping class group actions and with sewing, 
albeit in a less straightforward way than rigid pseudofunctors do -- in addition 
to conjugating by the functor, full Frobenius graphs are added to compensate for 
the incomplete preservation of the graphical calculi.


\section{Graphical calculus for pivotal bicategories} \label{chap:GCalc}

\Subsection{Rectangular string diagrams for bicategories} \label{sec:sdbicat}

A bicategory $\cb$ can be described as a category weakly enriched
in the symmetric monoidal 2-category $\ccat$ of small categories,
functors and natural transformations, with the Cartesian product as monoidal product.
Thus in particular for any pair of objects $a,b\iN\cb$ there is a \emph{hom-category}
$\cb(a,b)$. The only general requirements that we impose on $\cb$ is that all these 
hom-categories are themselves enriched in the category $\vct$ of (not necessarily
finite-dimensional) $\fk$-vector spaces and linear maps,
and that the endomorphisms of the identity 1-morphism $\id_a$ of any object $a\iN\cb$
are isomorphic as a $\fk$-algebra to the ground field  $\fk$.
Here $\fk$ is an algebraically closed field, which we fix once and for all.

For horizontal composition we use \emph{diagrammatic order}. We denote the
horizontal compositions generally by ``$\hoco$'', but for brevity 
in some long formulas suppress this symbol and indicate the composition instead just 
by juxtaposition. The vertical composition of 2-morphisms is denoted by ``$\circ$''. 
Composite 2-morphisms can be expressed through \emph{pasting diagrams} or,
alternatively. through \emph{string diagrams} on the standard square $I\Times I$ 
as a canvas. The two descriptions are Poincar\'e dual to each other.
For instance, given objects $a,b,c\iN\cb$, 1-morphisms $f,f',f'' \,{\in}\,
\cb(a,b)$ and $g,g'\iN\cb(b,c)$, and 2-morphisms 
$\alpha\colon f\Rightarro f'$, $\beta\colon f'\Rightarro f''$ and 
$\gamma\colon g\Rightarro g'$, the pasting diagram 
for $(\beta\cir\alpha)\hocO\gamma\colon f\hocO g\Rightarro f''\hocO g'$ is
  \be
  \begin{tikzcd}
  a && b && c
  \arrow[""{name=0,anchor=center,inner sep=0},"{f''}",curve={height=-24pt}, from=1-1,to=1-3]
  \arrow[""{name=1,anchor=center,inner sep=0},"f"',curve={height=24pt}, from=1-1,to=1-3]
  \arrow[""{name=2,anchor=center,inner sep=0},"{f'}"{description, pos=0.3}, from=1-1,to=1-3]
  \arrow[""{name=3,anchor=center,inner sep=0},"g"',curve={height=24pt}, from=1-3,to=1-5]
  \arrow[""{name=4,anchor=center,inner sep=0},"{g'}",curve={height=-24pt}, from=1-3,to=1-5]
  \arrow["\alpha"', shorten <=3pt, shorten >=3pt, Rightarrow, from=1,to=2]
  \arrow["\gamma"', shorten <=6pt, shorten >=6pt, Rightarrow, from=3,to=4]
  \arrow["\beta"', shorten <=3pt, shorten >=3pt, Rightarrow, from=2,to=0]
  \end{tikzcd}
  \label{eq:pasting}
  \ee
We portray the vertices of a string diagram as rectangular coupons and shade the
two-dimensional regions with different colors that indicate the different objects. 
The string diagram corresponding to the pasting diagram \eqref{eq:pasting} is then 
  \be
  \input{SDB0.tikz}
  \ee
Or, to give a somewhat more complicated example, the diagrams
  \be
  \begin{tikzcd}[column sep=1.5em]
  && {b''}
  \\
  a && {b'} && c & d 
  \\ 
  && b \arrow["{f'}"', from=2-1, to=2-3]
  \arrow[""{name=0, anchor=center, inner sep=0}, "{f''}", out=60, in=200, from=2-1, to=1-3]
  \arrow["{g'}"', from=2-3, to=1-3] \arrow["g", from=2-3, to=3-3]
  \arrow[""{name=1, anchor=center, inner sep=0}, "f"', out=-60, in=160, from=2-1, to=3-3]
  \arrow[""{name=2, anchor=center, inner sep=0}, "h"', out=20, in=240, from=3-3, to=2-5]
  \arrow[""{name=3, anchor=center, inner sep=0}, "{h'}"', out=120, in=-20, from=2-5, to=1-3]
  \arrow["i"', from=2-5, to=2-6]
  \arrow["\alpha"', shorten <=5pt, shorten >=2pt, Rightarrow, from=1, to=2-3]
  \arrow["\beta"', shorten <=2pt, shorten >=5pt, Rightarrow, from=2-3, to=0]
  \arrow["\gamma", shorten <=5pt, shorten >=4pt, Rightarrow, from=2, to=3]
  \end{tikzcd}
  \qquad \text{and} \qquad\quad
\input{SDB1.tikz}
  \ee
express the same 2-morphism in $\cb$.

Actually, to make sense of either type of diagram, one first needs to select for 
each layer of horizontal composite of 1-morphisms a bracketing, which includes 
a choice of insertions of identity 1-morphisms. However, thanks to the 
\emph{coherence theorem} for bicategories (see e.g.\ \cite[Sect.\,3.6]{JOYa}), 
between each pair of bracketed horizontal composites of the same composable sequence 
of 1-morphisms there is a unique 2-isomorphism made up of associators
and unitors that connects the pair. As a consequence, 
for any choice of bracketings for the \emph{source} and the \emph{target}
of a (string or pasting) diagram, there is a unique 2-morphism assigned
to the diagram, and the 2-morphisms for any two such choices are connected by a 
unique isomorphism of 2-hom spaces. Thus a pasting or string diagram uniquely 
determines a \emph{contractible groupoid} in $\vct$ whose vertices are the 
2-hom spaces corresponding to the possible bracketings, as well as a coherent 
choice of elements in each of the 2-hom spaces,
We refer to this coherent choice as the \emph{value} of the diagram. 

Put differently, every equality of string (or pasting) diagrams
stands for an infinite family of equalities in
different 2-hom spaces, one for each simultaneous choice of bracketings for both
sides of the equality. Alternatively, one can invoke the \emph{strictification
theorem} for bicategories (see e.g.\ \cite[Ch.\,2]{GUrs}), according to
which every bicategory is canonically biequivalent to a canonical strict 
2-category associated with it, and treat any bicategory as if it were strict. We 
will freely use both of these perspectives.

The value of a string diagram
is unaffected by any isotopy of the diagram that fixes the orientation
of the rectangular coupons while keeping the diagram \emph{progressive}
and the end points of its legs fixed. To allow also for non-progressive
string diagrams, one needs appropriate dualities.
A \emph{dual pair}, or \emph{adjoint pair}, in a bicategory $\cb$ is a quadruple 
$(f,g,\eta,\varepsilon)$ consisting of 1-morphisms $f\iN\cb(a,b)$ and $g\iN\cb(b,a)$
and 2-morphisms $\eta\colon\id_{a}\Rightarro f\hocO g$ and 
$\varepsilon\colon g\hocO f\Rightarro\id_{b}$, called the unit and counit of the
dual pair, that satisfy two \emph{yanking equalities}. When the unit and counit
are represented by the string diagrams
  \be
  \scalebox{0.8}{\input{SDB2.tikz}} \qquad\text{and}\qquad
  \scalebox{0.8}{\input{SDB3.tikz}} 
  \ee
the yanking equalities read
(after making the identity 1-morphisms invisible)
  \be
  \begin{aligned}
  & \scalebox{0.7}{\input{SDB4.tikz}} & \,=~ \scalebox{0.7}{\input{SDB5.tikz}}
  \\
  \text{and}\qquad 
  & \scalebox{0.7}{\input{SDB6.tikz}} & \,=~ \scalebox{0.7}{\input{SDB7.tikz}}
  \end{aligned}
  \ee
Duals are unique up to unique isomorphisms.
We call $f$ the \emph{left dual} (or \emph{left adjoint}) of
$g$ and write $f \eq \Vee g$, while $g$ is called the \emph{right
dual} (or \emph{right adjoint}) of $f$, written as $g \eq f^{\vee}$. 

A \emph{bicategory with duals }is a bicategory $\cb$ such that every
1-morphism in $\cb$ has both a left and a right dual. Fixing a right dual
for each 1-morphism yields a pseudofunctor 
  \be
 (-)^{\vee}\Colon \cb\rarr~ \cb^{\coop}
  \ee
from $\cb$ to the bicategory $\cb^{\coop}$ that obtained by reversing
both the 1- and the 2-morphisms in $\cb$. This pseudofunctor is
the identity on objects and sends every 1-morphism to its chosen right
dual and every 2-morphism $\alpha\colon f\Rightarro g$ to its \emph{transpose}
$\alpha^{\vee}\colon g^{\vee}\Rightarro f^{\vee}$. 
A \emph{pivotal} bicategory is one for which the double dual can be trivialized
\Cite{Sect.\,5.3}{fgjs}:

\begin{defn}
 \Itemize
 \item[{\rm (i)}]
A \emph{pivotal structure} on a bicategory $\cb$ with fixed left and right duals
is an identity component pseudonatural transformation (i.e.\ every component 
1-morphism is an identity) $\id_{\cb} \,{\Rightarrow}\, (-)^{\vee\vee}$.
Equipped with a pivotal structure, $\cb$ is called a \emph{pivotal bicategory}.
 \item[{\rm (ii)}]
A \emph{strictly pivotal bicategory} is a pivotal bicategory for which
the double dual is the identity,
  \be
  \id_{\cb} = (-)^{\vee\vee} .
  \ee
\end{itemize}
\end{defn}

In a strictly pivotal bicategory $\cb$ we have $\Vee f \eq (\Vee f)^{\vee\vee}
\eq f^{\vee}$, so that we can speak of \emph{the dual} of a 1-morphism. In a 
string diagram we can then replace a string labeled by the dual of $f$ by a string 
labeled by $f$ but having opposite direction. As a consequence, for strictly 
pivotal bicategories, non-progressive diagrams make sense. Moreover, for any 2-morphism
$\alpha\colon f\Rightarro g$ in a strictly pivotal bicategory one has
 \be
 \scalebox{0.8}{\input{SDB8.tikz}} ~~=~~ \scalebox{0.8}{\input{SDB9.tikz}}.
 \label{eq:rotatefc}
 \ee
It follows that an isotopy resulting in a $2\pi$-rotation of any coupon in a 
string diagram for a strictly pivotal bicategory does not affect the value 
of the diagram. Note that the canvas of the string diagrams -- the standard 
square $I\Times I$ -- comes with a canonical \emph{$2$-framing}, i.e.\ 
a trivialization of its tangent bundle $T$ given by two non-vanishing vector fields 
parallel to the $x$- and $y$-axis of $\mathbb R^2 \eq T(I\Times I)$, respectively.
In both diagrams in \eqref{eq:rotatefc} the coupon is aligned with this 
canonical 2-framing, so that in particular its boundary segments are 
parallel to those of $I\Times I$. We summarize these observations as

\begin{prop} \label{prop:sdpivb}
Any $($not necessarily progressive$)$ string diagram on the standard square
$($with coupons not necessarily aligned with the frame$)$ for a strictly pivotal 
bicategory $\cb$ has a well defined value.
Any isotopy of the diagram that keeps the end points of its legs fixed
$($but may rotate the coupons$)$ leaves this value is unchanged.
\end{prop}

\begin{proof}
Any string diagram $\Gama$ on $I\Times I$ is isotopic, through an isotopy that 
fixes the end points of its legs, to a string diagram $\widetilde{\Gama}$ whose
coupons are aligned with the frame. Moreover, any other choice $\widetilde{\Gama}'$
that yields such a diagram differs from $\widetilde{\Gama}$ by an isotopy that fixes
the end points and rotates each coupon by some multiple of $2\pi$, which because 
of the equality \eqref{eq:rotatefc} does not change the value of the diagram.
\end{proof}

There are two types -- $\cbc$ and $\cfrc$ --
of pivotal bicategories that are of particular interest to us:

\begin{example}\label{exa:bc}
Given a pivotal tensor category $\cc$, its \emph{delooping} $\cbc$, i.e.\ $\cc$ viewed
as a bicategory with a single object, is a pivotal bicategory, with the pivotal
structure of $\cc$ viewed as a pivotal structure for $\cbc$. Upon strictifying
the pivotal structure of $\cc$ (which is always possible \cite[Thm.\,2.2]{ngsc}),
the bicategory $\cbc$ becomes strictly pivotal.
\end{example}

\begin{example}\label{exa:cfrc}
For any pivotal tensor category $\cc$ there is a bicategory $\cfrc$ which has
simple special symmetric Frobenius algebras internal to $\cc$ as objects, bimodules 
over a pair of such algebras as 1-morphisms, and bimodule morphisms as 2-morphisms. 
The bicategory inherits from $\cc$ a canonical pivotal structure. If $\cc$ is 
strictly pivotal, then so is $\cfrc$.
That we require the symmetric Frobenius algebras which are objects of this bicategory
to be special is motivated by the application to conformal field theory that we will
describe in Section \ref{sec:SNC}. An essential ingredient in that application is the
idempotent that realizes (see Example \ref{exa:UFrCtoBC}) the tensor product over the
algebra.
\end{example}


\Subsection{Partially colored corollas and polarizations\label{sec:pccp}}

The string diagrams for bicategories described in Section \ref{sec:sdbicat} use
the standard square $I\Times I$ as the canvas. Asserting that a coupon
is \emph{aligned with the frame} is thus the same as saying that the coupon is 
aligned with the canonical 2-framing of the standard square. For a pivotal category,
by Proposition \ref{prop:sdpivb}, the 2-framing in the \emph{interior} of the 
square does not affect the evaluation of any string diagram for a (strictly)
pivotal bicategory: the value is unchanged under isotopies which do
not necessarily preserve the alignment of the coupons with the 2-framing.
In contrast, so far we still use the 2-framing at the \emph{boundary} of the square:
it tells us which part of the boundary is the bottom and which part is the top, 
namely the intervals in $\partial(I{\times}I)$ at which the framing is pointing
inwards and outwards, respectively. The distinction between top and the bottom 
separates the \emph{output ports} from the \emph{input ports} of a string diagram. As
we will show in this section, in the case of a \emph{pivotal} bicategory, the 
distinction between input and output is immaterial to the graphical calculus, so that
the 2-framing of the canvas can be completely disregarded. This result constitutes
a crucial step towards the formulation of string-net models -- in a sense, 
string-net models are generalizations of the graphical calculus for which 
string diagrams can have \emph{any} compact oriented surface as their canvas.

We now fix a strictly pivotal bicategory $\cb$. To proceed, let us introduce
the notion of a \emph{partially $\cb$-colored graph} on an oriented 
surface. For the moment we restrict our attention to the case that the surface is the
\emph{standard disk}. By definition, the standard disk $D \,{\subset}\, \complex$ is
the closed unit disk of radius 1 centered at $0\iN\complex$; the boundary 
$S^{1} \,{:=}\, \partial D \eq \{ z \iN \complex \,|\, |z| \eq 1\}$
of the standard disk is called the \emph{standard circle}.
More general surfaces than $D$ will be considered later on.

\begin{defn}
A \emph{partially $\cb$-colored graph} $\grph$ on a compact oriented
surface $\surf$ with possibly non-empty boundary consists of the following data:
 \Itemize
 \item[{\rm (i)}]
An underlying \emph{directed finite} graph, i.e.\ a diagram 
  \be
  \begin{tikzcd}
  E(\grph) \arrow[r, "\delta"', bend right]
  & H(\grph) \arrow["i"', loop, distance=2em, in=125, out=55] \arrow[l, two heads]
  & I(\grph) \arrow[l, "s"', hook'] \arrow[r, "t"] & V(\grph)
  \end{tikzcd}
  \ee
of finite sets, with $V(\grph)$, $H(\grph)$ and $I(\grph)$ the sets of 
\emph{internal vertices}, of \emph{half-edges}, and of half-edges that 
\emph{touch an internal vertex}, respectively. The map $t$ indicates the incidence of 
half-edges in $I(\grph)$ to the internal vertices, while $s$ is the canonical 
inclusion. The map $i$ is the
fixed-point-free involution that indicates the \emph{juncture} of pairs of half-edges;
its set of orbits is $E(\grph)$, to be interpreted as the set of \emph{edges}. An 
edge which consists of a pair of internal half-edges is called an \emph{internal edge},
while an edge consisting of a single half-edge is called a
\emph{leg}. The map $\delta$ is a section of the canonical quotient map 
$H(\grph)\,{\twoheadrightarrow}\, E(\grph)$; for each edge it picks out its 
\emph{starting} half-edge and thereby directs the edges.
 \item[{\rm (ii)}]
An \emph{embedding} into $\surf$ of the \emph{geometric realization} $|\grph|$ of the 
underlying graph, i.e.\ of the topological space 
  \be
  \big( \bigsqcup_{v\in V(\grph)} \! \{v\} \big) \,\sqcup\,
  \big( \bigsqcup_{e\in H(\grph)} \! [0,1]_e \big) \big/\!\sim \;,
  \ee
where the equivalence relation is generated by
$[0,1]_{e} \,{\ni}\, 1 \,{\sim}\, 1\iN [0,1]_{i(e)}$ for all $e\iN H(\grph)$ and
$[0,1]_{e} \,{\ni}\, 0 \,{\sim}\, \{t(e)\}$ for all $e\iN I(\grph)$.
The embedding is subject to the requirement that the intersection of $\partial\surf$ 
with the image of $|\grph|$ is exactly the image of the \emph{end points} of the legs,
i.e.\ the image of 
$\{0\iN[0,1]_{l} \,{\hookrightarrow}\, |\grph|\}_{l\in H(\grph)\setminus I(\grph)}^{}$.
 \item[{\rm (iii)}]
A coloring of the \emph{patches}, i.e.\ of the connected components of
$\surf\,{\setminus}\,|\grph|$, with objects of $\cb$, as well as a coloring of the
directed edges with 1-morphisms of $\cb$. More specifically, a directed edge $e$ 
is labeled with a 1-morphism $f \iN B(a_\mathrm{l},a_\mathrm{r})$, where $a_\mathrm{l}$
and $a_\mathrm{r}$ are the labels of the patches adjacent to the left and to the right
of $e$, respectively.
\end{itemize}
\end{defn}

Note that the vertices of the graph $\grph$ are \emph{not} colored -- hence the name
\emph{partially} $\cb$-colored graph -- and that the empty graph $\grph \eq \emptyset$
is allowed. Also, 
the distinction between left and right in part (iii) of the definition is meant to be
such that a vector pointing to the right and a vector pointing in the direction of $e$
form a coordinate system whose orientation coincides with the orientation of $\surf$.

\begin{defn}
A \emph{partially $\cb$-colored corolla} is a partially $\cb$-colored graph 
on the standard disk $D$ whose underlying directed finite 
graph has a single vertex -- called the \emph{center} of the corolla -- which
is mapped by the embedding to $0\iN D \,{\subset}\, \complex$, and for which 
the image of each edge is a straight line connecting the center with a point on
$\partial D$. (The number of edges is allowed to be zero.)  
\end{defn}

\begin{example}\label{exa:exaK}
The following picture shows a partially $\cb$-colored corolla
$K$ on the standard disk (which we equip  with the counterclockwise orientation):
  \be
  K ~=~ \input{GCPB1.tikz}
  \label{eq:K}
  \ee
Here $a,b,c$ are objects in $\cb$ and $f\iN\cb(a,b)$, $g\iN\cb(b,c)$ and  $h\iN\cb(c,a)$
are 1-morphisms. We also indicate the objects by shading the regions with different colors,\,%
 \footnote{~In the color version: green, blue, and purple.}
In the pictures below, we will only keep the shadings and suppress the 
corresponding object labels.
\end{example}

\begin{defn} \label{def:bdydatum}
For $\cb$ a strictly pivotal bicategory,
a $\cb$-\emph{boundary datum} $\msb$ on a compact oriented 1-manifold $\ell$
with possibly non-empty boundary consists of of the following data:
 \Itemize
 \item[{\rm (i)}]
A (possibly empty) finite set $O_\ell$ of points in the interior of  $\ell$.
 \item[{\rm (ii)}]
A coloring of each connected component $a \iN \pi_0(\ell \,{\setminus}\, O_\ell)$ 
with an object $\mathsf C_\msb(a) \iN \cb$.
 \item[{\rm (iii)}]
A coloring of each element of $O_\ell$ with a 1-morphism in $\cb$.
\end{itemize}
Here the convention is that the color of $p \iN O_\ell$ is a 1-morphism in
$\cb(\mathsf C_\msb(a),\mathsf C_\msb(a'))$ if the connected components $a$ and $a'$
of $\ell \,{\setminus}\, O_\ell$ are located to the left and right of $p$,
respectively, with the orientation of $\ell$ pointing from the right to the left. 
\end{defn}

In view of the lack of coloring for the center of a corolla, there is a canonical
bijection between the set of partially $\cb$-colored corollas and
the set of $\cb$-boundary data on $S^{1}$, as illustrated in the following picture:
  \be
  \input{GCPB00.tikz} \quad \longleftrightarrow \quad \input{GCPB0.tikz}
  \label{eq:kb}
  \ee

We are now going to associate to the center $v$ of a partially $\cb$-colored
corolla $K$ a vector space $H_{v}^{\cb}$, to be called the \emph{space of colors}
for the vertex $v$. We specify this space as a certain 2-hom space in $\cb$. To
define this space, we make use of the auxiliary datum of a linear order on the 
set $H(v)$ of half-edges incident to the vertex $v$.
The orientation of $D$ naturally induces a cyclic order on $H(v)$; with the chosen
counterclockwise orientation of $D$, this cyclic order is clockwise. A linear order
on $H(v)$ that is compatible with the induced cyclic order is uniquely determined by
the choice of a \emph{root}, or \emph{starting half-edge}, $e\iN H(v)$. 
The choice of $e$ determines a particular 2-hom space $h_{v}^{\cb}(e)$ in $\cb$ 
as a space of invariants. For instance, if for the corolla $K$
shown in Example \ref{exa:exaK} we choose the root $e$ to be the half-edge 
$e_h$ that is labeled by $h\iN\cb(c,a)$, then we associate to $v$ the 2-hom space
  \be
  h_{v}^{\cb}(e_h) \coloneqq \mathrm{End}_{\cb}(c)(\id_{c},h\hocO f\hocO g) \,.
  \ee
Similarly, for the choices $e \eq e_f$ and $e \eq e_g$ we obtain the spaces 
$h_{v}^{\cb}(e_f) \Coloneqq \mathrm{End}_{\cb}(a)(\id_{a},f\hocO g\hocO h)$
and $h_{v}^{\cb}(e_g) \Coloneqq \mathrm{End}_{\cb}(b)(\id_{b},g\hocO h\hocO f)$,
respectively.  

By using the units and counits of dual pairs to turn edges around, we can construct
canonically an isomorphism between any two of the so assigned 2-hom spaces. 
For instance, in the case of the corolla $K$ from Example \ref{exa:exaK}, changing
the root from $h$ to $f$ and then to $g$ gives rise to the chain 
  \be
  h_{v}^{\cb}(e_h) \stackrel{\cong}{\longrightarrow} h_{v}^{\cb}(e_f)
  \stackrel{\cong}{\longrightarrow} h_{v}^{\cb}(e_g)
  \ee
of isomorphisms of vector spaces, whose action on an element 
$\varphi\iN h_{v}^{\cb}(e_h) \eq \mathrm{End}_{\cb}(c)(\id_{c},h\hoc f\hoc g)$
is given by
  \be
  \input{GCPB2.tikz} ~\longmapsto~ \input{GCPB3.tikz} ~\longmapsto~ \input{GCPB4.tikz}
  \ee
 ~\\[-0.5em]
Further, owing to the yanking equations for duals together with \eqref{eq:rotatefc},
after one further step we get back the original element $\varphi$:
  \be
  \input{GCPB5.tikz} ~~=~~ \input{GCPB2.tikz}
  \ee
Altogether this yields a groupoid in $\vct$ that is generated by the diagram
  \be
  \begin{tikzcd}[column sep=1.0em]
  & {h^{\mathcal{B}}_{v}(e_f)}
  \\[5pt]
  {h^{\mathcal{B}}_{v}(e_h)} && {h^{\mathcal{B}}_{v}(e_g)}
  \arrow["\cong", sloped, tail reversed, from=2-1, to=1-2]
  \arrow["\cong", sloped, tail reversed, from=1-2, to=2-3]
  \arrow["\cong"', sloped, tail reversed, from=2-3, to=2-1]
  \end{tikzcd}
  \ee
This groupoid is contractible, i.e.\ between any ordered pair of objects there is a
single morphism, which in particular implies that any two composites of morphisms 
with the same source
and same target are equal. Thus by selecting an element in \emph{any} of the three
2-hom spaces, we simultaneously select an element in \emph{each} of the spaces.

This leads us to the following prescription: Let $K$ be a partially $\cb$-colored 
corolla with center $v$ and $n \eq |H(v)| \,{>}\, 0$. Then to $K$ we associate:
\Enumerate
 \item 
A contractible groupoid $\cg_{v}^{\cb}$. The objects of $\cg_{v}^{\cb}$ are the
elements of the set $H(v) \eq \{e_{1},e_{2},...\,,e_{n}\}$ of half-edges incident 
to $v$, which is conveniently indexed according to an arbitrary choice of compatible
linear order. The morphisms of $\cg_{v}^{\cb}$ are generated by the diagram
  \be
  \begin{tikzcd}
  {e_1} & {e_2} & {e_3} \\ {e_n} & \dotsb & {e_4}
  \arrow[tail reversed, from=1-1, to=1-2] \arrow[tail reversed, from=1-2, to=1-3]
  \arrow[tail reversed, from=1-3, to=2-3] \arrow[tail reversed, from=2-3, to=2-2]
  \arrow[tail reversed, from=2-2, to=2-1] \arrow[tail reversed, from=2-1, to=1-1]
  \end{tikzcd}
  \ee
with relations uniquely determined by requiring that $\cg_{v}^{\cb}$ is contractible,
i.e.\ that each hom-set $\cg_{v}^{\cb}(e_{i},e_{j})$ has a single element.
 \item
A functor $h_{v}^{\cb}\colon \cg_{v}^{\cb} \To \vct$ that acts on objects by 
  \be
  e_{i} \,\longmapsto\, h_{v}^{\cb}(e_{i}) \coloneqq \mathrm{End}_{\cb}(a_{i})
  (\id_{a_{i}},f_{i}^{\epsilon_{i}}\hocO f_{i+1}^{\epsilon_{i+1}} {\hocO} \dotsb \hocO
  f_{i+n-1}^{\epsilon_{i+n-1}}) \,,
  \ee
where the indices are taken mod $n$, $f_{j}$ is the color of
the half-edge $e_{j}$ with $f_{j}^{\epsilon_{j}}\iN\cb(a_{j},a_{j+1})$, and where
$\epsilon_{j} \eq {+}$ if $e_{j}$ is directed away from $v$ and $\epsilon_{j} \eq {-}$
otherwise. The action on the generating morphisms is given by ``dragging the leg around''.
\end{enumerate}

\noindent
We can now assign a space of colors to a partially colored corolla:

\begin{defn}
The \emph{vector space of colors} $H_{v}^{\cb}$ for a partially colored corolla
with at least one leg is the limit 
  \be
  H_{v}^{\cb} \coloneqq \lim h_{v}^{\cb} \,\in\vct \,.
  \label{eq:HvB}
  \ee
For a corolla $K_{a}$ without legs and its single patch $D\,{\setminus}\, v_{a}$ colored
with an object $a\iN\cb$, the space of colors is 
  \be
  H_{v_{a}}^{\cb} \coloneqq \mathrm{End}_{\cb}(a)(\id_{a},\id_{a}) \,.
  \label{eq:HvaB}
  \ee
\end{defn}
 
Being the limit of a contractible groupoid, the space $H_{v}^{\cb}$ is determined
by an isomorphism $ p_{v}^{e}\colon H_{v}^{\cb} \Rarr\cong h_{v}^{\cb}(e)$
for any choice of $e\iN H(v)$, which uniquely extends to a limit
cone over $h_{v}^{\cb}$ with every leg being an isomorphism of vector
spaces. Therefore, by choosing a \emph{color }for the vertex $v$,
i.e.\ an element $\clc\iN H_{v}^{\cb}$, we produce a family
  \be
  \{p_{v}^{e}(\clc)\in h_{v}^{\cb}(e)\}_{e\in H(v)}
  \ee
of 2-morphisms in $\cb$ that is coherent in the sense that for any two members 
$p_{v}^{e_{i}}(\clc)$ and $p_{v}^{e_{j}}(\clc)$ of the family there is a unique 
isomorphism $ h_{v}^{\cb}(e_{i})\Rarr\cong h_{v}^{\cb}(e_{j})$. This isomorphism
is obtained by evaluating the functor $h_{v}^{\cb}$ on the unique morphism
$e_{i}\Rarr\cong e_{j}$ in the groupoid $\cg_{v}^{\cb}$,
such that $p^{e_i}_v(\clc)$ is mapped to $p^{e_j}_v(\clc)$.

\begin{rem}
In still more detail, the relation between the elements in the different spaces of
2-morphisms that are determined by the same color $\clc\iN H_{v}^{\cb}$ can be 
expressed as follows: By selecting a root $e\iN H(v)$ we produce, up to isotopies that
fix the boundary $\partial D \eq S^{1}$, a string diagram with a rectangular
coupon on $D$, which is isotopic relative to $\partial D$ to the string
diagram produced by choosing any other root. Let us illustrate this with 
the corolla \eqref{eq:K}. The choice $e \eq e_h$ results in the string diagram
  \be
  \input{GCPB9.tikz}
  \label{eq:GCPB9}
  \ee
If we choose instead $e \eq e_f$, then the diagram is
  \be
  \input{GCPB10.tikz} ~=~ \input{GCPB11.tikz}
  \label{eq:GCPB1011}
  \ee
which is isotopic to \eqref{eq:GCPB9}.
\end{rem}

So far we only considered 2-morphisms that have all non-trivial 1-morphisms in 
their outputs. This requirement is too restrictive for applications; to be able to
remove it, we introduce the following notion: 

\begin{defn} \label{def:polarization}
A \emph{polarization} on the vertex $v$ of a partially $\cb$-colored corolla
is a partition 
  \be
  H(v) = H^{\rin}(v)\sqcup H^{\rout}(v)
  \ee
of the cyclically ordered set of half-edges into two linearly ordered sets
$H^{\rin}(v)$ and $H^{\rout}(v)$ of \emph{input} and \emph{output} half-edges, in
such a way that any two half-edges of the same type (i.e., either in- or output) 
are consecutive with respect to the cyclic order on $H(v)$ if they are consecutive
with respect to the linear order on $H^{\rin}(v)$ or $H^{\rout}(v)$ that is
induced by the cyclic order on $H(v)$. 
\end{defn}

If $H^{\rin}(v) \eq \emptyset$ or $H^{\rout}(v) \eq \emptyset$,
a polarization reduces to a compatible linear order on $H(v)$.
Note that whether a half-edge of a corolla is an input or output half-edge with
respect to a chosen polarization is not correlated with the direction of the
edge to which it belongs. 

With this notion we obtain a contractible groupoid $\widehat{\cg_{v}^{\cb}}$ 
that is generated by the set of all polarizations on $v$. We have 
a canonical extension 
  \be
  \cg_{v}^{\cb} \lhook\joinrel\xrightarrow{~\simeq~}
  \widehat{\cg_{v}^{\cb}} \xrightarrow{~\simeq~} 1 
  \ee
of groupoids, as well as an extension 
  \be
  \begin{tikzcd}
  {\widehat{\mathcal{G}^{\mathcal{B}}_v}} && {\mathrm{Vect}_{\mathbbm{k}}}
  \\
  {\mathcal{G}^{\mathcal{B}}_v}
  \arrow["\simeq", sloped, hook, from=2-1, to=1-1]
  \arrow["{h^{\mathcal{B}}_v}"', from=2-1, to=1-3]
  \arrow["{\widehat{h^{\mathcal{B}}_v}}", dashed, from=1-1, to=1-3]
  \end{tikzcd}
  \ee
of the functor $h_{v}^{\cb}$, which is defined by setting
  \be
  \widehat{h_{v}^{\cb}}(k) \coloneqq 
  \hom_{\cb(a_{i},a_{j+1})}(f_{i+n-1}^{-\epsilon_{i+n-1}} \hocO {\dotsb} \hocO
  f_{j+2}^{-\epsilon_{j+2}}\hocO f_{j+1}^{-\epsilon_{j+1}}, f_{i}^{\epsilon_{i}}
  \hocO f_{i+1}^{\epsilon_{i+1}} \hocO {\dotsb} \hocO f_{j}^{\epsilon_{j}}) \,.
  \ee
for a polarization $k$ with $H^{\rout}(v) \eq \{e_{i},e_{i+1},...\,,c,e_{j}\}$ and 
$H^{\rin}(v) \eq \{e_{j+1},e_{j+2},...\,,c,e_{i+n-1}\}$ (with indices counted mod $n$).

In the graphical description of a polarization we draw ingoing edges as dashed 
lines. (In the color version, we use in addition two different colors for the
1-morphisms that label the in- and outgoing edges.)

\begin{example}
For the corolla \eqref{eq:K} described in Example \ref{exa:exaK}, we denote by 
  \be
  \input{GCPB6.tikz}
  \label{eq:GCPB6}
  \ee
the polarization $k$ that has $H^{\rin}(v) \eq \{e_h\}$ and 
$H^{\rout}(v) \eq \{e_f,e_g\}$. For this polarization we have 
$\widehat{h_{v}^{\cb}}(k) \eq \hom_{\cb(a,c)}(h^{\vee},f\hocO g)$.
\end{example}
  
It follows that the space $H_{v}^{\cb}$ is also equipped with a unique limit
cone over $\widehat{h_{v}^{\cb}}\colon\widehat{\cg_{v}^{\cb}}\To\vct$
that restricts to the limit cone over $h_{v}^{\cb}$. We denote the legs of this cone
(which are isomorphisms) by 
  \be
  \{p_{v}^{k}\colon H_{v}^{\cb} \Rarr\cong \widehat{h_{v}^{\cb}}(k)\}
  _{k\in\ob\,\widehat{\cg_{v}^{\cb}}} \,.
  \label{eq:pvk}
  \ee
Now from a color $\clc\iN H_{v}^{\cb}$ we obtain an entire family of 2-morphisms 
$\{p_{v}^{k}(\clc)\iN\widehat{h_{v}^{\cb}}(k)\}_{k\in\ob\widehat{\cg_{v}^{\cb}}}$
that are related by dualities. For instance, with the polarization $k$ selected in
\eqref{eq:GCPB6}, the 2-morphism $p_{v}^{k}(\clc)\iN\hom_{\cb(a,c)}(h^{\vee},f\hocO g)$
can be expressed in terms of 
$p_{v}^{e_h}(\clc)\iN\mathrm{End}_{\cb}(c)(\id_{c},h\hocO f\hocO g)$ as 
  \be
  p_{v}^{k}(\clc) ~=~~ \input{GCPB7.tikz}
  \ee
Together with a choice of polarization, a color $c\iN H_{v}^{\cb}$
again determines an isotopy class of string diagrams on $D$ with a rectangular
coupon, and the isotopy class is independent of the choice of polarization.


\Subsection{Graphical calculus on disks for pivotal bicategories} \label{sec:ufgcpb}

To represent the isotopy class of the string diagrams which a color 
$\clc\iN H_{v}^{\cb}$ on the vertex $v$ of a corolla
produces, without choosing a specific polarization, 
we place a \emph{circular coupon} labeled by $\clc$ at the vertex $v$. For instance, 
  \be
  \input{GCPB8.tikz}
  \label{eq:GCPB8}
  \ee
represents the isotopy class of string diagrams that contains both the diagram
\eqref{eq:GCPB9} and the diagram \eqref{eq:GCPB1011}. Conversely, a diagram
like \eqref{eq:GCPB8} can be represented by a string diagram on $D$ with rectangular 
coupon that corresponds to any choice of polarization and a 2-morphism in the 
associated 2-hom vector space. 

We allot to a partially $\cb$-colored corolla $K$ the space
of colors for its center $v$. We denote this assignment by\,%
  \be
  \calb(K) \coloneqq H_{v}^{\cb} .
  \label{eq:calbk}
  \ee
(The notation $\calb$ is supposed to remind of the `\emph graphical \emph{cal}culus'.)
The assignment \eqref{eq:calbk} is functorial with respect to orientation preserving
embeddings of the standard disk to itself that induce a local
isomorphism of the partially colored embedded corollas, in such a way that the
colors of the patches match while a half-edge may be be mapped to either a 
half-edge with the same orientation and the same color or a half-edge with 
the opposite orientation and the dual color. For instance, the picture
  \be
  K ~=~~ \input{GCPB12.tikz} ~~\hookrightarrow~~ \input{GCPB13.tikz} ~~=~ K'
  \label{eq:KembedK'}
  \ee
shows an embedding of the corolla $K$ into $K'$, with the image of $K$ indicated 
by the shaded region in $K'$, for which the half-edge colored by the 1-morphism $h$ 
is mapped to the one colored by the dual $h^{\vee}$. The action on
the assignment \eqref{eq:calbk} is in this case given by 
  \be
  \calb(K) = H_{v}^{\cb} \iso h_{v}^{\cb}(e_{h})
  = \mathrm{End}_{\cb}(c)(\id_{c},h \hoc f \hoc g)
  = h_{v'}^{\cb}(e_{h^{\vee}}) \iso H_{v'}^{\cb} = \calb(K') \,.
  \ee
Note that due to this functoriality, for a \emph{monochromatic} corolla
$K_{\mathrm{mnc}}$, i.e.\ a corolla each of whose half-edges is colored
with the same 1-morphism and oriented in the same way, the vector
space $\calb(K_{\mathrm{mnc}})$ carries an action of the cyclic group
of appropriate order.

 \medskip

Next we extend the assignment \eqref{eq:calbk} to a symmetric monoidal functor 
  \be
  \calb \Colon \crlla \rarr~ \vct
  \label{eq:def:calb}
  \ee
from a category of partially $\cb$-colored corollas and graphs to $\vct$.
We first introduce the source category of $\calb$. For $\cb$ a pivotal bicategory,
a \emph{boundary datum} of a partially $\cb$-colored corolla is
the $\cb$-boundary datum, in the sense of Definition \ref{def:bdydatum},
for the boundary circle of the underlying disk that is determined by the
colors of the patches of the disk. 

\begin{defn}    
Let $\cb$ be a pivotal bicategory.
$\crlla$ is the following symmetric monoidal category:
\Itemizej
 \item
Objects in $\crlla$ are finite disjoint unions of partially $\cb$-colored corollas. 
(This includes the empty disjoint union $\emptyset$.)
 \item
A morphism of type $K_{1}\,{\sqcup}\dotsb{\sqcup}\, K_{n} \To K_{n+1}$
in $\crlla$ is a partially $\cb$-colored graph $G$ on the standard disk whose
boundary datum coincides with that of $K_{n+1}$, together with an orientation
preserving embedding $D_1\,{\sqcup}\dotsb{\sqcup}\, D_n \,{\hookrightarrow}\, D_{n+1}$
of the underlying disks of the source (the number of which is allowed to be zero) to the 
underlying disk of the target. This embedding is required to induce a local isomorphism 
of the graphs that respects the colorings of the patches, while the orientation or$_e$ 
and color $f_e$ of a half-edge $e$ are required to be respected either on the nose or 
up to simultaneous reversal of orientation and dualizing of the color. Moreover, each 
internal vertex of $G$ must be covered by the image of exactly one 
of the disks $D_1,...\,,D_n$.
 \item
General morphisms in $\crlla$ are obtained by taking disjoint unions of morphisms of
the type just described. The composition $G_{2}\cir G_{1}$ 
of morphisms is given by \emph{blowing up} the internal vertices of $G_{2}$ 
by $G_{1}$, using the embeddings of disks. 
 \item
The symmetric monoidal product on $\crlla$ is given by disjoint union, 
with monoidal unit $\emptyset$.
 \end{itemize}
\end{defn}

\begin{example}
An example for a morphism in $\crlla$ of type
  \be
  G \Colon~ \input{CALB0.tikz} \,\sqcup~\, \input{CALB1.tikz} ~~~\rarr{~~}~~~ \input{CALB2.tikz}
  \ee
is determined by the following disk in which,
as we already did in the picture \eqref{eq:KembedK'} above, the images of the source
disks in the target are indicated by shaded regions:
  \be
  G ~=~~~ \input{CALB3.tikz}
  \label{eq:calb3}
  \ee
An illustration of the blowing-up procedure that defines the composition 
of morphisms in $\crlla$ is given in the following picture:
  \be
  \hspace*{-1.1em}
  \scalebox{0.9}{\input{CALB5.tikz}} ~\circ~ \scalebox{0.9}{\input{CALB4.tikz}}
  ~=~~~ \scalebox{0.9}{\input{CALB6.tikz}}
  \ee
In words, the shaded region within the first disk $D_1$ specifies the embedding 
of the second disk $D_2$ into $D_1$, and in particular the boundary datum for 
the shaded region in $D_1$ matches the boundary datum of the disk $D_2$.
\end{example}

 \medskip

We are now in a position to define the functor \eqref{eq:def:calb}. On objects,
$\calb$ acts as
  \be
  \calb(K_{1}{\sqcup}\dotsb{\sqcup} K_{n})
  \coloneqq \calb(K_{1}) \,{\otk}\dotsb{\otk}\, \calb(K_{n})
  = H_{v_{1}}^{\cb} \,{\otk}\dotsb{\otk}\, H_{v_{n}}^{\cb} \,.
  \label{eq:calbobj}
  \ee
As the case $n \eq 0$, this includes $\calb(\emptyset) \,{\coloneqq}\, \fk$.
To define the action of $\calb$ on morphisms, a more extensive description is 
needed. We demonstrate it for the case of a morphism
$G\colon K_{1}\,{\sqcup}\, K_{2} 
       $\linebreak[0]$
       {\rarr~}\, K_{3}$ of the type shown in
\eqref{eq:calb3}: To an arbitrarily chosen element $\clc_{1}\Otk \clc_{2}
\iN\calb(K_{1}) \,{\otk} 
       $\linebreak[0]$
       \calb(K_{2}) \eq H_{v_{1}}^{\cb}\Otk H_{v_{2}}^{\cb}$
we assign an element $\clc_3\iN\calb(K_{3}) \eq H_{v_{3}}^{\cb}$ and then obtain
a linear map by linear extension. To specify $\clc_3$, we first choose
polarizations for the vertices $v_{1}\iN K_{1}$ and $v_{2}\iN K_{2}$, say
  \be
  k_{1} ~\colon~~ \input{CALB7.tikz} \qquad\text{and}\qquad k_{2} ~\colon~~ \input{CALB8.tikz}
  \ee
with $k_1$ defined by selecting the edge labeled with $f_6$ (drawn as a dashed line,
and with label displayed in red in the color version) as its single input edge and
$k_2$ defined by selecting the edge labeled with $f_4$ as its single input edge.
Now for $i \eq 1,2$ set $\widetilde{\clc}_{i} \,{:=}\, p_{v_{i}}^{k_{i}}(\clc_{i}) \iN
\widehat{h_{v_{i}}^{\cb}}(k_{i})$, with $p_{v}^{k}$ given by \eqref{eq:pvk}. According
to the discussion in Section \ref{sec:pccp} we obtain, up to isotopies relative
to the boundary of $D$, two string diagrams with rectangular coupons on $D$ that
correspond to the choice of colors and polarizations. These string diagrams are then
pushed forward along the embedding $D_1 \,{\sqcup}\, D_2 \,{\hookrightarrow}\, D_3$,
replacing the images of the corollas $K_{1}$ and $K_{2}$. Thereby we obtain,
up to isotopies fixing the boundary, a string diagram $G'$
with rectangular coupons on $D$ that has the same boundary datum as $K_{3}$:
  \be
  G' ~=~~ \input{CALB10.tikz}
  \ee
We now choose a polarization for the vertex $v_{3}$
of the corolla that is the target of the morphism, e.g. 
  \be
  k_{3}~\colon~~ \input{CALB9.tikz}
  \ee
with the edge labeled with $f_4$ as single input edge.
With this choice we produce a string diagram $G''$ on the standard square $I\Times I$,
uniquely up to isotopies fixing the top and the bottom of the square setwise: 
  \be
  G'' ~=~~~~ \input{CALB11.tikz}
  \ee
(Recall that, by construction, no 1-morphisms are ending on the left and right boundary
segments.) As explained in Section \ref{sec:sdbicat}, such a rectified string diagram
uniquely determines a 2-morphism $\widetilde{\clc}_3\iN\widehat{h_{v_3}^{\cb}}(k_3)$.
This 2-morphism $\widetilde{\clc}_{3}$ determines, in turn, a unique element
$\clc_3 \eq (p_{v_3}^{k_3})^{-1}(\widetilde{\clc}_3) \, 
      $\linebreak[0]$
{\in}\, H_{v_3}^{\cb} {=}\, \calb(K_3)$ which, besides on the colors $\clc_1$ and 
$\clc_2$, only depends on the isotopy class (relative to the boundary) of $G$; in
particular, it does not depend on the auxiliary choice of polarizations. We write
  \be
  \clc_{3} ~=~~~\left\langle~ \input{CALB12.tikz} ~~\right\rangle
  \label{eq:calb12}
  \ee
and refer to $\clc_3$ as the \emph{value} of the string diagram with circular coupon. 
Altogether we thus obtain a linear map 
  \be
  \calb(G)\Colon \calb(K_{1} \,{\sqcup}\, K_{2}) = \calb(K_{1}) \Otk\calb(K_{2})
  \rarr~ \calb(K_{3}) \,.
  \label{eq:calb13}
  \ee
We declare this linear map to be what the functor $\calb$ maps the morphism 
$G\colon K_{1} \,{\sqcup}\, K_{2}\To K_{3}$ to. Moreover, we set
$\calb(G_{1}{\sqcup} G_{2}) \,{:=}\, \calb(G_{1})\Otk\calb(G_{2})$.
This concludes the definition of $\calb$.

\begin{summ} \label{sum:calb}
The symmetric monoidal functor $\calb \colon\crlla \Rarr~ \vct$
from the symmetric monoidal category $\crlla$ of partially $\cb$-colored corollas 
and graphs to $\vct$ maps objects as in \eqref{eq:calbobj} and morphisms as in
\eqref{eq:calb13}. The functor $\calb$ completely
captures the graphical calculus on disks for pivotal bicategories.
\end{summ}


In the analysis above, as local models for internal vertices of embedded graphs we
have only treated corollas on standard disks. But in fact we can make sense of the
evaluation of an arbitrary
string diagram with circular coupons on $D$ as in \eqref{eq:calb12} without 
reference to embedded disks. In more detail, denote by $V(G)$ the set of internal 
vertices of a partially $\cb$-colored graph $G$ on $D$. By choosing for each internal
vertex $v\iN V(G)$ an embedding of some partially colored corolla $K_{v}$, we can
turn $G$ into a morphism $\widehat{G}\colon\bigsqcup_{v\in V(G)}K_{v}\To K_{G}$ in 
the category $\crlla$, where $K_{G}$ is the unique corolla on $D$ that has the same
boundary datum as $G$. The embeddings of corollas canonically induces an isomorphism 
$\bigotimes_{v\in V(G)}H_{v}^{\cb} \,{\iso}\, \bigotimes_{v\in V(G)}\calb(K_{v})$.
We thus obtain a linear map 
  \be
  \hcalb(G)\colon \bigotimes_{v\in V(G)}\!\! H_{v}^{\cb} \rarr{~\cong~}
  \bigotimes_{v\in V(G)}\!\! \calb(K_{v})\stackrel{\calb(\widehat G)}{\vla}
  \calb(K_{G}) = H_{v_G^{}}^{\cb} \,,
  \label{eq:hcalb}
  \ee
with $v_G^{}$ the single vertex of the corolla $K_G$.
Owing to the existence of coherent isomorphisms between the vector spaces 
$\bigotimes_{v\in V(G)}\calb(K_{v})$ arising from different choices of embeddings 
of corollas, the so obtained map $\hcalb(G)$ does not depend on that choice.
Moreover, $\hcalb(G)$ is unaffected by any isotopy of $G$ that fixes its boundary. 
We can therefore give

\begin{defn} \label{def:value}
Let $G_{\clc}$ be a \emph{fully colored} graph, with the coloring of its internal
vertices given by $\clc\iN\bigotimes_{v\in V(G)}H_{v}^{\cb}$. The \emph{value}
$\langle G_{\clc}\rangle\iN H_{v_{G}}^{\cb}$ of $G_{\clc}$ is defined to be the element
  \be
  \langle G_{\clc}\rangle \coloneqq \hcalb(G)(\clc) 
  \label{eq:evalG}
  \ee
of the vector space $H_{v_{G}}^{\cb}$ assigned to the vertex of the corolla that
corresponds to the boundary datum of $G_{\clc}$.
\end{defn}

We think of the value $\langle G\rangle$ of a fully colored graph $G$ as the color for
the vertex $v_G\iN K_G$ that is obtained by replacing $G$ with the corolla $K_{G}$. 

\begin{example}
The following picture shows a graph $G$ and the corresponding corolla $K_G$:
  \be
  G ~=~~ \scalebox{0.7}{\input{CALB13.tikz}} \quad\longmapsto\quad
  \scalebox{0.7}{\input{CALB14.tikz}} ~~=~ K_{G} \,.
  \ee
In this case we have 
  \be
  \begin{aligned}
  \hcalb(G)\Colon H_{v_{1}}^{\cb}\Otk H_{v_{2}}^{\cb} & \rarr~ H_{v_{G}}^{\cb} \,.
  \\[3pt]
  \clc = \clc_{1}\Otk \clc_{2} & \xmapsto{~~~} \langle G_{\clc}\rangle \,,
  \end{aligned}
  \ee
and the value $\langle G_{\clc}\rangle\iN H_{v_{G}}^{\cb} \eq \calb(K_{G})$ is the same
as the one in \eqref{eq:calb12}, except that now the color
$\clc \eq \clc_{1}\Otk \clc_{2}$ is to be understood as living in the vector space
$H_{v_{1}}^{\cb}\Otk H_{v_{2}}^{\cb}$ associated to the internal vertices of $G$.
\end{example}


\Subsection{Presentations of two categories of corollas} \label{sec:prstCrll}

We can think of the functor $\calb\colon\crlla\To\vct$ as a rule for
assigning a space of morphisms to every partially colored corolla and
a \emph{composition map} to every partially colored graph. 
An important observation is that the composition map obtained from any partially
colored graph can be decomposed into a finite sequence of maps each of which is
of one of four types, to be introduced in Proposition \ref{prop:CoroDecom}. 
We call an internal edge that connects a pair of distinct internal vertices a 
\emph{regular} edge, and any other internal edge a \emph{loop}. A partially 
colored embedded graph in the standard disk is called \emph{trivial} if it is isotopic
to a partially colored corolla or does not contain any internal vertex. A morphism
in $\crlla$ is called trivial if its underlying partially colored
embedded graph is trivial. A composition of the type 
  \be
  G_{2}\pcirc G_1 \coloneqq (K_1\,{\sqcup}\, K_2
  \xrightarrow{~G_1\,\sqcup\,\id~} K_3\,{\sqcup}\, K_2 \xrightarrow{~~G_2~~}K_4)
  \ee
is called a \emph{partial composition} of the morphisms $G_1$ and $G_2$.

\begin{prop} \label{prop:CoroDecom}
Any non-trivial morphism in $\crlla$ can be decomposed into a finite disjoint 
union of partial compositions of morphisms of the following types:
  \be
  \begin{aligned}
  \mathrm{(a)}~~ & \input{CALB15.tikz} \qquad\quad & \mathrm{(b)}~~ & \input{CALB16.tikz}
  \\[-2pt]
  & \text{operadic composition} && \text{~~~partial trace map}
  \\[4pt]
  \mathrm{(c)}~~ & \input{CALB17.tikz} \qquad &\quad \mathrm{(d)}~~ & \input{CALB18.tikz}
  \\[-2pt]
  & \text{~~horizontal product} && \text{~~~~~~whiskering}
  \end{aligned}
  \label{eq:abcd}
  \ee
\end{prop}

\begin{proof}
(1) If the morphism $G$ does not have any regular edge, jump to step (2) (with 
$G_n' \eq G$). Otherwise, pick a regular edge $e$ of $G$ and embed the standard disk
with image a small disk-shaped
neighborhood containing $e$. This embedding $\varphi$ pulls back a partially colored
graph on the standard disk, which after parametrizing each internal vertex by a 
corolla gives rise to a morphism $G_1$ of type (a). We thus have 
$G \eq G'_1\,{\pcirc}\,G_1^{\sss(1)}$, where $G'_1$ is obtained by replacing the part
of $G$ that is contained in the image of $\varphi$ by a corolla. We repeat this 
procedure until we end up with
$G \eq G'_n \,{\pcirc}\, G_n^{\sss(1)} \,{\pcirc}\dotsb{\pcirc}\, G_1^{\sss(1)}$, 
where each $G_i^{\sss(1)}$ is of type (a) and $G'_n$ does not contain any regular
edge. This expression is reached after finitely many steps, because all graphs are 
finite and in each step the number of internal vertices decreases. If $G'_{n}$ is
isotopic to a corolla, then the composite $G'_n\,{\pcirc}\, G_n^{\sss(1)}$ is of
type (a) and we are done. Otherwise we proceed to step (2).
 \\[2pt]
(2) If the morphism $G'_n$ does not have any loop, jump to step (3)
(with $G_m'' \eq G_n'$). Otherwise, pick a loop of $G'_n$, and embed a standard disk 
to a small neighborhood that contains that loop. This gives $G \eq G''_1\,{\pcirc}\,
G_{1}^{\sss(2)}\,{\pcirc}\, G_n^{\sss(1)}\,{\pcirc}\dotsb{\pcirc}\, G_1^{\sss(1)}$
with $G_1^{\sss(2)}$ of type (b). We repeat this procedure until we end up with
$G \eq G''_m\,{\pcirc}\, G_{m}^{\sss(2)}\,{\pcirc}\dotsb{\pcirc}\, G_1^{\sss(2)}
\,{\pcirc}\, G_n^{\sss(1)}\,{\pcirc}\,\dotsb\,{\pcirc}\, G_1^{\sss(1)}$ with each 
$G_j^{\sss(2)}$ of type (b) and $G''_m$ not containing any loop. If $G''_m$ is
isotopic to a corolla, then $G''_m\,{\pcirc}\, G_m^{\sss(2)}$ is of type
(b) and we are done. Otherwise, $G''_m$ is a union of corollas
and edges that do not contain internal vertices, and we proceed to step (3). 
 \\[2pt]
(3) If the morphism $G''_m$ contains at most one internal vertex, jump to step (4)
(with $G_l''' \eq G_m''$). Otherwise, pick a pair of internal vertices of $G''_m$ 
and embed a standard disk to a small neighborhood containing them. This gives 
$G \eq G'''_1\,{\pcirc}\, G_1^{\sss(3)}\,{\pcirc}\, G_m^{\sss(2)} \,{\pcirc}\dotsb
{\pcirc}\, G_1^{\sss(2)}\,{\pcirc}\, G_n^{\sss(1)}\,{\pcirc}\dotsb{\pcirc}\, G_1^{\sss(1)}$
with $G_1^{\sss(3)}$ of type (c). We repeat this procedure until
  \be
  G = G'''_l\,{\pcirc}\, G_l^{\sss(3)}\,{\pcirc}\dotsb{\pcirc}\, G_{1}^{\sss(3)}
  \,{\pcirc}\, G_m^{\sss(2)}\,{\pcirc}\dotsb{\pcirc}\, G_1^{\sss(2)}
  \,{\pcirc}\, G_n^{\sss(1)}\,{\pcirc}\dotsb{\pcirc}\, G_1^{\sss(1)} ,
  \ee
where each $G_{k}^{\sss(3)}$ is of type (c) and $G'''_l$ contains
a single internal vertex (and neither regular edges nor loops). If $G'''_l$ is isotopic
to a corolla, then $G'''_l\,{\pcirc}\, G_l^{\sss(3)}$ is of type (c) and we are 
done. Otherwise, $G'''_l$ consists of a single corolla and multiple edges without 
internal vertices, and we proceed to step (4).
 \\[2pt]
(4) After consecutively replacing each of the (finitely many) disk-shaped 
neighborhoods in $G'''_l$ that are of type (d) by a corolla we arrive at the desired
decomposition 
  \be
  G = G_k^{\sss(4)}\,{\pcirc}\,\dotsb\,{\pcirc}\, G_1^{\sss(4)}
  \,{\pcirc}\, G_l^{\sss(3)}\,{\pcirc}\,\dotsb\,{\pcirc}\, G_1^{\sss(3)}
  \,{\pcirc}\, G_m^{\sss(2)}\,{\pcirc}\,\dotsb\,{\pcirc}\, G_2^{\sss(2)}
  \,{\pcirc}\, G_n^{\sss(1)}\,{\pcirc}\,\dotsb\,{\pcirc}\, G_1^{\sss(1)},
  \ee
with all $G_p^{\sss(4)}$'s of type (d).
\end{proof}

As indicated in the list \eqref{eq:abcd}, we refer to the linear maps 
that are obtained by evaluating the functor $\calb\colon\crlla\To\vct$
at morphisms of the types (a), (b), (c), and (d) as operadic compositions,
partial trace maps, horizontal products, and whiskerings, respectively. Note that 
in our formulation of the graphical calculus, the horizontal product is \emph{not} a
special case of operadic composition, and whiskering is \emph{not} a special case of
horizontal product. Owing to Proposition \ref{prop:CoroDecom}, the linear map obtained 
from evaluating $\calb$ at any non-trivial morphism can be written either as a partial 
composite of such elementary maps or as a tensor product of such partial composites.
 
Next we consider the subcategory of $\crlla$ that contains all objects
of $\crlla$, but only those morphisms all of whose underlying graphs are connected 
in the disks they are embedded in. We denote this category by $\conncrlla$.
For instance, a finite disjoint union of morphisms 
of type (a) belongs to $\conncrlla$, whereas a morphism of type (c) does not. We have

\begin{cor} \label{cor:ConnCoroDecom}
The subcategory $\conncrlla$ is generated, under monoidal product $($i.e., disjoint
union$)$ and partial composition by the trivial morphisms with connected embedded 
graphs and by the morphisms of the types $\mathrm{(a)}$ and $\mathrm{(b)}$.
\end{cor}

We denote by 
  \be
  \conncalb\Colon \conncrlla \,{\hookrightarrow}\, \crlla \Rarr{\calb} \vct
  \ee
the restriction of the functor $\calb$ to the subcategory $\conncrlla$. According to 
Corollary \ref{cor:ConnCoroDecom}, evaluating $\conncalb$ at any non-trivial morphism
in its domain produces a partial composite of operadic compositions and partial trace 
maps, or a tensor product of such partial composites.

\begin{rem}
Inspiration for the constructions in this section comes from the study 
\cite{costK} of (symmetric) operads, cyclic operads and modular operads as 
symmetric monoidal functors defined on suitable categories of graphs. Here we work 
in a somewhat different setting: we deal with categories of graphs embedded in 
disks instead of graphs, and there is an additional coloring on the operads.
When the bicategory $\cb$ has a single object, i.e.\ $\cb$ is the delooping of 
some (strictly) pivotal tensor category, then the functor $\calb$ gives us its
underlying (not necessarily symmetric) colored modular operad. Since in the definition
of the category $\crlla$ we allow non-connected graphs to be embedded in the standard 
disk, this operad has horizontal products for operations, and the restricted functor 
$\conncalb$ gives us its underlying non-symmetric colored modular operad.
\end{rem}


\Subsection{Functoriality of the graphical calculus\label{sec:RSF}}

Let $\cb$ and $\cb'$ be strictly pivotal bicategories. A \emph{lax functor} from
$\cb$ to $\cb'$ is a triple $(F,F^{(2)},F^{(0)})$, where $F$ stands both for a map 
$F\colon\ob\cb\To\ob\cb'$ of objects and, for every pair of objects $a,b\iN\cb$,
for a \emph{local functor} $F\colon\cb(a,b)\To\cb'(Fa,Fb)$,
while $F^{(2)}$ and $F^{(0)}$ are families of natural transformations of the form
  \be
  \begin{tikzcd}
  \mathcal{B}(a,b)\Times \mathcal{B}(b,c) && \mathcal{B}(a,c)
  \\[5pt]
  \mathcal{B}'(Fa,Fb)\Times \mathcal{B}'(Fb,Fc) && \mathcal{B}'(Fa,Fc)
  \arrow["\hoco", from=1-1, to=1-3]	
  \arrow["F", from=1-3, to=2-3]
  \arrow["{F\times F}"', from=1-1, to=2-1]
  \arrow["{\hoco'}"', from=2-1, to=2-3]
  \arrow["{F^{(2)}}"', shorten <=20pt, shorten >=20pt, Rightarrow, from=2-1, to=1-3]
  \end{tikzcd}
  \ee
and
  \be
  \begin{tikzcd}
  1 & \mathcal{B}(a,a)
  \\
  & \mathcal{B}'(Fa,Fa)
  \arrow["{\mathrm{id}_a}", from=1-1, to=1-2]
  \arrow[""{name=0, inner sep=0}, "{\mathrm{id}_{Fa}}"', curve={height=12pt},
           end anchor={west}, from=1-1, to=2-2]
  \arrow["F", from=1-2, to=2-2]
  \arrow["{F^{(0)}}"{pos=0.6}, shift right=2, shorten <=9pt, shorten >=2pt, Rightarrow, from=0, to=1-2] 
  \end{tikzcd}
  \ee
for every triple $a,b,c\iN\cb$, where $\hoco$ and $\hoco'$ are the horizontal 
composition functors of $\cb$ and $\cb'$, respectively. 
$F^{(2)}$ and $F^{(0)}$ are called the \emph{lax functoriality constraint} and 
the \emph{lax unity constraint}, respectively; they are required to be natural 
and to satisfy conditions akin to the axioms of an algebra (i.e., a monoid object) 
in a monoidal category. The naturality of the lax unity constraint $F^{(0)}$ is 
redundant because the terminal category $1$ contains only the identity morphism.

We draw the component of $F^{(2)}$ on 1-morphisms $f$ and $g$ and the component of
$F^{(0)}$ on $\id_a$ as the string diagrams
  \be
  \input{RSFL0.tikz} \qquad\quad \text{and} \qquad\quad
  \raisebox{-0.7em}{\begin{tikzpicture}
	\begin{pgfonlayer}{nodelayer}
		\node [style=none] (28) at (0, 2.25) {};
		\node [style=antipode] (29) at (0, -1) {};
		\node [style=none] (30) at (0, 2.75) {$F\mathrm{id}_a$};
		\node [style=none] (31) at (0, -2.75) {};
	\end{pgfonlayer}
	\begin{pgfonlayer}{edgelayer}
		\draw [style=vt-bl-di] (29) to (28.center);
	\end{pgfonlayer}
\end{tikzpicture}
}
  \ee
respectively (throughout this section we use the conventional graphical calculus 
for bicategories). That is, an unlabeled trivalent vertex stands for a component 
of the lax functoriality constraint $F^{(2)}$ and an unlabeled univalent vertex 
for a component of the lax unity constraint $F^{(0)}$,
and we suppress both the symbol for the horizontal composition and the colors 
that indicate the objects labeling the different regions of the canvas.
To graphically describe the evaluation of the functor $F$ on a 2-morphism, we use 
(similarly as in e.g.\ \cite{mell4}) the symbol $F$ followed by a \emph{window} 
that encloses the string diagram that expresses the 2-morphism.
The defining properties of lax functoriality 
then amount to the following equalities of string diagrams:
\Enumerate
 \item
\emph{Naturality} is expressed as the family of equalities
  \be
  \input{RSF0.tikz} \quad=\quad \input{RSF1.tikz}
  \label{eq:laxFnat}
  \ee
for all objects $a,b,c\iN\cb$, 1-morphisms 
$f,f'\iN\cb(a,b)$ and $g,g'\iN\cb(b,c)$, and 2-morphisms $\alpha\colon f\Rightarro f'$
and $\beta\colon g\Rightarro g'$. 
 \item
\emph{Lax associativity} is expressed as the equalities
  \be
  \input{RSF2.tikz} ~~=~~ \input{RSF3.tikz}
  \label{eq:laxFass}
  \ee
for all composable triples $f,g,h$ of 1-morphisms in $\cb$.
 \item
\emph{Lax left and right unity} are expressed as
  \be
  \input{RSF4.tikz} ~~=~~ \begin{tikzpicture}
	\begin{pgfonlayer}{nodelayer}
		\node [style=none] (19) at (0, 3) {$Ff$};
		\node [style=none] (21) at (0, -2.5) {};
		\node [style=none] (27) at (0, -3) {$Ff$};
		\node [style=none] (28) at (0, 2.5) {};
	\end{pgfonlayer}
	\begin{pgfonlayer}{edgelayer}
		\draw [style=vt-bl-di] (21.center) to (28.center);
	\end{pgfonlayer}
\end{tikzpicture}
 ~~=~~ \input{RSF6.tikz}
  \label{eq:laxFuni}
  \ee
for all pairs of objects $a,b\iN\cb$ and all 1-morphisms $f\iN\cb(a,b)$.
 \end{enumerate}

For every composable string $a_1\Rarr{f_1}a_2\Rarr{f_2}{\dotsb}\Rarr{f_n}a_{n+1}$ 
of 1-morphisms in $\cb$, due to the lax associativity there is a unique 2-morphism 
  \be
  F_{f_1,\dotsc,f_n}^{(n)}\Colon Ff_{1}\Dots Ff_{n}\Longrightarrow F(f_1\Dots f_n) \,,
  \ee
from the composite of the 1-morphisms $Ff_i$ to the image of the composite of the
1-morphisms $f_i$ that is obtained by composing appropriate components of the lax
functoriality constraint in arbitrary order. This 2-morphism $F_{f_1,\dotsc,f_n}^{(n)}$
is in general not an isomorphism.  We depict it by the string diagram
  \be
  \input{RSF7.tikz}
  \ee

Dually, an \emph{oplax functor} between bicategories $\cb$ and 
$\cb'$ is a triple $(G,G_{(2)},G_{(0)})$, where $G_{(2)}$ and $G_{(0)}$, 
called \emph{oplax functoriality constraint} and \emph{oplax unity constraint},
are natural and obey conditions akin to those of a coalgebra in a 
monoidal category. (For details, see \cite[Sect.\,4.1]{JOYa}.)

Let now $F \,{\equiv}\, (F,F^{(2)},F^{(0)},F_{(2)},F_{(0)}) \colon\cb\To\cb'$ be a 
functor with both lax and oplax structures. 
The lax and oplax structures allow us to transform coupons in a way
that preserves valency.

\begin{defn} \label{def:Fconj}
Let $\alpha\colon f_1\hocO\dotsb\hocO f_m\,{\xRightarrow{~~}}\, g_1\hocO\dotsb\hocO 
g_n$ be a 2-morphism in a bicategory $\cb$, with $f_1,...\,,f_m$ and $g_1,...\,,g_n$
1-morphisms in $\cb$, and let $F\colon\cb\Rarr~\cb'$ be a functor with
lax and oplax structures between bicategories. The \emph{$F$-conjugate}
of $\alpha$ is the 2-morphism $\alpha^{F}$ in the target bicategory
$\cb'$ that is obtained by composition with the lax and oplax constraints: If
all the 1-morphisms in the domain and codomain of $\alpha$ are non-trivial, then
the $F$-conjugate of $\alpha$ is the composite
  \be
  \alpha^{F} \Colon F\!f_{1}{\hoco}\dotsb{\hoco} F\!f_{m}
  \xRightarrow{F_{f_{1},\dotsc,f_{m}}^{(m)}} F(f_{1}{\hoco}\dotsb{\hoco} f_{m})
  \xRightarrow{\,F\alpha_{\phantom{(}}} F(g_{1}{\hoco}\dotsb{\hoco} g_{n})
  \xRightarrow{F_{(n)\,g_{1},\dotsc,g_{n}}} F\!g_{1}{\hoco}\dotsb{\hoco} F\!g_{n}
  \ee
which is represented by the string diagram
  \be
  \input{RSF8.tikz}
  \label{eq:F-conj}
  \ee
If any of the 1-morphisms is trivial, we need to insert in addition the
lax and/or oplax unity constraints accordingly. For instance, if the
1-morphisms $f_1$ and $g_n$ are trivial, then $\alpha^{F}$ is defined to be
  \be
  \input{RSF50.tikz}
  \label{eq:F-conj-1}
  \ee
\end{defn}

\begin{rem}
According to Definition \ref{def:Fconj}, the precise form of the conjugation depends 
on the notion of a trivial 1-morphism. Conventionally, by the trivial 1-endomorphism 
of an object $a$ in a bicategory $\cb$ one means the identity morphism
$\id_a \iN \cb(a,a)$. However, in order to correctly describe the functoriality
of graphical calculi in terms of a monoidal natural transformation (see Theorem
\ref{thm:RSFconj} below), we need to 
sharpen this notion of triviality: in our context, a 1-morphism is trivial 
iff it is an identity 1-morphism \emph{and} it is not an edge color.
\end{rem}

For a generic functor with both lax and oplax structures there is no reason for 
the conjugation to preserve compositions or partial traces, nor to respect the 
coherent isomorphisms between 2-hom spaces that are related by dualities. To deal with
a functor that does satisfy such relations we need to impose appropriate properties on
the lax and oplax structures. Specifically, we need the following 
notions. (In the 1-object case, i.e.\ for monoidal categories, these have been 
studied in e.g.\ \cite{mcSt} and have found applications in constructions of 
topological field theories \cite{mulev3}.)

\begin{defn}
Let $F \,{\equiv}\, (F,F^{(2)},F^{(0)},F_{(2)},F_{(0)})\colon\cb\to\cb'$ be a functor
between two strictly pivotal bicategories that is equipped with lax and oplax structures. 
 \Itemize
 \item[{\rm (i)}]
$F$ is called \emph{rigid} if $F$-conjugation preserves the units and counits
of the duals, i.e.\ if for all $a,b\iN\cb$ and all $f\iN\cb(a,b)$ we have
$F(f^\vee) \eq (Ff)^\vee$ and 
  \be
  \input{RSF9.tikz}\quad=\quad\input{RSF10.tikz}
  \label{eq:Frigid}
  \ee
(the `antenna' on the left hand side stands for the oplax unit constraint 
for the object $b$) as well as similar conditions for the coevaluations.
 \item[{\rm (ii)}]
$F$ is called \emph{separable} if for every pair $a\Rarr{\,f} b\Rarr{\,g} c$
of composable 1-morphisms the lax and oplax structures are related by
  \be
  \input{RSF11.tikz} \quad=\quad \begin{tikzpicture}
	\begin{pgfonlayer}{nodelayer}
		\node [style=none] (4) at (0, 3) {};
		\node [style=none] (14) at (0, -3) {};
		\node [style=none] (22) at (0, -3.5) {$F(fg)$};
		\node [style=none] (23) at (0, 3.5) {$F(fg)$};
	\end{pgfonlayer}
	\begin{pgfonlayer}{edgelayer}
		\draw [style=vt-bl-di] (14.center) to (4.center);
	\end{pgfonlayer}
\end{tikzpicture}

  \label{eq:Fseparable}
  \ee
 \item[{\rm (iii)}]
$F$ is called \emph{Frobenius}\,%
if for every triple $a\Rarr{\,f} b\Rarr{\,g} c\Rarr{\,h} d$
of composable 1-morphisms the two compatibility relations
  \be
  \input{RSF13.tikz} \quad~~=\quad \input{RSF14.tikz}
  \label{eq:Ffrob1}
  \ee
and 
  \be
  \input{RSF15.tikz} \quad~~=\quad \input{RSF16.tikz}
  \label{eq:Ffrob2}
  \ee
between the lax and oplax structure are fulfilled.
\end{itemize}
\end{defn}

\begin{rem} \label{rem:FrobFrob}
It is worth stressing that the `Frobenius functors' considered here are 2-functors 
between bicategories. This notion should not be confused with the one of a Frobenius
functor between categories, defined as a functor admitting a two-sided adjoint.
A rigid separable Frobenius functor equips the image of the identity 1-morphism 
of any object in the domain bicategory with the structure of a $\Delta$-separable 
symmetric Frobenius algebra, equips the image of every 1-morphism in the domain 
bicategory with the structure of a bimodule, and equips the conjugate of every
2-morphism in the domain bicategory with the structure of a bimodule morphism.
(Here the notions of bimodule and bimodule morphism are generalized to the
bicategorical setting, compare e.g.\ \Cite{Sect.\,2.2}{caRun3}.)
\end{rem}

We will now show that a rigid separable Frobenius functor \emph{almost} preserves 
the graphical calculus on disks. Let $K$ be a partially $\cb$-colored corolla on the
standard disk and let $F\,{\equiv}\, (F,F^{(2)},F^{(0)},F_{(2)},F_{(0)})\colon\cb\To\cb'$
be a rigid separable Frobenius functor. The map of objects and the local functors 
that are entailed by $F$ give rise to a symmetric monoidal functor 
$F_{*}\colon\crlla\To\pcrlla$ by accordingly changing the colors of the patches and 
of the edges of the partially colored graphs. Moreover, since the change of colors 
does not affect the connectedness of the embedded graphs, this functor $F_*$ restricts
to a functor $F_{*}\colon\conncrlla\To\pconncrlla$. Restricting to the
subcategory $\conncrlla$ turns out to have an important consequence:

\begin{thm} \label{thm:RSFconj}
Let $F\colon\cb\To\cb'$ be a rigid Frobenius functor between two strictly pivotal
bicategories. The $F$-conjugation introduced in Definition $\ref{def:Fconj}$
canonically induces a monoidal natural transformation
  \be
  \begin{tikzcd}
  \conncrlla
  \\
  && {\mathrm{Vect}_{\mathbbm{k}}}
  \\
  \pconncrlla
  \arrow["{F_*}"', from=1-1, to=3-1]
  \arrow[""{name=0,anchor=center,inner sep=0},"{\!\!\!\GCal_{\mathcal B}^{\mathrm{conn}}}", from=1-1, to=2-3]
  \arrow["{\GCal_{\mathcal{B}'}^{\mathrm{conn}}}"', from=3-1, to=2-3]
  \arrow["{(-)^F}"{pos=0.4},shift right=2,shorten <=13pt,shorten >=13pt,Rightarrow, from=0, to=3-1]
  \end{tikzcd}
  \ee
$($to which we still refer as $F$-conjugation$)$. Its component at a corolla
$K\iN\conncalb$ is given by 
  \be
  (-)_K^F\Colon \conncalb(K) \stackrel{p_v^k}{\xrightarrow{\hspace*{0.9cm}}}
  \widehat{h_{v}^{\cb}}(k) \stackrel{(-)^F}{\xrightarrow{\hspace*{0.9cm}}}
  \widehat{h_{v'}^{\cb'}}(k) \stackrel{(p_{v'}^k)^{-1}}{\xrightarrow{\hspace*{0.9cm}}}
  \conncalb(F_*K)
  \label{eq:FconjK}
  \ee
for any choice of polarization $k$ on the vertex $v\iN K$, where
$p_v^k$ is the isomorphism defined in formula \eqref{eq:pvk} and where
the labels of all edges are treated as non-trivial 1-morphisms, while
$(-)_{K_1\sqcup K_2}^F \,{\coloneqq}\, (-)_{K_1}^F\Otk(-)_{K_2}^F$.
\end{thm}

\begin{proof}
We first show that (\ref{eq:FconjK}) is well defined, i.e.\ does not depend on 
the choice of polarization. This amounts to showing that the $F$-conjugation commutes
with all canonical isomorphisms between 2-hom spaces that are related by duality. We
present the argument for one specific case -- its generalization to other cases is
straightforward: we claim that the diagram
  \be
  \begin{tikzcd}[column sep=2.7em,row sep=2.1em]
  {\mathrm{Hom}_{\mathcal{B}(a,a)}(\mathrm {id}_a,f\hocO g)} 
  & {\mathrm{Hom}_{\mathcal{B}(a,b)}(g^{\vee},f)} 
  \\
  {\mathrm{Hom}_{\mathcal{B}'(Fa,Fa)}(\mathrm {id}_{Fa},Ff\hocO Fg)}
  & {\mathrm{Hom}_{\mathcal{B}'(Fa,Fb)}((Fg)^{\vee},Ff)}
  \arrow["{(-)^F}"', from=1-1, to=2-1] \arrow["{(-)^F}", from=1-2, to=2-2]
  \arrow["\cong", from=1-1, to=1-2]    \arrow["\cong"', from=2-1, to=2-2]
  \end{tikzcd}
  \label{eq:square2}
  \ee
commutes for all objects $a,b\iN\cb$ and 1-morphisms $f\iN\cb(a,b)$ and $g\iN\cb(b,a)$
(note that by rigidity we have $F(g^{\vee}) \eq (Fg)^{\vee}$). To prove this claim we
pick a 2-morphism $\alpha\iN\hom_{\cb(a,a)}(\id_a,fg)$. Since the 1-morphisms $f$
and $g$ come from the coloring of edges, they are regarded as
non-trivial, and they will be treated as such by the $F$-conjugation. Now when first 
going right and then downwards in the square \eqref{eq:square2} we get
  \be
  \input{RSF17.tikz} \quad\longmapsto\quad \input{RSF18.tikz}
  \quad\longmapsto\quad \input{RSF19.tikz}
  \label{eq:square2rd}
  \ee
The graph on the right hand side of \eqref{eq:square2rd} can be rewritten as
  \be
  \begin{aligned}
  \input{RSF19a.tikz} \quad & =\quad\input{RSF20.tikz} \quad=\quad \input{RSF21.tikz}
  \\[0.5em]~\\
  & \equ{eq:Ffrob1} \quad \input{RSF22.tikz} \quad=\quad \input{RSF23.tikz}
  \end{aligned}
  \label{eq:square2rd2}
  \ee
The final graph in \eqref{eq:square2rd2} is the same as the one obtained when
tracing the diagram \eqref{eq:square2} first downwards and then right.
 \\[2pt]
Next we address naturality. The naturality squares for the trivial morphisms in 
$\conncrlla$ commute trivially. In view of Corollary \ref{cor:ConnCoroDecom} we
need to verify naturality only for morphisms of the types (a) and (b), i.e.\ we must
show that $F$-conjugation commutes with operadic compositions and partial trace maps.
First consider operadic compositions. To keep the discussion short we pick 
a specific morphism of type (a), say $G\colon K_1\,{\sqcup}\, K_2\To K_3$ given by 
  \be
  G ~=\quad \input{RSFOC.tikz}
  \label{eq:G=fgh}
  \ee
with $f,g,h\iN\cb(a,b)$. The naturality square for \eqref{eq:G=fgh} reads
  \be
  \begin{tikzcd}[column sep=5.9em,row sep=2.9em]
  {\GCal_{\mathcal B}^{\mathrm{conn}}(K_1) \Otk \GCal_{\mathcal B}^{\mathrm{conn}}(K_2)} 
  & {\GCal_{\mathcal B}^{\mathrm{conn}}(K_3)} 
  \\
  {\GCal_{\mathcal{B}'}^{\mathrm{conn}}(F_*K_1) \Otk 
  \GCal_{\mathcal{B}'}^{\mathrm{conn}}(F_*K_2)}
  & {\GCal_{\mathcal{B}'}^{\mathrm{conn}}(F_*K_3)}
  \arrow["{(-)^F_{K_1} \otk\, (-)^F_{K_2}}"', from=1-1, to=2-1]
  \arrow["{(-)^F_{K_3}}", from=1-2, to=2-2]
  \arrow["{\GCal_{\mathcal B}^{\mathrm{conn}}(G)}", from=1-1, to=1-2]
  \arrow["{\GCal_{\mathcal{B}'}^{\mathrm{conn}}(F_*G)}"', from=2-1, to=2-2]
  \end{tikzcd}
  \label{eq:natsquare:operadic}
  \ee
We choose a polarization for each corolla $K_i$, as well as elements
$\alpha\iN\widehat{h_{v_1}^{\cb}}(k_1) \eq \hom_{\cb(a,a)}(\id_a,fg^{\vee})$
and $\beta\iN\widehat{h_{v_2}^{\cb}}(k_2) \eq \hom_{\cb(a,a)}(\id_{a},gh)$ in the
respective 2-hom spaces for $K_1$ and $K_2$. The following calculation shows that
the two paths in the square \eqref{eq:natsquare:operadic} give 
the same element in $\widehat{h_{v_3}^{\cb}}(k_{3}) \eq \hom_{\cb(a,a)}(\id_a,fh)$:
  \be
  \hspace*{-0.6em}
  \begin{aligned}
  (\alpha \,{\hoco_g}\,\beta)^F~ & =\quad \input{RSF24.tikz} \quad=\quad \input{RSF25.tikz}
  \\~\\[-2pt]
  & = \quad \input{RSF26.tikz} \quad \equ{eq:Ffrob2} \quad \input{RSF27.tikz}
  \\~\\[-2pt]
  & \!\!\! \equ{eq:Ffrob1} ~~ \input{RSF28.tikz} \quad \equ{eq:Frigid} \quad \input{RSF29.tikz}
  \quad=~ \alpha^F {\hoco_{Fg}^{}}\,\beta^F .
  \end{aligned}
  \ee
(Here $(-)^F$ denotes $F$-conjugation (see Definition \ref{def:Fconj}), while
$\hoco_g$ stands for horizontal composition along the 1-morphism $g$.) 
 \\[2pt]
Next consider partial trace maps. Let $G\colon K_1\To K_2$
be the morphism of type (b) in $\conncrlla$ given by 
  \be
  G ~=\quad\input{RSFPT.tikz}
  \ee
with $f\iN\cb(a,a)$ and $g\iN\cb(b,a)$, for which the naturality square is
  \be
  \begin{tikzcd}[column sep=5.7em,row sep=2.6em]
  {\GCal_{\mathcal B}^{\mathrm{conn}}(K_1)}
  & {\GCal_{\mathcal B}^{\mathrm{conn}}(K_2)}
  \\
  {\GCal_{\mathcal B'}^{\mathrm{conn}}(F_*K_1)} 
  & {\GCal_{\mathcal B'}^{\mathrm{conn}}(F_*K_2)}
  \arrow["{(-)^F_{K_1}}"', from=1-1, to=2-1]
  \arrow["{(-)^F_{K_2}}", from=1-2, to=2-2]
  \arrow["{\GCal_{\mathcal B}^{\mathrm{conn}}(G)}", from=1-1, to=1-2]
  \arrow["{\GCal_{\mathcal B'}^{\mathrm{conn}}(F_*G)}"', from=2-1, to=2-2]
  \end{tikzcd}
  \ee
We select polarizations $k_1$ and $k_2$ for the corollas $K_1$ and $K_2$ and choose a
2-morphism $\alpha\iN\widehat{h_{v_1}^{\cb}}(k_1) 
      $\linebreak[0]$
    {=}\, \hom_{\cb(a,a)}(\id_a,fg^{\vee}g)$,
and we consider the case that $\widehat{h_{v_2}^{\cb}}(k_2) \eq \hom_{\cb(a,a)}(\id_a,f)$. 
Then the following calculation shows that taking the partial trace commutes with the 
$F$-conjugation:
  \be
  \begin{aligned}
  (\tr_g\,\alpha)^{F}~ & = \quad \input{RSF30.tikz} \quad=\quad \input{RSF31.tikz}
  ~~ \equ{eq:Fseparable} ~~ \input{RSF32.tikz}
  \\~\\[-2pt]
  & =\quad \input{RSF33.tikz} ~\quad=~ \input{RSF34.tikz} \quad=~ \tr_{Fg}\,\alpha^F .
  \end{aligned}
  \ee
In conclusion, we did construct a natural transformation 
$(-)^{F}\colon\conncalb\Rightarro\pconncalb\cir F_{*}$. Moreover, this natural
transformation is evidently monoidal.
\end{proof}

It is worth stressing that a rigid separable Frobenius functor does \emph{not}, in 
general, preserve the \emph{entire} graphical calculus on disks. To understand 
what goes wrong, consider a morphism $G\colon K_1\,{\sqcup}\, K_2\To K_3$
of horizontal-product type (c), say
  \be
  G ~=\quad \input{RSFHP.tikz}
  \ee
with $f_1,f_2\iN\cb(a,b)$ and $g_1,g_2\iN\cb(b,c)$. Upon choosing appropriate 
polarizations and 2-morphisms, $G$ yields a horizontal product
  \be
  \input{RSF35.tikz}
  \ee
The $F$-conjugate of this horizontal product is
  \be
  (\alpha\hocO\beta)^F ~=\quad \input{RSF36.tikz} \quad=\quad \input{RSF37.tikz}
  \ee
Thus the so obtained 2-morphism differs from 
$\alpha^F{\hoco}\,\beta^F \eq F\alpha\hocO F\beta$ by
  \be
  \input{RSF38.tikz}
  \label{eq:RSFidem}
  \ee
which owing to the separability of $F$ is an idempotent 
on the space of 2-endomorphism of $F f_2 \hocO F g_2$. This observation motivates

\begin{defn}  
A functor $F\colon\cb\To\cb'$ equipped with both lax and oplax structures for which
the lax functoriality constraint is the inverse of its oplax functoriality constraint
is called \emph{strongly separable}. 
\end{defn}  

{}From the previous observations it follows immediately that a \emph{rigid strongly
separable Frobenius functor} preserves \emph{horizontal products} as well as 
\emph{whiskerings}, and thus preserves the complete graphical calculus on disks of its
domain bicategory. However, there is in fact no need for a separate notion of 
a strongly separable Frobenius functor. Recall that a \emph{pseudofunctor} can be 
regarded as a functor with lax and oplax structures that are mutually inverse. We have

\begin{lem}
For $F$ a functor between bicategories that is equipped with lax and oplax structures,
we have:
 \Itemize
 \item [{\rm(i)}] 
$F$ is strongly separable if and only if it is a pseudofunctor.
 \item [{\rm(ii)}]
If $F$ is a pseudofunctor, then it is Frobenius.
 \end{itemize}
\end{lem}

\begin{proof}
(i) A pseudofunctor is strongly separable by definition. To see the converse, assume 
that $F$ is strongly separable. Then what remains to be shown is that $F^{(0)}$ and 
$F_{(0)}$ are mutually inverse. Now for any object $a\iN\cb$ we have 
  \be
  \input{RSF39.tikz} \quad=\quad \input{RSF40.tikz} \quad=\quad \input{RSF41.tikz}
  \quad=\quad \begin{tikzpicture}
	\begin{pgfonlayer}{nodelayer}
		\node [style=none] (0) at (0, 3) {};
		\node [style=none] (3) at (0, -3) {};
		\node [style=none] (4) at (0, -3.5) {$F\mathrm{id}_a$};
		\node [style=none] (5) at (0, 3.5) {$F\mathrm{id}_a$};
	\end{pgfonlayer}
	\begin{pgfonlayer}{edgelayer}
		\draw [style=vt-bl] (0.center) to (3.center);
	\end{pgfonlayer}
\end{tikzpicture}

  \ee
where the premise that $F$ is strongly separable is used in the second equality,
showing that $F_{a}^{(0)}\cir F_{(0)\,a} \eq \id_{Fa}$. 
Moreover, for a 1-morphism $f\iN\cb(a,b)$ with non-zero 2-endomorphism space we have
  \be
  \begin{tikzpicture}
	\begin{pgfonlayer}{nodelayer}
		\node [style=none] (0) at (0, 3) {};
		\node [style=none] (3) at (0, -3) {};
		\node [style=none] (4) at (0, -3.5) {$Ff$};
		\node [style=none] (5) at (0, 3.5) {$Ff$};
	\end{pgfonlayer}
	\begin{pgfonlayer}{edgelayer}
		\draw [style=vt-bl-di] (3.center) to (0.center);
	\end{pgfonlayer}
\end{tikzpicture}
 \quad=\quad \input{RSF44.tikz} \quad=\quad \input{RSF45.tikz}
  \ee
so that also $F_{(0)\,a}\cir F_{a}^{(0)} \eq \id_{Fa}$.
(If there does not exist such a 1-morphism $f$ for any $b\iN\cb$, then the desired 
statement holds true trivially.)
 \\[3pt]
(ii) We have
  \be
  \input{RSF15a.tikz} \quad=\quad \input{RSF46.tikz}
  ~~=~~ \input{RSF47.tikz} ~\,=~ \input{RSF16a.tikz} \qquad
  \ee
where associativity is used in the second equality, while strong separability
(which holds by part (i)) is used in the first and third.
Thus \eqref{eq:Ffrob2} is satisfied. The proof of \eqref{eq:Ffrob1} is analogous.
\end{proof}

In particular, a strongly separable Frobenius functor is just the same as a 
pseudofunctor. Now recall the notion of $F$-conjugation; when restricting to the
stronger class of functors consisting of rigid pseudofunctors, in Theorem 
\ref{thm:RSFconj} we can drop the restriction on the allowed corollas, so that it
is sharpened to

\begin{cor} \label{cor:RPconj}
Let $F\colon\cb\To\cb'$ be a rigid pseudofunctor between strictly pivotal 
bicategories. The $F$-conjugation $($defined in Theorem $\ref{thm:RSFconj})$
canonically extends to a monoidal natural transformation
  \be
  \begin{tikzcd}
  \crlla
  \\
  && \mathrm{Vect}_{\mathbbm{k}}
  \\
  \pcrlla
  \arrow["{F_*}"', from=1-1, to=3-1]
  \arrow[""{name=0,anchor=center,inner sep=0}, "\!\!\!\GCal_{\mathcal{B}}", from=1-1, to=2-3]
  \arrow["\!\GCal_{\mathcal{B}'}"', from=3-1, to=2-3]
  \arrow["{(-)^F}"{pos=0.4},shift right=2,shorten <=12pt,shorten >=12pt,Rightarrow, from=0, to=3-1] 
  \end{tikzcd}
  \ee
\end{cor}

We end this section with an example of a rigid separable Frobenius functor
that plays an important role in the application to conformal quantum field theories
that we will discuss in Section \ref{sec:SNC}.

\begin{example} \label{exa:UFrCtoBC}
Let $\cc$ be a strictly pivotal fusion category and $\cfrc$ the strictly pivotal
bicategory of simple special symmetric Frobenius algebras in $\cc$ (see Example
\ref{exa:cfrc}). Recall that $\cc$ can be viewed as a strictly pivotal bicategory $\cbc$
(its delooping, see Example \ref{exa:bc}) with a single object $*$. Consider the functor
  \be
  \cu \Colon \cfrc\rarr~\cbc
  \label{eq:UFrCtoBC}
  \ee
that is defined by 
  \be
  \begin{tikzcd}
  A &&&& {*} 
  \\ 
  & {} & {} & {} \\
  B &&&& {*}
  \arrow["{\mathcal{U}}", shorten <=14pt, shorten >=14pt, maps to, from=2-2, to=2-4]
  \arrow[""{name=0, anchor=center, inner sep=0}, "X"', curve={height=12pt},
            start anchor={south west}, end anchor={north west}, from=1-1, to=3-1]
  \arrow[""{name=1, anchor=center, inner sep=0}, "Y", curve={height=-12pt},
            start anchor={south east}, end anchor={north east}, from=1-1, to=3-1]
  \arrow[""{name=2, anchor=center, inner sep=0}, "{\dot{X}}"', curve={height=12pt},
            start anchor={south west}, end anchor={north west}, from=1-5, to=3-5]
  \arrow[""{name=3, anchor=center, inner sep=0}, "{\dot{Y}}", curve={height=-12pt},
           start anchor={south east}, end anchor={north east}, from=1-5, to=3-5]
  \arrow["\alpha" {yshift=2pt}, shorten <=7pt, shorten >=7pt, Rightarrow, from=0, to=1]
  \arrow["{\dot{\alpha}}" {yshift=2pt},shorten <=7pt,shorten >=7pt,Rightarrow, from=2, to=3] 
  \end{tikzcd}
  \ee
for Frobenius algebras $A,B\iN\cfrc$, $A$-$B$-bimodules $X,Y$ and a bimodule morphism 
$\alpha\colon X\Rightarro Y$, i.e.\ $\cu$ sends every object in $\cfrc$ to the sole 
object $*$ in $\cbc$ and every bimodule and bimodule morphism to their underlying 
object and morphism in $\cc$, respectively. In the sequel, we suppress the dot
and just use the same symbol for bimodules and bimodule morphisms as for
their underlying objects and morphisms.
 \\[2pt]
The functor $\cu$ is canonically a rigid separable Frobenius functor,
with the components 
  \be
  \cu_{X,Y}^{(2)}\Colon X\oti Y\rarr~ X\Otb Y \qquad\text{and}\qquad
  \cu_{(2)\,X,Y}\Colon X\Otb Y\rarr~ X\oti Y
  \ee
of the lax and oplax functoriality constraints at $(X,Y)\iN\cfrc(A,B)\Times\cfrc(B,C)$
given by the splitting of the idempotent that realizes the tensor product $\otimes_B$ as
a retract of $\otimes$, i.e.\ $\cu_{X,Y}^{(2)} \cir \cu_{(2)\,X,Y} \eq \id_{X\otb Y}$ and
  \be
  \cu_{(2)\,X,Y}\circ\cu_{X,Y}^{(2)}~
  =\quad \input{RSF48.tikz} \quad=\quad \input{RSF49.tikz}
  \label{eq:uconjidem}
  \ee
and with the lax and oplax unity constraints given by the units and counits
of the Frobenius algebras. Note that to satisfy the axioms, 
associators and unitors must be inserted accordingly. 
Also, for the example to be non-trivial, a crucial requirement is that
special symmetric Frobenius algebras other than the monoidal unit exist in $\cc$.
\end{example}


\section{String-net models based on pivotal bicategories} \label{sec:SNbicat}

In Section \ref{chap:GCalc} we have developed the graphical calculus on a canvas 
that is homeomorphic to a disk. It is natural to seek an extension of this
calculus for which the canvas can have non-trivial topology. String-net models 
for surfaces provide such an extension.
Throughout this section $\cb$ is a strictly pivotal bicategory. In addition we assume
that $\cb$ is small and locally small, i.e.\ its objects
form a set and all its hom-categories are small.

\Subsection{Bicategorical string-net spaces} \label{sec:BSNspaces}

Recall from Definition \ref{def:bdydatum} the notion of a $\cb$-boundary datum
on a compact oriented 1-manifold. Let $\surf$ be a compact oriented surface
and let $\msb$ be a $\cb$-boundary datum on $\partial\surf$.
A (fully) \emph{$\cb$-colored graph} $\grph$ on $\surf$ with $\cb$-boundary datum 
$\msb$ on $\partial\surf$ is a partially $\cb$-colored graph $\ogrph$ on $\surf$
together with a coloring of its internal vertices, i.e.\ a choice of an element in
the vector space $H_{v}^{\cb}$ for each internal vertex $v\iN V(\ogrph)$, such that
the canonical embedding $\partial\surf\,{\hookrightarrow}\,\surf$, when viewed as
an outgoing parametrization of the boundary, pulls back the $\cb$-boundary datum 
$\msb$ on $\partial\surf$. 

Denote by $\mathrm{G}(\surf,\msb)$ the set of all $\cb$-colored graphs on $\surf$
with prescribed $\cb$-boundary datum $\msb$, and by $\fk\mathrm{G}(\surf,\msb)$
the $\fk$-vector space freely generated by it. Also recall from Definition
\ref{def:value} the notion of the \emph{value} of a $\cb$-colored embedded graph on 
the standard disk. We define the 
\emph{string-net space} assigned to the pair $(\surf,\msb)$ as follows:

\begin{defn} \label{def:osnb}
Let $\cb$ be a strictly pivotal bicategory, $\surf$ a compact
oriented surface, and $\msb$ a $\cb$-boundary datum on $\partial\surf$. 
 \begin{itemize}
 \item [{\rm(i)}] 
A \emph{null graph} on $\surf$ is an element $\sum_i\lambda_i\grph_i$ of 
$\fk\mathrm{G}(\surf,\msb)$ such that there exists an embedding
$\varphi\colon D \,{\hookrightarrow}\, \mathrm{int}(\surf)$ of the standard disk $D$
to the interior of $\surf$ that satisfies the following requirements:
the circle $\varphi(\partial D)$ does not
contain any vertex of any of the graphs $\grph_{i}$; any intersection of
$\varphi(\partial D)$ and an edge of any of the graphs $\grph_{i}$ is transversal;
on the complement $\surf\,{\setminus}\,\varphi(D)$ all graphs $\grph_{i}$ coincide;
and the values of the graphs pulled back by $\varphi$ sum up to zero, 
$\sum_i\lambda_i\, \langle\grph_i\cap \varphi(D)\rangle_D \eq 0$.
 \item [{\rm(ii)}] 
The (bare) \emph{string-net space} $\osnb(\surf,\msb)$ is the quotient 
  \be
  \osnb(\surf,\msb) \coloneqq \fk\mathrm{G}(\surf,\msb)/\mathrm{N}(\surf,\msb) \,,
  \label{eq:osnb}
  \ee
where $\mathrm{N}(\surf,\msb)$ is the subspace of $\fk\mathrm{G}(\surf,\msb)$
spanned by all null graphs on $\surf$.
 \end{itemize}
\end{defn}

We call a vector in the quotient space $\osnb(\surf,\msb)$ that
is the image of an element of the generating set $\mathrm{G}(\surf,\msb)$
of $\fk\mathrm{G}(\surf,\msb)$ a \emph{bare string net}, or also just
a \emph{string net}. The qualification ``bare'' and the notation $\osnb$ 
are chosen because later on we will introduce a Karoubified version $\snb$
of string-net spaces for which the boundary data are
enriched by certain idempotents. A string net that has a $\cb$-colored graph
$\grph$ as a representative will be denoted by $[\grph]$. By abuse of language the
term \emph{string net} is also used when referring to an individual graph that
represents an element $[\grph]\iN\osnb(\surf,\msb)$. The string-net space
$\osnb(\surf,\msb)$ is linear in the color of each vertex of a graph  $\grph$ and
additive with respect to taking direct sums of objects labeling an edge of $\grph$.
By the nature of the graphical calculus for disks, isotopic graphs
represent the same string net. Furthermore, all identities that are valid in
the graphical calculus for $\cb$ also hold inside any disk embedded
in $\surf$. In other words, for string nets the graphical calculus for $\cb$
applies locally on $\surf$.

Homeomorphisms of the surface $\surf$ act naturally on embedded graphs. 
Isotopies can be localized on disks \Cite{Cor.\,1.3}{edKi}; as a consequence,
isotopic graphs are identified in the string-net space, and hence 
the action of homeomorphisms descends to an action of the mapping
class group $\mcg(\surf)$ of the surface on the space $\osnb(\surf,\msb)$. 
Moreover, string nets with matching boundary data can be \emph{concatenated}.

\begin{example}
Let $a,b,c\iN\cb$ be three objects in a pivotal bicategory (in the color version of 
the picture \ref{eq:pic:pop} below, they are indicated as green, blue, and purple,
respectively), and let $f_{1},f_{5}\iN\cb(b,c)$, $f_{2},f_{4}\iN\cb(c,b)$, 
$f_{3}\iN\cb(b,b)$ and $f_{6}\iN\cb(a,c)$ be 1-morphisms in $\cb$. Indicate by
  \be
  \input{SNB0.tikz}
  \ee
a $\cb$-boundary datum $\msb$ on the boundary of a genus-1
surface $\surf$ with three boundary circles. Then
  \be
  \grph ~= \quad\input{SNB1.tikz}
  \label{eq:pic:pop}
  \ee
is an element of the set $\mathsf{G}(\surf,\msb)$ of $\cb$-colored
graphs on $\surf$, where $\clc_{1}$ and $\clc_{2}$
are elements of the appropriate spaces of vertex colors. $\grph$
represents a string net $[\grph]\iN\osnb(\surf,\msb)$.
\end{example}

\begin{rem}
Recall that a color $\clc\iN H_{v}^{\cb}$ for an internal vertex $v$
of a partially $\cb$-colored graph $\ogrph$ is completely determined
by a choice of 2-morphism $\clc_k\iN\widehat{h_{v}^{\cb}}(e_k)$ for
any choice of polarization on $v$, and the 2-morphisms for different
choices of polarization are related by coherent isomorphisms. Since
these coherent isomorphisms are devised in such a way that the string
diagrams produced according to different choices of polarizations have 
the same value when restricted to embedded disks, it is equally admissible,
if not more convenient for calculations, to represent string nets by
string diagrams with rectangular coupons.
\end{rem}

\begin{example} \label{exa:osnbd}
Let $\msb$ be a $\cb$-boundary datum on $S^{1} \eq \partial D$. Recall that 
according to the prescription \eqref{eq:kb} there is a unique partially colored 
corolla $K^{\msb}$ associated to $\msb$. We have a canonical isomorphism 
  \be
  \osnb(D,\msb)\iso\calb(K^{\msb})
  \ee
that is given by the evaluation $\grph \,{\xmapsto{\,~\,}}\, 
\clc_{\grph}^{} \,{\xmapsto{\,\hcalb(\ogrph)\,}}\, \langle\grph\rangle$,
where $\grph$ is any graph representing the string net $[\grph] \iN \osnb(D,\msb)$, 
and where $\clc_{\grph}^{} \,{\in}\, \bigotimes_{v\in V(\ogrph)}H_v^{\cb}$
stands for the coloring of the internal vertex of the underlying partially 
colored graph $\ogrph$ that is given by $\grph$. The linear map obtained by this 
prescription is well defined 
because the graphical calculus is \emph{local} in nature; it is injective 
since graphs that evaluate on $D$ to the same value also must have the same 
value when evaluated on a slightly smaller disk embedded in $D$ (after being 
replaced by isotopic graphs when necessary) and hence represent the same string net;
and it is also surjective because, given any element $\clc\iN\calb(K^{\msb})$,
coloring the center of $K^{\msb}$ with $\clc$ produces a fully
colored corolla $K_\clc^{\msb}$ whose value is $\clc$ itself.
\end{example}

\begin{rem}
\Enumerate
 \item 
The idea to utilize pivotal bicategories for the construction of modular functors,
and even of topological field theories, dates back at least to \cite{moWa}.
 \item 
When the input bicategory $\cb$ is the delooping of a spherical fusion category
$\cc$, Definition \ref{def:osnb} reduces to the definition of string-net spaces
for $\cc$ that was used in e.g.\ \cite{leWe,kirI24,fusY,bart8}. We then write 
  \be
  \osnc(\surf,\msb)\equiv\mathrm{SN}_{\cbc}^{\circ}(\surf,\msb) \,.
  \ee
It should be appreciated that our definition of a string-net space 
$\mathrm{SN}_{\cb}^{\circ}$ for a general strictly pivotal bicategory $\cb$
does not require any homological properties such as semisimplicity or finiteness.
Also note that, while sphericality is needed for the string-net modular functor based
on $\cc$ to be isomorphic to the Turaev-Viro modular functor, it is 
neither required for the string-net construction itself \cite{runk15},
nor for other constructions of modular functors \Cite{Sect.\,6.4}{brWo}.
\end{enumerate}
\end{rem}


\Subsection{String-net spaces as colimits} \label{sec:SN-colimit}

As we will see now, the string-net spaces, which were defined as quotients 
of vector spaces, admit a natural description as colimits.
Such a description will e.g.\ be instrumental when trying to generalize
string-net constructions to a derived setting.

\begin{defn}
Let $\cb$ be a strictly pivotal bicategory, $\surf$ a compact oriented
surface, and $\msb$ a $\cb$-boundary datum on $\partial\surf$. 
 \begin{itemize}
 \item [{\rm(i)}] 
$\grb(\surf,\msb)$ is the following category: The objects of $\grb(\surf,\msb)$
are partially $\cb$-colored graphs on $\surf$ that have $\msb$ as their boundary
datum. The morphisms of $\grb(\surf,\msb)$ are generated under composition (and
upon adjoining identities) by the morphisms exemplified in the following example:
  \be
  \input{SNCL0.tikz} \quad\stackrel{\input{SNCL1.tikz}}{\xrightarrow{\hspace*{4.7cm}}}
  \quad\input{SNCL2.tikz} \quad
  \label{eq:grphbmor}
  \ee
In more detail, a generating morphism is given by an embedding $\varphi$ of the 
standard disk $D$ into the interior of $\surf$, such that $\varphi(\partial D)$ 
intersects the edges of the partially colored graph $\ogrph_{1}$ in the domain 
transversally and does not meet any vertices, while the codomain is the graph 
$\ogrph_2$ that is obtained by replacing $\ogrph_1\,{\cap}\,\varphi(D)$ with the 
image of the unique partially colored corolla on $D$ associated with the boundary 
datum on $S^{1} \eq \partial D$ pulled back by $\varphi$.
 \item [{\rm(ii)}] 
The \emph{evaluation functor}
  \be
  \ce_{\cb}^{\surf,\msb} \Colon \grb(\surf,\msb) \rarr~ \vct
  \label{eq:evafunctor}
  \ee
is the following functor: $\ce_{\cb}^{\surf,\msb}$ sends an object, i.e.\ a 
partially $\cb$-colored graph $\ogrph$ on $\surf$, to the vector space 
$\ce_{\cb}^{\surf,\msb}(\ogrph) \,{\coloneqq}\, \bigotimes_{v\in V(\ogrph)} 
H_v^{\cb}$; and it sends a generating morphism $\gamma\colon\ogrph_1\To\ogrph_2$,
with associated embedding $\varphi\colon D \,{\hookrightarrow}\, \surf$,
to the linear map 
  \be
  \ce_{\cb}^{\surf,\msb}(\gamma)\colon~~ \bigotimes_{v\in V(\ogrph_{1})}\!\!\!
   H_v^{\cb} \xrightarrow{~\id\,\otk\,\hcalb(\ogrph_1\cap D_\gamma)~}\!\!
   \bigotimes_{v'\in V(\ogrph_2)}\!\!\!\! H_{v'}^{\cb}
  \ee
that is obtained by applying the graphical calculus on disks to the partially
colored graph on $D$ pulled back by the embedding $\varphi$.
\end{itemize}
\end{defn}

\begin{thm} \label{thm:SN-colim}
Let $\cb$ be a strictly pivotal bicategory, $\surf$ a compact oriented
surface, and $\msb$ a $\cb$-boundary datum on $\partial\surf$. The string-net space
$\osnb(\surf,\msb)$ satisfies
  \be
  \osnb(\surf,\msb) = \colim\ce_{\cb}^{\surf,\msb} ,
  \ee
where the legs of the cocone are given by 
  \be
  \begin{aligned}
  \ce_{\cb}^{\surf,\msb}(\ogrph) = \bigotimes_{v\in V(\ogrph)}\!\! H_v^{\cb}
  & \rarr~\, \osnb(\surf,\msb) \,,
  \\
  \clc = \!\bigotimes_{v\in V(\ogrph)}\!\! \clc_v & \xmapsto{~~~} [\ogrph_{\clc}]
  \end{aligned}
  \label{eq:coconeosnb}
  \ee
for every $\ogrph\iN\grb(\surf,\msb)$. Here $\ogrph_{\clc}$ is the fully colored graph 
that is obtained by coloring $\ogrph$ with $\clc\iN\bigotimes_{v\in V(\ogrph)}H_{v}^{\cb}$.
\end{thm}

\begin{proof}
Since graphs related by the local graphical calculus of $\cb$ represent the same 
string net, the prescription \eqref{eq:coconeosnb} indeed gives rise to a cocone.
To show that the cocone is initial, consider an arbitrary cocone 
  \be
  \{f_{\ogrph}\colon\ce_{\cb}^{\surf,\msb}(\ogrph)\To V\}
  _{\ogrph\in\grb(\surf,\msb)}^{} 
  \label{eq:coconef}
  \ee
to some $\fk$-vector space $V$. We need to show that there is a 
unique linear map $f\colon\osnb(\surf,\msb)\To V$ that makes the diagram 
  \be
  \begin{tikzcd}
  \ce_{\cb}^{\surf,\msb}(\ogrph)
  \\[4pt]
  \osnb(\surf,\msb) && V
  \arrow[from=1-1, to=2-1]
  \arrow["f"', dashed, from=2-1, to=2-3]
  \arrow["{f_{\mathring{\grph}}^{}}", from=1-1, to=2-3]
  \end{tikzcd}
  \ee
commute for every object $\ogrph\iN\grb(\surf,\msb)$. We claim that the 
desired linear map is given by
  \be
  \grph \xmapsto{~~} \clc_{\grph} \xmapsto{~~} f_{\ogrph}(\clc_{\grph}) \,,
  \ee
where $\grph$ is any graph representing $[\grph] \iN \osnb(\surf,\msb)$ and
$\clc_{\grph}\iN\ce_{\cb}^{\surf,\msb}(\ogrph)$.
To show that this map is well-defined, assume that the two colored graphs $\grph$
and $\grph'$ both represent the string-net $[\grph]$. By the definition of 
the string-net space, the underlying partially colored graphs $\ogrph$ and 
$\ogrph'$ are connected by a \emph{zigzag} 
  \be
  \begin{tikzcd} && \dots & {\mathring{\grph}'} 
  \\[-14pt]
  & \vdots
  \\
  \mathring{\grph}^{(1)} & \mathring{\grph}^{(2)}
  \\
  \mathring{\grph}
  \arrow[from=3-1, to=4-1] \arrow[from=3-1, to=3-2]
  \arrow[from=1-3, to=1-4] \arrow[from=2-2, to=3-2] 
  \end{tikzcd}
  \label{eq:zig}
  \ee
in the category $\grb(\surf,\msb)$, and the corresponding vertex colors 
$\clc_{\grph}$ and $\clc_{\grph'}$ are related by the zigzag in $\vct$ that 
is obtained by applying the functor $\ce_{\cb}^{\surf,\msb}$ to \eqref{eq:zig},
and are therefore mapped to the same element in $V$ by the legs of the
cocone (\ref{eq:coconef}). By construction, $f$ makes the relevant
diagrams commute and is unique.
\end{proof}

By recognizing the string-net space $\osnb(\surf,\msb)$ as a colimit we shed
new light on the canonical mapping class group action: Let
$\xi\iN\mcg(\surf)$ be a mapping class group element and $x$ a homeomorphism 
representing $\xi$. Then there is a canonical natural isomorphism 
  \be
  \begin{tikzcd}
  {\mathcal{G}\mathrm{raphs}_\mathcal{B}(\surf,\mathsf{b})} 
  \\
  && {\mathrm{Vect}_{\mathbbm{k}}}
  \\
  {\mathcal{G}\mathrm{raphs}_\mathcal{B}(\surf,\mathsf{b})}
  \arrow["{x_*}"', from=1-1, to=3-1]
  \arrow[""{name=0,anchor=center,inner sep=0},
         "{\mathcal{E}_\mathcal{B}^{\surf,\mathsf{b}}}", from=1-1, to=2-3]
  \arrow["{\mathcal{E}_\mathcal{B}^{\surf,\mathsf{b}}}"', from=3-1, to=2-3]
  \arrow["\cong"{pos=0.4},shift right=2,shorten <=13pt,shorten >=13pt,Rightarrow, from=0, to=3-1]
  \end{tikzcd}
  \ee
whose component at an object $\ogrph\iN\grb(\surf,\msb)$ is 
the canonical identification 
  \be
  \ce_{\cb}^{\surf,\msb}(\ogrph) = \!\bigotimes_{v\in V(\ogrph)}\!\! H_v^{\cb}
  ~\rarr{~\cong~} \!\!\bigotimes_{v'\in V(x_{*}\ogrph)}\!\!\!\! H_{v'}^{\cb}
  = \ce_{\cb}^{\surf,\msb}(x_{*}\ogrph) \,,
  \ee
where $x_{*}$ is the endofunctor obtained by pushing forward the
partially colored graphs. Now consider the square
  \be
  \begin{tikzcd}[row sep=2.1em]
  {\mathcal{E}_\mathcal{B}^{\surf,\mathsf{b}}(\mathring{\grph})
  = \bigotimes_{v\in V(\mathring{\grph})}\! H_v^{\mathcal{B}}}
  && {\mathrm{SN}^{\circ}_\mathcal{B}(\surf,\mathsf{b})} 
  \\
  {\mathcal{E}_\mathcal{B}^{\surf,\mathsf{b}}(x_* \mathring{\grph})
  = \bigotimes_{v'\in V(x_* \mathring{\grph})}\! H_{v'}^{\mathcal{B}}}
  && {\mathrm{SN}^{\circ}_\mathcal{B}(\surf,\mathsf{b})}
  \arrow["\cong"',shift left=5, from=1-1, to=2-1]
  \arrow["{\mathrm{SN}^{\circ}_\mathcal{B}(\xi,\mathsf{b})}", dashed, from=1-3, to=2-3]
  \arrow[from=1-1, to=1-3]
  \arrow[from=2-1, to=2-3]
  \end{tikzcd}
  \label{eq:squareEESS}
  \ee
Since the composite 
of the left vertical isomorphism and the horizontal morphism in the bottom row
is a component of a natural transformation followed by a leg of a cocone, it is the
leg of a cocone under $\ce_{\cb}^{\surf,\msb}$. As a consequence, there is a unique
endomorphism $\osnb(\xi,\msb)$ of the string-net space $\osnb(\surf,\msb)$ which 
provides the dashed vertical arrow that makes the square \eqref{eq:squareEESS} 
commute. By a straightforward diagram-chase we see that this endomorphism coincides 
with the action of $\xi \eq [x]\iN\mcg(\surf)$ on $\osnb(\surf,\msb)$.

Note that the vector space $\osnb(\surf,\msb)$ is, in general, not finite-dimensional.
In particular, $\osnb(D,\msb)$ is infinite-dimensional if the corresponding 2-hom 
space is infinite-dimensional.


\Subsection{Functoriality under rigid pseudofunctors} \label{sec:functoriality}

As we have seen in Section \ref{sec:RSF}, rigid pseudofunctors
preserve the graphical calculus on disks for strictly pivotal bicategories.
Since the string-net spaces are built with the help of this graphical calculus,
one should expect that a rigid pseudofunctor between such bicategories induces 
canonical linear maps between the respective string-net spaces. Indeed we have

\begin{thm} \label{thm:osnb-posnb}
\label{thm:osnf}Let $\surf$ be a compact oriented surface, $\cb$
and $\cb'$ two strictly pivotal bicategories, $F\colon\cb\To\cb'$
a rigid pseudofunctor, and $\msb$ a $\cb$-boundary datum on $\partial\surf$.
There is a canonical $\mcg(\surf)$-intertwiner 
  \be
  \mathrm{SN}_{F}^\circ(\surf,\msb)
  \Colon\osnb(\surf,\msb) \rarr~ \posnb(\surf,F_*\msb) \,,
  \label{eq:SNB2SNBp}
  \ee
where $F_*\msb$ is the $\cb'$-boundary datum on $\partial\surf$
obtained by changing the coloring according to the map of objects and the local
functors entailed by $F$. The linear map \eqref{eq:SNB2SNBp} is defined by 
sending each representing $\cb$-colored graph to the $\cb'$-colored graph
obtained by applying the $F$-conjugation \emph{(\ref{eq:F-conj})}. Moreover,
the collection of such intertwiners corresponding to different surfaces
and boundary data is compatible with the concatenation of string nets.
\end{thm}

\begin{proof}
Consider the natural transformation 
  \be
  \begin{tikzcd}
  \mathcal{G}\mathrm{raphs}_\mathcal{B}(\surf,\mathsf{b})
  \\
  && \mathrm{Vect}_{\mathbbm{k}}
  \\
  \mathcal{G}\mathrm{raphs}_{\mathcal{B}'}(\surf,F_*\mathsf{b})
  \arrow["{F_*^{}}"', from=1-1, to=3-1]
  \arrow[""{name=0, anchor=center, inner sep=0},
         "{\mathcal{E}_\mathcal{B}^{\surf,\mathsf{b}}}", from=1-1, to=2-3]
  \arrow["{\mathcal{E}_{\mathcal{B}'}^{\surf,F_*\mathsf{b}}}"', from=3-1, to=2-3]
  \arrow["{(-)^F}"{pos=0.4},shift right=2,shorten <=14pt,shorten >=14pt,Rightarrow, from=0, to=3-1]
  \end{tikzcd}
  \ee
whose component at $\ogrph\iN\grb(\surf,\msb)$ is the $F$-conjugation
  \be
  (-)_{\ogrph}^{F} \Colon\ce_{\cb}^{\surf,\msb}(\ogrph)
  = \bigotimes_{v\in V(\ogrph)}\!\! H_{v}^{\cb}\
  \rarr~ \!\!\bigotimes_{v'\in V(F_{*}\ogrph)}\!\! H_{v'}^{\cb'}
  = \ce_{\cb'}^{\surf,F_{*}\msb}(\ogrph) \,.
  \ee
The naturality square for a generating morphism $\gamma\colon\ogrph_1\To\ogrph_2$
reads 
  \be
  \begin{tikzcd}[column sep=8.8em,row sep=2.1em]
  \bigotimes_{v_1\in V(\mathring{\grph}_1)}H_{v_1}^{\mathcal{B}}
  & \bigotimes_{v_2\in V(\mathring{\grph}_2)}H_{v_2}^{\mathcal{B}}
  \\
  \bigotimes_{v_1'\in V(F_*\mathring{\grph}_1)}H_{v_1'}^{\mathcal{B}}
  & \bigotimes_{v_2'\in V(F_*\mathring{\grph}_2)}H_{v_2'}^{\mathcal{B}}
  \arrow["{(-)^F_{\mathring{\grph}_1}}"', from=1-1, to=2-1]
  \arrow["{(-)^F_{\mathring{\grph}_2}}", from=1-2, to=2-2]
  \arrow["{\mathrm{id}\,\mathbin{\otimes_{\mathbbm{k}}}\,\mathsf{\widehat{\GCal}}_{B}
         (\mathring{\grph}_{1}\cap D_{\gamma})}", from=1-1, to=1-2]
  \arrow["{\mathrm{id}\,\mathbin{\otimes_{\mathbbm{k}}}\,\mathsf{\widehat{\GCal}}_{B'}
         (F_*\mathring{\grph}_{1}\cap D_{F_*\gamma})}"', from=2-1, to=2-2] 
  \end{tikzcd}
  \ee
Commutativity of this diagram follows from Corollary \ref{cor:RPconj}. Now consider 
the square
  \be
  \begin{tikzcd}[column sep=3.8em,row sep=2.1em]
  {\mathcal{E}_\mathcal{B}^{\surf,\mathsf{b}}(\mathring{\grph})
  = \bigotimes_{v\in V(\mathring{\grph})}\! H_{v}^{\mathcal{B}}}
  & {\mathrm{SN}^{\circ}_\mathcal{B}(\surf,\mathsf{b})}
  \\
  {\mathcal{E}_{\mathcal{B}'}^{\surf,F_*\mathsf{b}}(F_*\mathring{\grph})
  = \bigotimes_{v'\in V(F_*\mathring{\grph})}\! H_{v'}^{\mathcal{B}}}
  & {\mathrm{SN}^{\circ}_{\mathcal{B}'}(\surf,F_*\mathsf{b})}
  \arrow["{(-)^F_{\mathring{\grph}}}"', shift left=5, from=1-1, to=2-1]
  \arrow["{\mathrm{SN}^{\circ}_F(\surf,\mathsf{b})}", dashed, from=1-2, to=2-2]
  \arrow[from=1-1, to=1-2, "{\eqref{eq:coconeosnb}}"]
  \arrow["{\eqref{eq:coconeosnb}}", from=2-1, to=2-2] 
  \end{tikzcd}
  \ee
Since the composite of the left vertical arrow and the horizontal arrow in the bottom
row is a component of a natural transformation followed by a leg of a cocone, it is the
leg of a cocone under $\ce_{\cb}^{\surf,\msb}$. Therefore there is a unique linear map
  \be
  \mathrm{SN}_F^{\circ}(\surf,\msb) \Colon\osnb(\surf,\msb) \rarr~ \posnb(\surf,F_*\msb)
  \ee
that makes the square commute. A direct diagram-chase shows that this
linear map has the asserted form. Equivariance and 
compatibility with concatenation are evident.
\end{proof}

\begin{rem}
Since a rigid separable Frobenius functor in general preserves horizontal products 
and whiskerings only up to idempotents of the form \eqref{eq:RSFidem}, the change 
of colors that is brought about by conjugation with respect to a rigid separable 
Frobenius functor \emph{does not descend} to linear maps between string-net spaces.
However, as we will see in Theorem \ref{thm:universal} as a special case, every 
rigid separable Frobenius functor does induce linear maps (intertwining the
mapping class group actions) between string-net spaces with the help of 
Frobenius graphs.
\end{rem}


\Subsection{Cylinder categories over circles} \label{sec:oCylS}

The notion of string-net spaces for a strictly pivotal bicategory
$\cb$ allows us to promote the set of $\cb$-boundary data on a closed
oriented 1-manifold to a (small) $\fk$-linear category.

\begin{defn} \label{def:oCylS}
Let $\cb$ be a strictly pivotal bicategory and $\ell$ a closed oriented
1-manifold. If $\ell$ is non-empty, the \emph{cylinder category} $\ocyl(\cb,\ell)$
for \emph{$\cb$} over $\ell$ is the following category:
An object of $\ocyl(\cb,\ell)$ is a $\cb$-boundary datum on $\ell$. A morphism of
$\ocyl(\cb,\ell)$ between two boundary data is given by a string net on the cylinder 
$\ell\Times I$ that matches the boundary data at  $\ell\Times \{0\}$ and  $\ell\Times
\{1\}$. The composition of morphisms is given by the concatenation of string nets.
 \\[2pt]
For the empty 1-manifold $\emptyset$ we set $\ocyl(\cb,\emptyset)\,{\coloneqq}\,\vct$.
\end{defn}

\begin{example}
For instance, for any choice of $\alpha$ and $\beta$,
  \be
  \input{CC0.tikz}\quad\colon\qquad \input{CC2.tikz} \quad\rarr{~~}\quad \input{CC1.tikz}
  \label{eq:ocylbmor}
  \ee
is a morphism in $\ocyl(\cb,S^{1})$, with an appropriate $\cb$-coloring.
(The boundary component $\ell\Times\{0\}$ is regarded as \emph{in-coming} and supports
the domain boundary datum, hence the opposite convention for the edge labels.) 
\end{example}

\begin{rem}
It follows in particular that for every compact oriented surface $\surf$ the string-net
construction provides a functor 
  \be
  \osnb(\surf,-)\Colon \ocyl(\cb,\partial\surf) \rarr~ \vct
  \ee
which maps a morphism of the cylinder category to the linear map that is obtained
by sewing the cylinder to the boundary.
\end{rem}


\Subsection{Pointed pivotal bicategories and cylinder categories over intervals} \label{sec:oCylI}

In Definition \ref{def:oCylS} we have introduced cylinder categories over 
closed oriented 1-manifolds, which uses a pivotal bicategory $\cb$ as an 
input. This is sufficient for obtaining a closed modular functor. In order to
obtain instead an open-closed modular functor -- a goal that we will achieve in
Section \ref{sec:open-closed} -- we need to define cylinder categories 
for 1-manifolds with boundary as well. As we will explain in the next section,
to this end we must make use of a \emph{pointed} 
strictly pivotal bicategory $(\cb,*_{\cb})$, i.e.\ a strictly pivotal bicategory
endowed with a distinguished object $*_{\cb}\in\cb$.

\begin{defn} \label{def:oCylI}
Let $(\cb,*_{\cb})$ be a pointed strictly pivotal bicategory and $\ell$ a 
compact oriented 1-manifold with possibly non-empty boundary. The 
\emph{cylinder category $\ocyl(\cb,*_{\cb},\ell)$ for $(\cb,*_{\cb})$ over
$\mathscr{\ell}$} is the following category:
The objects of $\ocyl(\cb,*_{\cb},\ell)$ are the \emph{$(\cb,*_{\cb})$-boundary
data} on $\ell$, i.e.\ the $\cb$-boundary data whose 1-cells adjacent to a 
boundary point of $\ell$ are all colored with the distinguished object $*_{\cb}$. 
The morphisms of $\ocyl(\cb,*_{\cb},\ell)$ are string nets on the cylinder
$\ell\Times I$ over $\ell$. Their composition is given by concatenating string nets.
\end{defn}

We use a specific color (gray, in the color version) to indicate the distinguished 
object $*_{\cb}$. As an illustration,
  \be
  \msb ~=~ \input{CI0.tikz}
  \ee
is an object in $\ocyl(\cb,*_{\cb},I)$. 
For a \emph{closed} oriented 1-manifold $\ell$, we just have
$\ocyl(\cb,*_{\cb},\ell)\equiv\ocyl(\cb,\ell)$.

The following examples of pointed strictly pivotal  bicategories are important to us:

\begin{example}
The delooping $\cbc$ of a strictly pivotal tensor category $\cc$ has only a single
object and is thus automatically pointed.  The restriction on the objects in Definition 
\ref{def:oCylI} is vacuous in this case. Accordingly we write 
  \be
  \ocyl(\cc,\ell)\equiv\ocyl(\cbc,*,\ell)
  \ee
and call an object in this bicategory a \emph{$\cc$-boundary value} on $\ell$.
\end{example}

\begin{example}\label{exa:Frc*}
The strictly pivotal bicategory $\cfrc$ of simple special symmetric Frobenius
algebras in a strictly pivotal tensor category $\cc$ is canonically pointed with 
  \be
  *_{\cfrc}^{} \coloneqq \tu \,\in\cfrc \,,
  \ee
i.e.\ the tensor unit of $\cc$, canonically viewed as a Frobenius algebra.
\end{example}


\Subsection{Functoriality under embeddings} \label{sec:func-embed}

The primary reason for taking as the categorical input for the definition of cylinder 
categories over compact oriented 1-manifolds with boundary a \emph{pointed} strictly
pivotal bicategory is that we want the prescription to be functorial
with respect to the embedding of manifolds. This requirement arises, for instance, 
as a natural property when thinking of topological field theories and their modular 
functors in the spirit \cite{brfv} of general covariance in local quantum field theory.
For us, functoriality under embeddings is a crucial ingredient for being able to
associate a profunctor $\cosnb(\surf;-,\sim)$ to a two-dimensional bordisms $\surf$,
which will be done in \eqref{eq:SNoprofunctor} below.

Any continuous map $f\colon\ell_{1}\To\ell_{2}$ 
between 1-manifolds $\ell_{1}$ and $\ell_{2}$ extends canonically
to a map $f{\times} I\colon\ell_{1}\Times I\To\ell_{2}\Times I$ between
the corresponding cylinders, by setting $f{\times} I\colon(p,t) \,{\mapsto}\, (f(p),t)$.
Therefore an orientation preserving automorphism
$x\colon S^1\To S^1$ of the circle induces a functor 
  \be
  \ocyl(\cb,x)\Colon \ocyl(\cb,S^{1}) \rarr~ \ocyl(\cb,S^{1})
  \ee
by pushing forward the objects via $x$ and the morphisms via $x\Times I$.
In particular, the cylinder category $\ocyl(\cb,S^{1})$ carries an action
of the circle group $\mathrm{U}(1) \,{\subset}\, \mathrm{Aut}(S^{1})$.

As we will see now, in order to achieve functoriality also under embeddings of general 
oriented 1-manifolds, it does not suffice to take a strictly pivotal bicategory $\cb$ 
as an input, but we also need to fix a suitable distinguished object $*_{\cb}$ of $\cb$.
To understand this requirement, consider first an orientation preserving embedding 
$f\colon\bigsqcup_{i=1}^{n}\! I \,{\hookrightarrow}\, I$ of $n$ copies of the
standard interval into itself. Once we impose the restriction
on the $\cb$-coloring of the 1-cells of the boundary data as stated in Definition 
\ref{def:oCylI}, we can define a functor 
  \be
  \ocyl(\cb,*_{\cb},f)\Colon \ocyl(\cb,*_{\cb},\bigsqcup_{i=1}^{n}I)
  \,=\, \prod_{i=1}^{n}\ocyl(\cb,*_{\cb},I) \rarr~ \ocyl(\cb,*_{\cb},I)
  \label{eq:oCylnItoI}
  \ee
by pushing forward the objects and morphisms via $f$ and $f\Times I$,
respectively and then coloring the complement $I\,{\setminus}\,\mathrm{im\,}f$
(respectively, $I^{2}\,{\setminus}\,\mathrm{im}(f\Times I)$) with the object $*_{\cb}$.
In contrast, in the absence of a distinguished object of $\cb$ no consistent
coloring of these complements is possible. As a special case, the embedding 
$\emptyset \,{\hookrightarrow}\, I$ of the empty 1-manifold
induces a functor of the type $\ocyl(\cb,\emptyset) \eq \vct\Rarr~\ocyl(\cb,*_{\cb},I)$
by sending $\fk$ to the $\cb$-boundary datum $\msb_{I}^{*}$ on
$I$ that has the entire interval colored with $*_{\cb}$. (This
works analogously for the embedding $\emptyset \,{\hookrightarrow}\, \ell$ for
any oriented 1-manifold $\ell$. In this sense, \emph{all} cylinder
categories are canonically pointed, by pointing the input bicategory.)
Further, by fixing a binary embedding $I\,{\sqcup}\, I\To I$, the prescription
\eqref{eq:oCylnItoI} endows the category $\ocyl(\cb,*_{\cb},I)$ with a monoidal 
structure, with the tensor unit given by $\msb_{I}^*$. Moreover, by sending
every object in $\ocyl(\cb,*_{\cb},I)$ to the corresponding horizontal
composite of the coloring 1-morphisms, we obtain an equivalence
$\ocyl(\cb,*_{\cb},I) \,{\simeq}\, \mathrm{End}_{\cb}(*_{\cb})$
of tensor categories, with the monoidal product for the endomorphism category given
by the horizontal composition. Thus we have arrived at

\begin{prop} \label{prop:ocylIEnd}
Let $(\cb,*_{\cb})$ be a pointed strictly pivotal
bicategory. There is a canonical monoidal equivalence 
  \be
  \ocyl(\cb,*_{\cb},I) \,\simeq\, \mathrm{End}_{\cb}(*_{\cb}) \,.
  \ee
\end{prop}

Next consider an embedding $f\colon I \,{\hookrightarrow}\, S^1$. Again
by pushing forward the objects (respectively, morphisms) via the embedding
(respectively, via $f\Times I$) and coloring the complement of the image with the
distinguished color, we obtain a functor 
  \be
  \ocyl(\cb,*_{\cb},f)\Colon \ocyl(\cb,*_{\cb},I)
  \rarr~ \ocyl(\cb,S^1) \equiv\ocyl(\cb,*_{\cb},S^1) \,.
  \ee
This is demonstrated in the following picture:
  \be
  \input{FE0.tikz} \qquad\xmapsto{~~~~}\qquad \input{FE1.tikz}
  \ee
~

We now promote the assignment of cylinder categories to 1-manifolds to a symmetric
monoidal functor between the following
categories $\mathrm{Emb}_1^{\mathrm{or}}$ and $\cat_{\fk}$:

\begin{defn}
The symmetric monoidal category $\mathrm{Emb}_1^{\mathrm{or}}$ has as objects 
compact oriented 1-manifolds and as morphisms orientation preserving embeddings; 
the monoidal product of $\mathrm{Emb}_1^{\mathrm{or}}$ is disjoint union. 
  \\[2pt]
The symmetric monoidal category $\cat_{\fk}$ has as objects small $\fk$-linear
categories and as morphisms $\fk$-linear functors; the monoidal product of $\cat_{\fk}$
is the Cartesian product.
\end{defn}
 
Combining the considerations above we have

\begin{prop} \label{prop:cyl-funct}
The assignment of cylinder categories for a pointed pivotal bicategory
$(\cb,*_{\cb})$ to compact oriented 1-manifolds canonically extends
to a symmetric monoidal functor
  \be
  \ocyl(\cb,*_{\cb},-)\Colon \mathrm{Emb}_1^{\mathrm{or}} \rarr~ \cat_{\fk} \,.
  \ee
\end{prop}

\begin{rem}
As already mentioned, every cylinder category is pointed by the embedding of
the empty manifold $\emptyset$. As a consequence,
the functor $\ocyl(\cb,*_{\cb},-)$ factors through the forgetful functor
$\mathrm{Cat}_{\fk}^{\mathrm{pointed}}\To\mathrm{Cat}_{\fk}$.
Even better, we actually obtain a symmetric monoidal 2-functor $\cc\mathrm{yl}
^{\circ}(\cb,*_{\cb},-)\colon\ce\mathrm{mb}_1^{\mathrm{or}}\To\ccat_{\fk}$,
where $\ce\mathrm{mb}_{1}^{\mathrm{or}}$ is the symmetric monoidal (2,1)-category 
-- that is, a symmetric monoidal bicategory with only invertible 2-morphisms --
of compact oriented 1-manifolds, orientation preserving
embeddings and isotopy classes of isotopies between embeddings.
\end{rem}

Finally, let $\ell$ be a compact oriented 1-manifold and $\bar{\ell}$
the same underlying 1-manifold but with opposite orientation.
Due to the strict pivotality of $\cb$ we have a canonical isomorphism
  \be
  \ocyl(\cb,*_{\cb},\ell)^{\op}\iso\ocyl(\cb,*_{\cb},\bar{\ell})
  \ee
which sends an object $\overline{\msb} \iN \ocyl(\cb,*_{\cb},\ell)^{\op}$ to the
boundary datum $\msb^\vee$ on $\bar{\ell}$ that is obtained from $\msb$ by taking 
the duals of the coloring 1-morphisms. Through this identification an orientation
\emph{reversing} embedding $f\colon\ell_1 \To \ell_2$ induces a functor 
  \be
  \ocyl(\cb,*_{\cb},f)\Colon 
  \ocyl(\cb,*_{\cb},\ell_1)^{\op} \rarr~ \ocyl(\cb,*_{\cb},\ell_2) \,.
  \ee

\begin{rem} \label{rem:shadowTHH}
The string net-construction can be adapted to bicategories $\cb$ 
that do not possess a pivotal structure \cite{knsT}: one endows both 1-
and 2-manifolds with 2-framings and restricts to string nets 
that are \emph{progressive} with respect to the 2-framing.
Concretely, as before we regard the standard sphere $S^1$ as a the unit circle 
in the complex plane $\complex$ and the standard interval $I$ as the subset 
$[0,1]$ of the real axis. A 2-framing on the cylinder over $I$ and over $S^1$, 
respectively, is then obtained by using as a first vector field a non-zero 
vector field given by the orientation of $I$ and $S^1$, and as a second vector 
field in the case of $I$ one that uniformly points towards the positive 
$y$-direction of $\complex$ and in the case of $S^1$ one that points
from each point on $S^1$ towards the origin $0\iN\complex$. Upon fixing a 
framed embedding $I\,{\hookrightarrow}\, S^1$, we then obtain a functor 
  \be
  \bigsqcup_{a\in\cb}\! \cb(a,a) \simeq \bigsqcup_{a\in\cb}\! \ocyl(\cb,a,I)
  \xrightarrow{~\bigsqcup_{a\in\cb}\ocyl(\cb,a,f)~} \ocyl(\cb,S^1) \,.
  \label{eq:shadowsn}
  \ee
This functor admits the structure of a \emph{categorified trace} on $\cb$: A 
categorified trace on a bicategory $\cb$ with values in a category $\ca$ is a 
functor $ [\![ - ]\!] \colon \bigsqcup_{a\in\cb}\cb(a,a)\Rarr~\ca$ from the
hom-categories of $\cb$ to the category $\ca$ equipped with a natural isomorphism
$ \theta\colon [\![ fg ]\!] \Rarr\cong [\![ gf ]\!] $ for any pair $f\colon a\To b$
and $g\colon b\To a$ of cyclically composable 1-morphisms, satisfying two hexagon
and two triangle identities \cite{ponto,poSh} (compare also \cite{fScS}).
If the natural isomorphism $\theta$ is an involution, then the categorified 
trace is called \emph{symmetric}, or also a \emph{shadow}.
An important example of a symmetric categorified trace is provided by the 
\emph{topological Hochschild homology} (THH) of spectral categories 
\Cite{Thm.\,2.17}{caPo2}. Recently \cite{bermJ,heRa}, the notion of 
topological Hochschild homology was extended to that of a bicategory. The THH of
a bicategory $\cb$ is a category $\mathrm{THH}(\cb)$ given by a pseudocolimit of 
a certain 2-truncated cyclic bar construction. As shown in \Cite{Thm.\,3.19}{heRa},
$\mathrm{THH}(\cb)$ is canonically endowed with the structure of a
\emph{universal shadow} on $\cb$, in the sense that for every category $\cd$
there is a canonical equivalence 
  \be
  \mathrm{Fun}(\mathrm{THH}(\cb),\cd) \rarr~ \mathcal{S}\mathrm{ha}(\cb,\cd)
  \ee
of categories, where $\mathcal{S}\mathrm{ha}(\cb,\cd)$
is the category of shadows on $\cb$ with values in $\cd$. We expect
that the cylinder category $\ocyl(\cb,S^{1})$ provides a non-symmetric
analogue of $\mathrm{THH}(\cb)$ -- the \emph{topological cyclic homology} of 
the bicategory $\cb$. More specifically, we expect that the non-symmetric 
categorified traces obtained this way are universal, and that the cylinder 
categories over the circle are pseudocolimits of a variant of the cyclic bar 
construction in which the simplicial category $\varDelta$ is replaced by the 
cyclic category $\varDelta C$ that was introduced in \cite{conn5}.
\end{rem}


\Subsection{Idempotent completion} \label{sec:kar}

The \emph{idempotent completion}, or \emph{Karoubi envelope}, of a category $\ca$
is the category $\kar(\ca)$ whose objects of are the idempotents in $\ca$, while
a morphism $f\iN\kar(\ca)(p_1,p_2)$ between the idempotents $p_1\iN\eend_{\ca}(a_1)$
and $p_2\iN\eend_{\ca}(a_2)$ is a morphism $f\colon a_1\To a_2$ in $\ca$ satisfying 
$f\cir p_1 \eq f \eq p_2\cir f$. The identity on an idempotent 
$p\iN\eend_{\ca}(a)$ viewed as an object in $\kar(\ca)$ is $p$ itself, while 
$\id_a$ is not in $\eend_{\kar(\ca)}(p)$ unless $p \eq \id_a$. The Karoubi 
envelope $\kar(\ca)$ comes with a canonical fully faithful functor 
$K_\ca \colon \ca \Rarr~ \kar(\ca)$ that sends an object $a\iN\ca$ to the 
idempotent $\id_a$. The functor $K_\ca$ is universal among the functors from 
$\ca$ that have a chosen \emph{splitting} for every idempotent in their codomain. 
A crucial further property of $K_\ca$ is that it is \emph{cofinal} in the sense\,%
 \footnote{~We adopt the terminology of \cite{BOrc1} and \cite{LUri}. 
 This differs from the one in \cite{MAcl} and \cite{JOhnP}, where
 such a functor is instead called a \emph{final} functor.} 
that we can restrict functors from $\kar(\ca)$ to functors
from $\ca$ along $K_\ca$ without changing their colimits (compare
\cite[Lemma\,5.1.4.6]{LUri} for the $(\infty,1)$-version of this result).

In the present subsection we introduce idempotent completions of the cylinder categories,
to which we refer as \emph{Karoubified cylinder categories}, and define string-net
spaces whose boundary data are objects in these Karoubified cylinder categories. 
This is in line with the common practice in skein theoretic quantum topology to
study various kinds of completions. For instance, one considers the free
cocompletions of the skein categories -- the higher-dimensional analogues of 
cylinder categories that are defined for surfaces -- to recover the locally finitely 
presentable factorization homology \cite{cooke2} and provide geometric models for 
quantum character varieties when applied to quantum groups \cite{bebJ}.
In dimension (2,1), for the case that $\cb$ is the delooping $\cbc$ of a 
spherical fusion category $\cc$, the Karoubified cylinder category over the 
standard circle is canonically equivalent to the Drinfeld center $\mathcal Z(\cc)$
of the fusion category $\cc$ \cite{kirI24}, and the corresponding Karoubified 
string-net construction extends to a (3,2,1)-dimensional topological field theory 
that is equivalent to the Turaev-Viro state-sum model for $\cc$ (\cite{bart8};
see also \cite{goos} for an approach based on the presentation of 
$\mathcal B\mathrm{ord}_{3,2,1}^{\mathrm{or}}$ that was conjectured in \cite{bdsv}.) 
In the present paper, a major motivation for performing the idempotent completion 
of cylinder categories is that in the application of string-net models to the 
construction of correlators of rational conformal field theories, field objects 
are naturally realized as objects in the Karoubified cylinder categories
(for details, see Chapter 3 of \cite{fusY} and Section \ref{sec:SNC} below).

\begin{defn} \label{def:KarSN}
Let $\ell$ be a compact oriented (not necessarily closed) 1-manifold,
and let $(\cb,*_{\cb})$ a pointed strictly pivotal bicategory. 
 \begin{itemize}
 \item [{\rm(i)}]
The \emph{Karoubified cylinder category} for $(\cb,*_{\cb})$ over
$\ell$ is the Karoubi envelope 
  \be
  \cyl(\cb,*_{\cb},\ell) \coloneqq \kar(\ocyl(\cb,*_{\cb},\ell)) \,.
  \ee
Thus an object $\msB\iN\cyl(\cb,*_{\cb},\ell)$ is an idempotent
$\msB\colon\omsB\To\omsB$ in the ordinary cylinder category $\ocyl(\cb,*_{\cb},\ell)$;
we call it a \emph{thickened $(\cb,*_{\cb})$-boundary datum} on $\ell$. 
 \item [{\rm(ii)}]
The \emph{Karoubified string-net space} $\snb(\surf,\msB)$ for a compact 
oriented surface $\surf$ and a thickened $(\cb,*_{\cb})$-boundary datum 
$\msB\iN\cyl(\cb,*_{\cb},\partial\surf)$ is the subspace 
  \be
  \snb(\surf,\msB) \coloneqq \osnb(\surf,\omsB)^{\msB} \subset\osnb(\surf,\omsB)
  \label{eq:snb}
  \ee
consisting of string nets that are invariant under concatenation with the idempotent
$\msB$. 
 \end{itemize}
\end{defn}

Note that, since $\partial\surf$ has empty boundary, the cylinder category over 
$\partial\surf$ actually does not depend on the distinguished object, i.e.\
$\cyl(\cb,*_{\cb},\partial\surf) \eq \cyl(\cb,\partial\surf)$. Thus 
with Definition \ref{def:KarSN} we obtain a functor 
  \be
  \snb(\surf,-) \Colon \cyl(\cb,\partial\surf) \rarr~ \vct
  \ee
for every compact oriented surface $\surf$. Note that, just like $\ocyl(\cb,*_{\cb},-)$,
also the assignment $\cyl(\cb,*_{\cb},-)$ is functorial with respect to the
embeddings of 1-manifolds.


\Subsection{Factorization} \label{sec:fact}

Recall, e.g.\ from \Cite{Def.\,2.1}{fusY}, the description of the symmetric monoidal 
bicategory $\bbord$ of \emph{two-dimensional open-closed bordisms}: An object
$\alpha\iN\bbord$ is a finite disjoint union of copies of the standard interval 
$I \eq [0,1] \,{\subset}\, \mathbb R$ (oriented from $1$ to $0$) and the standard 
circle $S^1 \,{\subset}\, \complex$ (oriented counterclockwise).  A 1-morphism
between objects $\alpha$ and $\beta$ -- referred to as an open-closed bordism, or 
just \emph{bordism}, for short, and to be denoted by
$\surf\colon\alpha\ptO\beta$ -- consists of an underlying
compact oriented surface $\surf$, an in-going parametrization 
$\phi_{-}\colon\bar{\alpha} \,{\hookrightarrow}\, \partial\surf$ and an
out-going parametrization $\phi_{+}\colon\beta \,{\hookrightarrow}\, \partial\surf$.

The functoriality of $\ocyl(\cb,*_{\cb},-)$ and of $\cyl(\cb,*_{\cb},-)$ under 
embeddings of 1-manifolds that we established in Section \ref{sec:func-embed} implies

\begin{lem}
Let $(\cb,*_{\cb})$ be a pointed strictly pivotal bicategory. To any bordism 
$\surf\colon\alpha\ptO\beta$ there is naturally associated a $\fk$-linear profunctor
  \be
  \cosnb(\surf;-,\sim)\Colon \ocyl(\cb,*_{\cb},\alpha)\PtO\ocyl(\cb,*_{\cb},\beta) \,,
  \label{eq:SNoprofunctor}
  \ee
as well as a Karoubified version
  \be
  \csnb(\surf;-,\sim) \Colon\cyl(\cb,*_{\cb},\alpha) \PtO \cyl(\cb,*_{\cb},\beta)
  \ee
of $\cosnb(\surf;-,\sim)$.
\end{lem}

\begin{proof}
The profunctor $\cosnb(\surf;-,\sim)$ is given by the composite 
  \be
  \ocyl(\cb,*_{\cb},\alpha)^{\op} \Times \ocyl(\cb,*_{\cb},\beta)
  \xrightarrow{~(\phi_{-})_*^{}\,\sqcup\,(\phi_{+})_*^{}~}
  \ocyl(\cb,*_{\cb},\partial\surf) \xrightarrow{~\osnb(\surf,-)~} \vct 
  \ee
of the functor $\osnb(\surf,-)$ with the in- and out-going parametrizations of 
$\partial\surf$ (here we abbreviate $\ocyl(\cb,*_{\cb},\phi)$ by $\phi_*$). 
The Karoubified version is defined analogously.
\end{proof}

We are now ready to state the following result, which appears to be a variant 
of a folk theorem (see e.g.\ \cite{walk4,gujs,kiTha}):

\begin{thm} \label{thm:osnfct}
Let $\surf\colon\alpha\,{\sqcup}\,\beta \ptO \beta\,{\sqcup}\,\gamma$
be a bordism, with $\alpha,\beta,\gamma\iN\bbord$. Then the family 
  \be
  \big\{ s_{-,\msb_{0},\sim}^{\surf}\colon \cosnb(\surf;-,\msb_{0},\msb_{0},\sim)
  \,{\Longrightarrow}\, \cosnb(\cup_{\beta}\surf;-,\sim) \big\}
  _{\msb_0^{}\in\ocyl(\cb,*_{\cb},\beta)}
  \label{eq:cosndinat}
  \ee
of natural transformations whose members are given by the sewing 
of string nets, 
is dinatural and exhibits the functor $\cosnb(\cup_{\beta}\surf;-,\sim)$ as the coend 
  \be
  \int^{\msb\in\ocyl(\cb,*_{\cb},\beta)}\! \cosnb(\surf;-,\msb,\msb,\sim) \Colon
  \ocyl(\cb,*_{\cb},\alpha)\ptO\ocyl(\cb,*_{\cb},\gamma) \,.
  \ee
\end{thm}

\begin{proof}
We need to show that the family \ref{eq:cosndinat} is dinatural, and that it is
universal among dinatural families of the same type. Dinaturality 
holds by the fact that sewing a morphism $F$ in the cylinder category 
$\ocyl(\cb,*_{\cb},\beta)$ to the first slot (along the appropriate 1-manifold) 
of the profunctor and then sewing along $\beta$ yields the same result as sewing 
$F$ to the the second slot and then doing the final sewing.
To show universality, consider a dinatural family 
  \be
  \big\{ g_{\msb_{0}} \colon \cosnb(\surf;\msa_{0},\msb_{0},\msb_{0},\msc_{0})
  \rarr~ V \big\}_{\msb_{0}\in\ocyl(\cb,*_{\cb},\beta)}
  \ee
for arbitrary boundary data $\msa_{0}\iN\ocyl(\cb,*_{\cb},\alpha)$
and $\msc_{0}\iN\ocyl(\cb,*_{\cb},\gamma)$ and an arbitrary $\fk$-vector space
$V\iN\vct$. Define the linear map
  \be
  \begin{aligned}
  g\Colon\cosnb(\cup_{\beta}\surf;\msa_{0},\msc_{0}) & \longrightarrow V \,,
  \\
  [\grph] & \longmapsto[\mathrm{cut}(\grph)]
  \longmapsto g_{\msb_{\grph}}([\mathrm{cut}(\grph)]) \,,
  \label{eq:gcutgrph}
  \end{aligned}
  \ee
where $\grph$ is a representative graph that intersects $\beta$ only at edges 
and any such intersection is transversal (such a representative can be chosen
without loss of generality), $\mathrm{cut}(\grph)$ is the fully colored graph
on $\surf$ obtained by cutting the representative graph $\grph$ along $\beta$,
and $\msb_{\grph}\iN\ocyl(\cb,*_{\cb},\beta)$ is the boundary datum arising 
from the cut. The scenario is illustrated by the following schematic example:
  \be
  \includegraphics[valign=c]{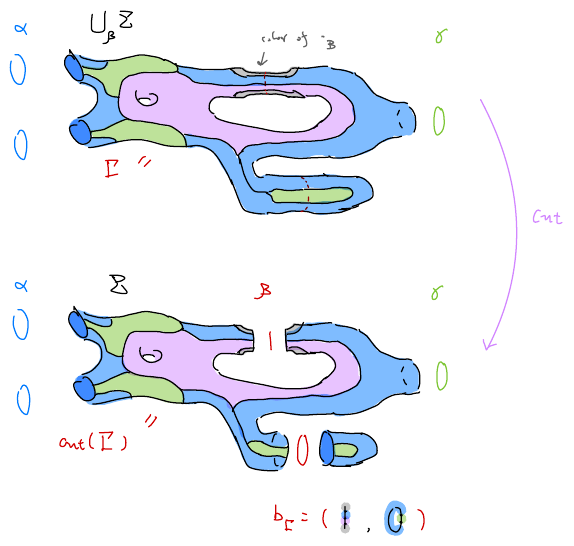}
  \ee
The linear map \eqref{eq:gcutgrph} is well-defined because the family 
$\{g_{\msb_{0}}\}_{\msb_{0}\in\ocyl(\cb,*_{\cb},\beta)}$ is dinatural by assumption 
and because all relevant isotopies as well as the generating local relations 
provided by the graphical calculus on disks are contained within embedded disks,
which implies that a different choice of representative for the string-net 
$[\grph]\iN\cosnb(\cup_{\beta}\surf;\msa_{0},\msc_{0})$
only differs by the action of some morphism in the cylinder category
$\ocyl(\cb,*_{\cb},\beta)$.
 \\
Uniqueness of the map \eqref{eq:gcutgrph} is guaranteed  by design. Hence the
dinatural family at each component does exhibit a coend at the level of vector
spaces and linear maps. By a standard argument the so defined component-wise 
coends automatically combine to a coend at the level of linear functors
and linear natural transformations.
\end{proof}

An analogous statement holds for the Karoubified string-net functors. This
is achieved by the following observation. Recall that the 
\emph{twisted arrow category} $\tw(\text{\ensuremath{\ca}})$ associated with
a category $\ca$ is the category whose objects are the morphisms of $\ca$ and
for which a morphism from $f\colon a\To b$ to $f'\colon a'\To b'$ is a pair 
$(g,h)\iN\ca^{\op}(a,a')\Times\ca(b,b')$ such that the square 
  \be
  \begin{tikzcd}
  a & b \\
  {a'} & {b'}
  \arrow["f", from=1-1, to=1-2] \arrow["g", from=2-1, to=1-1]
  \arrow["h", from=1-2, to=2-2] \arrow["{f'}"', from=2-1, to=2-2] 
  \end{tikzcd}
  \ee
commutes. The category $\tw(\ca)$ comes with a canonical projection functor 
  \be
  \pi^{}_{\ca}\Colon \tw(\ca) \to \ca^{\op}\Times\ca \,,
  \ee
which sends an object $f\colon a\To b$ to $(a,b)\iN\ca^{\op}\Times\ca$
and keeps the morphisms as they are. The relevance of this construction
to us is that the coend of a functor $F\colon \ca^{\op}\Times\ca \To \vct$
can be identified with a \emph{colimit} over the twisted arrow category,
according to
  \be
  \int^{a\in\ca}\! F(a,a) = \mathrm{colim}\big( \tw(\ca^{\op})^{\op}
  \xrightarrow{~\pi_{\ca^{\op}}^{\op}~} \ca^{\op}\Times\ca
  \xrightarrow{~~F~~} \vct \big) \,,
  \label{eq:coendcolimTW}
  \ee
see e.g.\ \cite[Sect.\,1.2]{LOreg}.

\begin{lem} \label{lem:coendKar}
For $\ca$ a category, let $F\colon\kar(\ca)^{\op}\Times\kar(\ca)\To\vct$ be
a functor whose coend exists. Then 
  \be
  \int^{a\in\ca}\! F|_{\ca}^{}(a,a) = \int^{A\in\kar(\ca)}\! F(A,A) \,,
  \label{eq:coendKar}
  \ee
where $F|_{\ca}\colon\ca^{\op}\Times\ca \Rarr~ \kar(\ca)^{\op}\Times\kar(\ca)
\Rarr{~F\,} \vct$ is the restriction of $F$ along the canonical embedding.
\end{lem}

\begin{proof}
Consider the functor 
  \be
  \begin{aligned}
  G\Colon \tw(\ca^{\op})^{\op} & \,\longrightarrow\, \tw(\kar(\ca)^{\op})^{\op} \,,
  \\
  (a\xleftarrow{~f~}b) & \,\longmapsto\, (\id_{a}\xleftarrow{~f~}\id_{b}) \,.
  \end{aligned}
  \ee
Owing to the commutativity of the square 
  \be
  \begin{tikzcd}[column sep=3.4em,row sep=2.8em]
  {\mathrm{Tw}(\mathcal{A}^\mathrm{op})^\mathrm{op}}
  & {\mathrm{Tw}(\mathrm{Kar}(\mathcal{A})^\mathrm{op})^\mathrm{op}}
  \\
  {\mathcal{A}^\mathrm{op}\Times\mathcal{A}}
  & {\mathrm{Kar}(\mathcal{A})^\mathrm{op}\Times\mathrm{Kar}(\mathcal{A})}
  \arrow["G", from=1-1, to=1-2]
  \arrow["{\pi_{\mathrm{Kar}(\mathcal{A})^{\mathrm{op}}}^{\mathrm{op}}}", from=1-2, to=2-2]
  \arrow[from=2-1, to=2-2]
  \arrow["{\pi_{\mathcal{A}^{\mathrm{op}}}^{\mathrm{op}}}"', from=1-1, to=2-1] 
  \end{tikzcd}
  \ee
it suffices to show that $G$ is cofinal. Indeed, if that is the case, we have 
  \begin{align}
  \int^{a\in\ca}\! F^{}|_{\ca}(a,a) 
  & = \mathrm{colim}\big( \tw(\ca^{\op})^{\op} \rarr{\pi_{\ca^{\op}}^{\op}}
  \ca^{\op}\Times\ca \rarr{~~} \kar(\ca)^{\op}\Times\kar(\ca) \rarr{\,F\,} \vct \big)
  \nonumber \\
  & = \mathrm{colim}\big( \tw(\ca^{\op})^{\op} \rarr{\,G\,} \tw(\kar(\ca)^{\op})^{\op}
  \nonumber \\
  & \hspace{3.4cm} \rarr{\pi_{\kar(\ca)^{\op}}^{\op}}
  \kar(\ca)^{\op}\Times\kar(\ca) \rarr{\,F\,} \vct \big)
  \nonumber \\
  & =\mathrm{colim}\big( \tw(\kar(\ca)^{\op})^{\op} \rarr{\pi_{\kar(\ca)^{\op}}^{\op}}
  \kar(\ca)^{\op}\Times\kar(\ca) \rarr{\,F\,} \vct \big)
  \nonumber \\[3pt]
  & =\int^{A\in\kar(\ca)}\! F(A,A) \,,
  \end{align}
where cofinality of $G$ is used in the third equality.
 \\[2pt]
To show that the functor $G\colon\tw(\ca^{\op})^{\op}\To\tw(\kar(\ca)^{\op})^{\op}$
is cofinal, we use the fact that this is the case iff for every object 
$(p\Rarr{~g~}q)\iN\tw(\kar(\ca)^{\op})^{\op}$,
with $p\iN\eend_{\ca}(a)$ and $q\iN\eend_{\ca}(b)$ idempotents,
the comma category $g\,{\downarrow}\, G$ is connected, i.e.\ it is non-empty and for
each pair of objects in $g\,{\downarrow}\, G$ there exists a zigzag connecting 
them. To see that the latter conditions are satisfied, first note that the square 
  \be
  \begin{tikzcd}
  p & q 
  \\
  {\mathrm{id}_a} & {\mathrm{id}_b} 
  \arrow["g"', from=1-2, to=1-1] \arrow["p", from=2-1, to=1-1]
  \arrow["q", from=1-2, to=2-2] \arrow["g", from=2-2, to=2-1] 
  \end{tikzcd}
  \ee
commutes, due to the defining condition for $g$ to be a morphism
of type $q\To p$. This commutative square provides us with an object
$(g\To Gg)\iN g \,{\downarrow}\, G$, hence $g\,{\downarrow}\, G$ is non-empty.
Next assume that we have a pair of objects $g\To Gf$ and $g\To Gh$
in the comma category $g\,{\downarrow}\, G$, given by the two commutative squares
  \be
  \begin{tikzcd}
  p & q & \raisebox{-2.1em}{and} & p & q 
  \\[-13pt]
  {\mathrm{id}_c} & {\mathrm{id}_d} & {} & {\mathrm{id}_{c'}} & {\mathrm{id}_{d'}}
  \arrow["g"', from=1-2, to=1-1] \arrow["r", from=2-1, to=1-1]
  \arrow["s", from=1-2, to=2-2]  \arrow["f", from=2-2, to=2-1]
  \arrow["g"', from=1-5, to=1-4] \arrow["{r'}", from=2-4, to=1-4]
  \arrow["{s'}", from=1-5, to=2-5] \arrow["h", from=2-5, to=2-4]
  \end{tikzcd}
  \ee
Observe that the diagram 
  \be
  \begin{tikzcd}
  & {\mathrm{id}_c} && {\mathrm{id}_d}
  \\
  && {\mathrm{id}_a} && {\mathrm{id}_b}
  \\
  p && q 
  \\
  & {\mathrm{id}_{c'}} && {\mathrm{id}_{d'}}
  \arrow["q"{description}, from=3-3, to=2-5]
  \arrow["{r'}"{description}, from=4-2, to=3-1]
  \arrow["{s'}"{description}, from=3-3, to=4-4]
  \arrow["h"{description}, color={rgb,255:red,92;green,92;blue,214}, from=4-4, to=4-2]
  \arrow["g"{description}, color={rgb,255:red,92;green,92;blue,214}, from=2-5, to=2-3]
  \arrow["s"{description}, color={rgb,255:red,92;green,92;blue,214}, from=2-5, to=1-4]
  \arrow["{s'}"{description, pos=0.7}, color={rgb,255:red,92;green,92;blue,214}, from=2-5, to=4-4]
  \arrow["{r'}"{description, pos=0.3}, color={rgb,255:red,92;green,92;blue,214}, from=4-2, to=2-3]
  \arrow["r"{description}, color={rgb,255:red,92;green,92;blue,214}, from=1-2, to=2-3]
  \arrow["r"{description, pos=0.3}, from=1-2, to=3-1]
  \arrow[from=uu, crossing over, "s"{description, pos=0.7}, from=3-3, to=1-4]
  \arrow["f"{description}, color={rgb,255:red,92;green,92;blue,214}, from=1-4, to=1-2]
  \arrow[from=uu, crossing over, "g"{description}, from=3-3, to=3-1]
  \arrow["p"{description}, from=2-3, to=3-1] 
  \end{tikzcd}
  \ee
commutes and hence the two squares that do not involve the objects $p$ and $q$ 
(and in the color version are drawn in violet) give rise to a span 
  \be
  (g\To Gf) \longleftarrow (g\To Gg) \longrightarrow (g\To Gh)
  \ee
in the comma category $g\,{\downarrow}\, G$.
It follows that for every $g\iN\tw(\kar(\ca)^{\op})^{\op}$ the comma category
$g\,{\downarrow}\, G$ is connected. Hence the functor
$G\colon\tw(\ca^{\op})^{\op}\to\tw(\kar(\ca)^{\op})^{\op}$ is cofinal, as claimed.
\end{proof}

Combining Theorem \ref{thm:osnfct} with Lemma \ref{lem:coendKar} we arrive at

\begin{cor} \label{cor:snfct}
Let $\surf\colon\alpha\,{\sqcup}\,\beta \ptO \beta\,{\sqcup}\,\gamma$
be a bordism, with $\alpha,\beta,\gamma\iN\bbord$. Then the family 
  \be
  \big\{ \widehat{s}_{-,\msb_{0},\sim}^{\surf}\colon
  \csnb(\surf;-,\msB_{0},\msB_{0},\sim)\,{\Rightarrow}\,
  \csnb(\cup_{\beta}\surf;-,\sim) \big\}_{\msB_{0}\in\cyl(\cb,*_{\cb},\beta)}
  \label{eq:cosndinat-1}
  \ee
whose members are given by the sewing of string nets, is dinatural and exhibits
the functor $\csnb(\cup_{\beta}\surf;-,\sim)$ as the coend 
  \be
  \int^{\msB\in\cyl(\cb,*_{\cb},\beta)}\! \csnb(\surf;-,\msB,\msB,\sim) \Colon
  \cyl(\cb,*_{\cb},\alpha) \ptO \cyl(\cb,*_{\cb},\gamma) \,.
  \ee
\end{cor}

\begin{proof}
By direct calculation we have
  \be
  \begin{aligned}
  \csnb(\cup_{\beta}\surf;,\msA_{0},\msC_{0}) 
  & =\, \cosnb(\cup_{\beta}\surf;,\msA_{0}^{\circ},\msC_{0}^{\circ})^{(\msA_{0},\msC_{0})}
  \\[3pt]
  & = \int^{\msb\in\ocyl(\cb,*_{\cb},\beta)}\!
  \cosnb(\surf;\msA_{0}^{\circ},\msb,\msb,\msC_{0}^{\circ})^{(\msA_{0},\msC_{0})}
  \\
  & = \int^{\msb\in\ocyl(\cb,*_{\cb},\beta)}\!
  \csnb(\surf;\msA_{0},\id_{\msb},\id_{\msb},\msC_{0})\\
  & \!\!\! \equ{eq:coendKar} \int^{\msB\in\cyl(\cb,*_{\cb},\beta)}\!
  \csnb(\surf;\msA_{0},\msB,\msB,\msC_{0})
  \end{aligned}
  \ee
for every $\msA_0\iN\cyl(\cb,*_{\cb},\alpha)$ and $\msC_0\iN\cyl(\cb,*_{\cb},\gamma)$.
\end{proof}


\Subsection{Open-closed modular functors from bicategorical string nets}
\label{sec:open-closed}

We are now going to show that bicategorical string nets provide us with modular
functors. With the application to conformal field theory in mind (see Section
\ref{sec:SNC}, and in particular Remark \ref{rem:theend}), we are interested
in \emph{open-closed} modular functors, for which the 1-manifolds that are objects 
of the domain bicategory may have a non-empty boundary. In Section \ref{sec:oCylI}
we have already prepared the ground for dealing with this case by defining cylinder 
categories not only over circles, but also over intervals.

\begin{defn}
An \emph{open-closed} modular functor is a symmetric monoidal pseudofunctor 
  \be
  \bbord \rarr~ \bprof
  \ee
from the symmetric monoidal bicategory $\bbord$ of two-dimensional open-closed 
bordisms to the symmetric monoidal bicategory of $\bprof$ of $\fk$-linear profunctors.
\end{defn}

The objects of the bicategory $\bbord$ are finite disjoint unions of copies of the 
standard circle $S^1 \eq \{z \,|\, |z|\eq 1\} \,{\subset}\, \complex$ and of the 
standard interval $I \eq [0,1] \,{\subset}\, \mathbb R \,{\subset}\, \complex$, the
1-morphisms are bordisms between such 1-manifolds, and the 2-morphisms are isotopy 
classes of diffeomorphisms; for more details on $\bbord$ see e.g.\ Definition 2.1 of 
\cite{fusY}. The bicategory $\bprof$ is defined as follows:\,%
 \footnote{~The definition used here deviates somewhat from the one in \cite{fusY}.}
 \Itemize
 \item 
Objects of $\bprof$ are small categories enriched in the cocomplete category
$\vct$ of (not necessarily finite-dimensional) $\fk$-vector spaces.
 \item 
For objects $A$ and $B$, a 1-morphism $P\colon A\PtO B$ is a
\emph{$\fk$-linear profunctor} from $A$ to $B$, that is, a $\fk$-linear
functor $P\colon A^{\op}\Times B\Rarr{~~}\vct$.
 \item
For 1-morphisms $P,\,Q\colon A\PtO B$, a 2-morphism 
 $\varphi\colon P\,{\xRightarrow{~~}}\, Q$
is a natural transformation of the underlying functors.
 \item
The composite of a composable pair 
 $\!\!\begin{tikzcd}[column sep=1.5em] A\! & \!B\! & \!C
 \arrow["Q", "\shortmid"{marking}, from=1-2, to=1-3]
 \arrow["P", "\shortmid"{marking}, from=1-1, to=1-2] \end{tikzcd}\!\!$
of 1-morphisms is the coend\,%
  \be
  P\cdo Q \,\coloneqq\, \int^{b\in B}\! P(-,b)\otk Q(b,\sim) \,\colon~~ A\PtO C \,.
  \ee
The horizontal composition of 2-morphisms is induced by the composition of 1-morphisms.
(The respective coends exist because all the domain categories are small and the target 
categories are cocomplete, see e.g. \cite[Prop.\,4.5.3]{RIch}.)
 \item 
Vertical composition is given by the vertical composition of natural transformations.
 \item 
The monoidal structure given by the Cartesian product of $\fk$-linear categories,\,%
 \footnote{~The appropriate notion of Cartesian product for $\fk$-linear categories
 is the one provided by the framework of enriched categories, i.e.\ for small
 $\fk$-linear categories $A$ and $B$, $A\Times B$ is the small $\fk$-linear
 category whose objects are ordered pairs and whose morphism spaces are obtained as
 tensor products over $\fk$, i.e.\ $(A\Times B)(a_{1}{\times} b_{1},a_{2}{\times} b_{2})
 \coloneqq A(a_{1},a_{2})\otk B(b_{1},b_{2})$ (see e.g.\ \cite[Sect.\,1.4]{KEll}, 
 where this product is instead referred to as \emph{tensor product}).}
with the obvious symmetric braiding.
\end{itemize}

A specific implication of Theorem \ref{thm:osnfct} is that for any pair of 
composable bordisms $\surf\colon\alpha\ptO\beta$ and $\surf'\colon\beta\ptO\gamma$
the dinatural family of sewing maps exhibits the structure of a coend on
  \be
  \begin{aligned}
  \cosnb(\surf \,{\cup_{\beta}}\, \surf';-,\sim)
  & =\int^{\msb\in\ocyl(\cb,*_{\cb},\beta)}\!
  \cosnb(\surf\,{\sqcup}\,\surf';-,\msb,\msb,\sim)
  \\
  & =\int^{\msb\in\ocyl(\cb,*_{\cb},\beta)}\!
  \cosnb(\surf;-,\msb)\otk\cosnb(\surf';\msb,\sim) \,.
  \end{aligned}
  \ee
An analogous statement follows from Corollary \ref{cor:snfct} for the
Karoubified string-net spaces.
This immediately implies

\begin{thm} \label{thm:thm}
Let $(\cb,*_{\cb})$ be a pointed strictly pivotal bicategory. Then the assignments
  \be
  \alpha\mapsto\ocyl(\cb,*_{\cb},\alpha) \qquad \text{and} \qquad 
  \surf\mapsto\cosnb(\surf;-,\sim)
  \ee
extend to an \emph{open-closed modular functor}, i.e.\ a symmetric monoidal 
pseudofunctor 
  \be
  \cosnb\Colon \bbord\rarr~\bprof
  \ee
from the symmetric monoidal bicategory of open-closed bordisms to the
symmetric monoidal bicategory of $\mathbbm{k}$-linear profunctors.
Similarly, the Karoubified cylinder categories and string-net spaces give rise to 
another open-closed modular functor 
  \be
  \csnb\Colon \bbord\rarr~\bprof \,.
  \ee
\end{thm}

Now let $\cc$ be a modular fusion category that encodes the chiral data for
a rational conformal field theory. The results of Section \ref{sec:kar} then amount
to the statement that the open-closed modular functor $\csnc \eq \csnbc$ models 
the conformal blocks of the conformal field theory. In other words, by setting
  \be
  \blc(\alpha)\coloneqq\cyl(\cc,\alpha)
  \ee
for every $\alpha\iN\bbord$, and 
  \be
  \blc(\surf;-,\sim)\coloneqq\csnc(\surf;-,\sim)
  \ee
for every bordism $\surf$, we obtain 
an open-closed modular functor that when restricted to the closed sector
is isomorphic to the modular functor that is provided by the Turaev-Viro
state sum construction and that furnishes canonical equivalences 
$\blc(I) \,{\simeq}\, \cc$ and $\blc(S^1) \,{\simeq}\, \mathcal Z(\cc)$.
(For more details, see Chapter 3.3 of \cite{fusY}.)

\begin{rem}
For $\cc$ a modular fusion category, the categories $\blc(I) \eq \cyl(\cc,I) 
\,{\simeq}\, \cc$ and $\blc(S^{1}) \eq \cyl(\cc,S^{1}) \,{\simeq}\, \czc$ are 
finite and semisimple. As a consequence, for any bordism $\surf\colon\alpha\ptO\beta$
the functor $\blc(\surf;-,\sim) \,{=}\, \csnc(\surf;-,\sim)$ is exact in each of its 
variables. It can therefore be substituted by an exact functor $\widehat{\bl}_{\cc}
(\surf;-\,{\boxtimes}\,\sim) \colon \blc(\alpha)^{\op}\boti\blc(\beta)\To\vct$. 
Thus the open-closed modular functor $\blc$ factors through the forgetful functor 
$\mathcal{P}\mathrm{rof}_{\fk}^{\mathcal{L}\mathrm{ex}} \Rarr~ \bprof$,
where $\mathcal{P}\mathrm{rof}_{\fk}^{\mathcal{L}\mathrm{ex}}$ is the symmetric
monoidal bicategory of finite $\fk$-linear abelian categories, left exact profunctors
and natural transformations, the horizontal composition of which is given by left
exact coends (in the sense of \Cite{Sect.\,3.2}{fuSc23}), and with the monoidal 
product being the Deligne product.
We thus also have a modular functor in the sense of Definition 2.1 of \cite{fusY}, 
in which $\mathcal{P}\mathrm{rof}_{\fk}^{\mathcal{L}\mathrm{ex}}$ is taken as the
target category. (The motivation for making this choice in \cite{fusY} is that 
it suits a potential extension to non-semisimple modular tensor categories.
In the case of a modular fusion category $\cc$ considered here, the profunctors are
in fact both left and right exact.)
\end{rem}

We finally mention that not only $\csnc$, but also the open-closed modular functor
  \be
  \cosnfrc \Colon \bbord \rarr~ \bprof
  \ee
that is obtained by taking $(\cfrc,\tu)$ as the decorating pointed pivotal
bicategory is relevant to rational conformal field theory. This will
be explained in Section \ref{sec:SNC}.


\section{Application: Universal correlators in RCFT} \label{sec:SNC}

The string-net construction of correlators for a two-dimensional rational
conformal field theory (RCFT) that was developed in \cite{fusY} provides a natural 
motivation for the bicategorical string-net construction that we presented in
Section \ref{sec:SNbicat}: Given a modular fusion category $\cc$ describing the chiral
data for the RCFT, a topological world sheet $\ws$ with physical boundaries and
topological defect lines -- a \emph{world sheet}, for short --
can be regarded as a $\cfrc$-colored string diagram
$\cws$ on its underlying surface $\sws$. In \cite{fusY}, we assigned to each 
world sheet $\ws$ an element $\corc(\ws)\iN\snc(\sws,\df_{\partial\sws}(\msb_{\ws}))$
in a (Karoubified) string-net space colored by the modular fusion category $\cc$, where 
$\msb_{\ws}$ is the boundary datum of the $\cfrc$-colored string diagram $\cws$, while
$\df_{\partial\sws} \colon \ocyl(\cfrc,\tu,\partial\sws) \Rarr~ \cyl(\cc,\partial\sws)$
is a canonical functor sending an $\cfrc$-colored boundary datum to an idempotent
in the $\cc$-colored cylinder category over the boundary $\partial\sws$ of $\sws$.
The  $\cc$-colored string-net space $\snc(\sws,\df_{\partial\sws}(\msb_{\ws}))$ 
can be regarded as the \emph{space of conformal blocks} for the world sheet 
$\ws$ (which only depends on the underlying surface $\sws$).
As shown in \cite{fusY}, the specific element $\corc(\ws)$ in this space of 
conformal blocks is invariant under the action of a subgroup $\mcg(\ws)$ of the 
mapping class group $\mcg(\sws)$ that is determined by the world sheet $\ws$, and 
the assignment $\ws\,{\mapsto}\,\corc(\ws)$ is compatible with sewing. Accordingly,
$\corc(\ws)$ is naturally interpreted as the \emph{correlator} for the world sheet 
$\ws$. 

An evident question to ask is whether the correlator $\corc(\ws)$ only depends on the 
equivalence class $[\cws]$ in the $\cfrc$-colored string-net space 
$\osnfrc(\sws,\msb_{\ws})$, i.e.\ whether two world sheets with the same underlying 
surface that are related by the local graphical calculus for the pivotal bicategory 
$\cfrc$ have the same correlator. In \cite{fusY} an affirmative answer to this 
question was given without proof. Making use of the tools developed in the 
previous section, below we give a complete proof of our claim.


\Subsection{The string-net construction of RCFT correlators}

Let us briefly review the string-net construction of RCFT correlators. For 
details we refer to \cite{fusY} and \cite{yangYa3}.

In the categorical approach to conformal field theory, one first develops the
representation theory of the pertinent chiral vertex operator algebra 
$\mathfrak{V}$. If the chiral symmetries are sufficiently nice, the representation 
category of $\mathfrak{V}$ has the structure of a \emph{modular fusion category};
conformal field theories with such chiral symmetries are called \emph{rational}
CFTs, or RCFTs, for short. The modular fusion category $\cc$ of chiral data 
controls the monodromy behavior of the conformal blocks and gives rise to an
open-closed modular functor
  \be
  \blc\Colon \bbord \rarr~ \bprof \,.
  \ee
Conjecturally \cite{BAki}, this modular functor is equivalent to the modular 
functor furnished by the Reshetikhin-Turaev surgery topological field theory
for the Drinfeld center $\czc$ or, equivalently, \cite{TUvi,bals2}, by the
Turaev-Viro state-sum TFT for $\cc$.
The following result \cite{kirI24} allows us to adopt the string-net 
modular functor $\csnc\colon \bbord \Rarr~ \bprof$
as an alternative realization of the modular functor $\blc$:

\begin{thm} \label{thm:SN=TV}
Let $\cc$ be a spherical fusion category. The Karoubified string-net modular functor
$\snc$ is equivalent to the Turaev-Viro modular functor $\mathrm{TV}_{\cc}$.
\end{thm}

\begin{rem}
In fact, the modular functor $\snc$ canonically extends to a $3$-$2$-$1$ 
topological field theory (in the sense of \cite{bdsv3}) that is 
isomorphic to the once-extended Turaev-Viro TFT for $\cc$
\cite{goos}. For these topological field theories to exist it is crucial
that the modular tensor category $\cc$ is a fusion category, i.e.\ semisimple.
The absence of any semisimplicity condition in the constructions in the previous 
sections may be taken as an indication that, in contrast, Theorem \ref{thm:SN=TV} 
admits a generalization to non-semisimple modular tensor categories.
\end{rem}

A world sheet $\ws$ is a stratified surface. The 1-cells in the interior of $\ws$ are
topological \emph{defect lines}, while the 1-cells on the geometric boundary
$\partial\ws$ are either \emph{sewing boundaries}, along which world sheets can be
sewn, or \emph{physical boundaries}. Further, $\ws$ comes with a decoration by
the pointed strictly pivotal bicategory $\cfrc$ of simple special symmetric 
Frobenius algebras (as described in Examples \ref{exa:cfrc} and \ref{exa:Frc*}) 
that is associated to the modular fusion category $\cc$: each 2-cell is colored 
by a \emph{phase} of the RCFT, which is an object of $\cfrc$; defect lines and 
physical boundaries are colored by 1-morphisms in $\cfrc$, and junctions of defect 
lines as well as junctions of defect lines and physical boundaries by 2-morphisms. 
We illustrate this decoration by the following example:

\begin{example} \label{exa:ws1}
The world sheet 
  \be
  \scalebox{1.4}{\input{WS1.tikz}}
  \label{eq:ws1}
  \ee
is decorated as follows: Its two 2-cells are colored with phases $A,B\iN\cfrc$, which 
we indicate by different shadings of the 2-cells (green and blue, respectively, in the
color version); its six line defects are colored by \emph{defect conditions} 
$X_1,X_2,X_3,X_5\iN A\bimodc B$, $X_4\iN A\bimodc A$ and $X_6\iN B\bimodc A$;
its two physical boundary segments are colored with \emph{boundary conditions} 
$M_{1}\iN\rmodc A \eq \tu\bimodc A$ and $M_{2}\iN\rmodc B \eq \tu\bimodc B$; finally, 
its three point defects are colored by $\varphi_{1}$, $\varphi_{2}$ and $\varphi_{3}$,
which (upon making the auxiliary choice of a polarization for each point defect
viewed as an internal vertex) are bimodule morphisms of appropriate type.
\end{example}

Owing to the presence of physical boundary segments, world sheets like the one
shown in the picture \eqref{eq:ws1} are not quite $\cfrc$-colored graphs on their
underlying surface in the sense defined at the beginning of Section \ref{sec:BSNspaces}:
These segments are colored edges contained in the boundary of the surface, and an
$\cfrc$-colored graph cannot have such edges. 
To remedy this problem, we replace the world sheet $\ws$ by another stratified surface
$\cws$ that does constitute an $\cfrc$-colored graph. In performing this replacement 
we are guided by the requirement to keep the correct space of conformal blocks, which 
means that the topology of the world sheet should not be altered, together with 
the observation that a physical boundary segment adjacent
to a 2-cell with phase label $A$ can equivalently be viewed as a line defect between
the phase $A$ and the \emph{trivial phase} -- the phase labeled with the trivial
Frobenius algebra $\tu\iN\cfrc$. This leads us to the following procedure for
turning a world sheet $\ws$ into an $\cfrc$-colored string diagram on a
surface $\cws$ that is homeomorphic to $\ws$ (for more details see
\Cite{Sect.\,3.4}{fusY} and \Cite{Sect.\,2.4}{yangYa3}):
If a geometric boundary circle $c$ of $\ws$ contains both physical boundary segments 
and sewing boundaries, we attach to each connected component of the union of
physical boundaries in $c$ a 2-cell that is homeomorphic to a disk and is
\emph{transparent} in the sense that it is colored by $\tu$, as illustrated in the
following picture: 
  \be
  \input{MB0.tikz}\qquad\longmapsto\qquad\input{MB1.tikz}
  \ee
If, on the other hand, $c$ is a \emph{pure physical boundary}, i.e.\
does not contain any sewing boundaries, then we attach to it a
transparent 2-cell that is homeomorphic to a cylinder, for instance
  \be
  \input{MB3.tikz}\qquad\longmapsto\qquad\input{MB4.tikz}.
  \ee
For any world sheet $\ws$ this prescription gives us a new stratified surface
$\cws$, to be referred to as the \emph{complemented world sheet}
of $\ws$. We denote the underlying surface of $\cws$, obtained by forgetting all
strata along with their labels, by $\sws$ and call it the \emph{ambient surface} 
of $\ws$. Note that there is a canonical embedding $\ws\,{\hookrightarrow}\,\sws$.
Strictly speaking, neither the complemented world sheet $\cws$ nor
the ambient surface $\sws$ are uniquely determined by our prescription.
However, given any two such constructions, there exists a unique homeomorphism
(up to isotopies) between the ambient surfaces that is compatible with the embeddings.

\begin{example}
For the world sheet of Example \ref{exa:ws1} the complemented world sheet looks as follows:
  \be
  \scalebox{1.4}{\input{WS2.tikz}}
  \ee
\end{example}

The complemented world sheet $\cws$ of a world sheet $\ws$ can be
regarded as an $\cfrc$-colored graph on the ambient surface $\sws$.
Denote by $\msb_{\ws}\iN\ocyl(\cfrc,\partial\surf)$ the boundary
datum of $\cws$. The \emph{field functor} \Cite{Sect.\,5.2}{yangYa3}
$\df_{\partial\sws}\colon\ocyl(\cfrc,\tu,\partial\sws) \Rarr~ \cyl(\cc,\partial\sws)$
for the 1-manifold $\partial\surf$ sends $\msb_{\ws}$ to an object
$\df_{\partial\sws}(\msb_{\ws})$ in the Karoubified cylinder category
$\cyl(\cc,\partial\sws)$. We define the \emph{space of conformal blocks
for the world sheet $\ws$} to be the Karoubified $\cc$-colored
string-net space for the pair $(\sws,\df_{\partial\sws}(\msb_{\ws}))$:
  \be
  \blc(\ws)\coloneqq\snc(\sws,\df_{\partial\sws}(\msb_{\ws})) \,.
  \ee

Now recall the rigid separable Frobenius functor $\cu\colon\cfrc \Rarr~ \cbc$
from Example \ref{exa:UFrCtoBC}. The correlator 
  \be
  \corc(\ws)\in\snc(\sws,\df_{\partial\sws}(\msb_{\ws}))
  \ee
for the world sheet $\ws$ is defined as the $\cc$-colored string-net that is 
represented by a $\cc$-colored graph on $\sws$ which is obtained by the following
two-step procedure: First we perform $\cu$-conjugation, as given in
Definition \ref{def:Fconj}), on $\cws$, that is, relabel the edges of $\cws$
with the underlying $\cc$-objects of the bimodules and relabel
the internal vertices with the $\cu$-conjugates of the bimodule morphisms;
the 2-cells are relabeled with the unique object $*\iN\cbc$ of the delooping
of $\cc$. This first step results in a $\cc$-colored graph $\grph_{\ws}$
on $\sws$, which we call the \emph{partial defect network for $\ws$}. In the second
step we add a \emph{full Frobenius graph} \Cite{Def.\,3.20}{fusY}
$\grph_{\vartheta}$ to each 2-cell of $\cws$. The correlator is then
the string-net equivalence class of the so obtained graph \Cite{Def.\,3.26}{fusY}:
  \be
  \corc(\ws)\coloneqq \big[ \grph_{\ws}\cup\!\bigcup_{\vartheta\in\cws}\!
  \grph_{\vartheta} \big] \,\in \snc(\sws,\df_{\partial\surf}(\msb_{\ws})) \,.
  \ee
The properties of special Frobenius algebras and their bimodules ensure that the string-net
correlator $\corc(\ws)$ is well defined and invariant under the action of $\mcg(\ws)$.
As shown in Theorem 3.28 of \cite{fusY}, the assignment $\ws\,{\longmapsto}\,\corc(\ws)$
gives a consistent system of correlators.

\begin{example}
As an illustration of the prescription, consider again the world sheet 
\eqref{eq:ws1}). A representative $\cc$-colored graph for the string net that 
gives the correlator $\corc(\ws)$ is shown in the following picture:
  \be
  \scalebox{1.4}{\input{WS3.tikz}}
  \ee
Here the unlabeled trivalent graphs (drawn in light green and light blue in the color
version) are implicitly colored with the corresponding Frobenius algebra $A$ and $B$ and
by their structure morphisms (using simplified graphical calculus, as explained in
Appendix A.8 of \cite{fusY}).
\end{example}


\Subsection{Universal correlators}

Let $\surf$ be a compact oriented surface and $\msb\iN\ocyl(\cfrc,\partial\surf)$ an
$\cfrc$-boundary datum over $\partial\surf$. Denote by $\fk\mathsf{G}_{\cfrc}(\surf,\msb)$
the vector space freely generated by the set of $\cfrc$-colored graphs on $\surf$ with
boundary datum $\msb$ (every such graph can be viewed as a complemented world sheet).
The assignment $\cws\,{\mapsto}\,\corc(\surf,\msb)$ defines a linear map 
  \be
  \corc(\surf,\msb)\Colon \fk\mathsf{G}_{\cfrc}(\surf,\msb)
  \rarr~ \snc(\surf,\df_{\partial\surf}(\msb)) \,.
  \ee
Evaluating the map $\corc(\surf,\msb)$ on any vector of the distinguished basis 
$\mathsf{G}_{\cfrc}$ of $\fk\mathsf{G}_{\cfrc}$ yields a correlator.
In this sense, $\corc(\surf,\msb)$ collects correlators for
different world sheet structures on the surface $\surf$ with boundary datum $\msb$.

The following result, which was conjectured in \Cite{Sect.\,6.2}{fusY}, shows that
in fact some of these correlators coincide.

\begin{thm} \label{thm:universal}
For every compact oriented surface $\surf$ and every $\cfrc$-boundary datum 
$\msb\iN\ocyl 
       $\linebreak[0]$
(\cfrc,\partial\surf)$ there exists a unique $\mcg(\surf)$-intertwiner
  \be
  \ucorc(\surf,\msb)\Colon \osnfrc(\surf,\msb) \rarr~ \snc(\surf,\df_{\partial\surf}(\msb))
  \ee
such that the diagram 
  \be
  \begin{tikzcd}[column sep=4.4em,row sep=3.2em]
  {\mathbbm{k}\mathsf{G}_{\mathcal{F}r(\mathcal{C})}(\varSigma,\mathsf{b})}
  & {\mathrm{SN}_{\mathcal{C}}(\varSigma,\mathbb{F}_{\partial\varSigma}(\mathsf{b}))}
  \\
  {\mathrm{SN}^{\circ}_{\mathcal{F}r(\mathcal{C})}(\varSigma,\mathsf{b})}
  \arrow[two heads, from=1-1, to=2-1]
  \arrow["{\mathrm{Cor}_\mathcal{C}(\varSigma,\mathsf{b})\,}", from=1-1, to=1-2]
  \arrow["{\mathrm{UCor}_\mathcal{C}(\varSigma,\mathsf{b})}"', dotted, from=2-1, to=1-2]
  \end{tikzcd}
  \label{eq:defUcor}
  \ee
commutes, where $\fk\mathsf G_{\cfrc}(\surf,\msb)\,{\twoheadrightarrow}\,\osnfrc(\surf,\msb)$
is the canonical quotient map.
\end{thm}

\noindent
We call this unique map $\ucorc(\surf,\msb)$ the \emph{universal correlator}
for the pair $(\surf,\msb)$.

\begin{proof}
(i) Every representative of a string net in $\osnfrc(\surf,\msb)$ can be regarded
as a complemented world sheet. Therefore commutativity of the triangle \eqref{eq:defUcor}
forces the universal correlator to be given by
  \be
  \begin{aligned}
  \ucorc(\surf,\msb)\Colon \osnfrc(\surf,\msb)
  & \rarr{\,~} \snc(\surf,\df_{\partial\surf}(\msb)) \,,
  \\
  [\cws] & \longmapsto \corc(\ws) \,. 
  \label{eq:defucor}
  \end{aligned}
  \ee
Thus we first need to show that this is indeed well defined as a linear map, i.e.\
that any two world sheets related by the graphical calculus on disks for the pivotal 
bicategory $\cfrc$ have the same correlator. To see this we assume, without loss of
generality, that $\cws_{1}$ and $\cws_{2}$ are two complemented world sheets which have
the same ambient surface $\surf$, are identical outside an embedded disk 
$D \,{\hookrightarrow}\, \surf$, and yield the same value, in the sense of 
\eqref{eq:calb12}, on the disk, i.e.\ 
$\langle\cws_{1}\cap D\rangle_{\cfrc} \eq \langle\cws_{2}\cap D\rangle_{\cfrc}$. 
Let $\corc(\ws_{1}) \eq [\grph_{1}]$ and $\corc(\ws_{2}) \eq [\grph_{2}]$
be the correlators for the corresponding world sheets, with the representative graphs
$\grph_{1}$ and $\grph_{2}$ chosen in a way such that they coincide outside $D$, and
such that for each of them there are no other Frobenius lines within $D$ apart from, for
each pair of distinct connected components (within $D$) of the partial defect network,
a single Frobenius line connecting the two components. Such a choice is possible 
because every $\cu$-conjugation of bimodule morphisms commutes with the action of the
relevant Frobenius algebras and because all the Frobenius graphs involved are \emph{full}.
It then remains to be shown that the equality $\langle\grph_1\,{\cap}\, D\rangle_{\cc} \eq 
\langle\grph_{2}\cap D\rangle_{\cc}$ holds. This is indeed the case, because the functor 
$\cu\colon\cfrc\Rarr~\cbc$ introduced in Example \ref{exa:UFrCtoBC} (with the canonical 
lax and oplax structures) is rigid separable Frobenius, so that $\cu$-conjugation 
preserves \emph{operadic compositions} and\emph{ partial trace maps} (see Theorem 
\ref{thm:RSFconj}), while the presence of the Frobenius lines which connect the connected
components (on the embedded disk $D$) of the partial defect networks compensates
for the fact that $\cu$-conjugation preserves \emph{horizontal products} and 
\emph{whiskerings} only up to idempotents of the type \eqref{eq:uconjidem}.
It is also readily clear that the linear map \eqref{eq:defucor} intertwines
the mapping class group actions.
 \\[2pt]
(ii) That the collection of universal correlators is compatible with
sewing translates exactly to the statement that the prescription of
string-net correlators is compatible with sewing.
\end{proof}

\begin{rem}
In \cite{fusY} we called the string net $[\cws]\iN\osnb(\sws,\msb_{\ws})$ a
\emph{quantum world sheet}. That the correlator $\corc(\ws)$ depends only on the
quantum world sheet $[\cws]$ implies the validity of the \emph{calculus of defects}
which is a crucial feature of conformal field theory: one is allowed to 
modify a world sheet $\ws$ locally according to the graphical calculus for the 
pivotal bicategory $\cfrc$ of defects without changing the value of its correlator 
$\corc(\ws)$. This feature is e.g.\ implicitly assumed in \cite{ffrs5}, where a less 
conceptual justification is given. In \cite{frmT} it is demonstrated
that the calculus of defects, in a higher-dimensional setting, is a powerful tool for
implementing categorical symmetries and should therefore be postulated
for \emph{every} reasonable quantum field theory with topological defects.
\end{rem}

Consider now a rigid separable Frobenius functor $F\colon\cb\Rarr~\cb'$ between two
arbitrary pivotal bicategories. Recall from Remark \ref{rem:FrobFrob} that the
lax and oplax structures of $F$ canonically provide for every object $a$ in the
domain bicategory $\cb$ a $\Delta$-separable symmetric Frobenius algebra $F(\id_{a})$
in the pivotal tensor category $\cb'(Fa,Fa)$, for every 1-morphism $f\colon a\Rarr~ b$
in $\cb$ an $F(\id_{a})$-$F(\id_{b})$-bimodule $F(f)\iN\cb'(Fa,Fb)$, 
and for every 2-morphism $\alpha$ in $\cb$ a bimodule morphism given by the 
$F$-conjugate $\alpha^F$ of $\alpha$. (As an example, for 1-morphisms 
$f\colon a\Rarr~ b$, $g\colon b\Rarr~ c$, $h\colon a\Rarr~ b'$ and
$k\colon b'\Rarr~ c$ in $\cb$, the $F$-conjugate of a 2-morphism 
$\alpha \iN \hom_{\cb(a,c)}(f\hocO g,h\hocO k)$ is a bimodule morphism
$\alpha^F \iN \hom_{\cb(Fa,Fc)}(Ff\,{\ot_{F\id_{b}}}Fg,Fh\,{\ot_{F\id_{b'}}}Fk)$.)

Just like in the case of string-net correlators, we can associate to any
$\cb$-colored graph on any oriented surface $\surf$ a $\cb'$-colored graph (and thus
a string net) on $\surf$ by performing a change of color prescribed by the 2-functor
$F$ accompanied by adding a full Frobenius graph in each 2-cell of the embedded 
graph. This is demonstrated schematically by the following figure:
  \be
  \input{SNRSF0.tikz} \quad\longmapsto\quad \input{SNRSF1.tikz}
  \ee
in which the Frobenius graphs are labeled with the images of identity
1-morphisms and lax/oplax functoriality and unit constraints. Moreover,
by an argument analogous to the one that yields Theorem \ref{thm:universal},
such a transformation of colored graphs descends to the level of string nets. We
then obtain the following result which is parallel to Theorem \ref{thm:osnb-posnb}:

\begin{thm} \label{thm:B2B'}
Let $\surf$ be a compact oriented surface, $\cb$ and $\cb'$ two
strictly pivotal bicategories, $F\colon\cb\Rarr~\cb'$ a rigid separable
Frobenius functor, and $\msb$ a $\cb$-boundary datum on $\partial\surf$.
There is a canonical $\mcg(\surf)$-intertwiner
  \be
  \mathrm{SN}_{F}^{\circ}(\surf,\msb) \Colon
  \osnb(\surf,\msb) \rarr~ \mathrm{SN_{\cb'}}(\surf,\df^{F}\msb) \,,
  \label{eq:SNoF}
  \ee
where $\df^{F}\msb$ is an object in the Karoubified cylinder category
$\cyl(\cb',\partial\surf)$. The linear map \eqref{eq:SNoF} is defined by sending
each representing $\cb$-colored graph to the $\cb'$-colored graph obtained via
$F$-conjugation, with full Frobenius graphs added.
Moreover, the collection of such intertwiners corresponding to different
surfaces and boundary data is compatible with the concatenation of string nets.
\end{thm}

\begin{rem} \label{rem:theend}
The ambient surface $\sws$ of any world sheet $\ws$ can be equipped with the
structure of an open-closed bordism; accordingly we refer to $\sws$ as an 
\emph{ambient bordism} of $\ws$. The boundary parametrization of $\sws$ is
compatible with the sewing intervals and sewing circles that are contained in the
boundary of $\ws$ (see \Cite{Sect.\,2.4}{yangYa3}). Moreover, the parametrization
map for an interval $I$ determines from the complemented world sheet $\cws$
an object in the cylinder category $\ocyl(\cfrc,\tu,I)$.
Conversely, the $\fk$-linear profunctor 
  \be
  \cosnfrc(\surf) \Colon \ocyl(\cfrc,\tu,\alpha_1) \Pto \ocyl(\cfrc,\tu,\alpha_2)
  \ee
that is the value of an open-closed bordism $\surf\colon\alpha_1\PtO\alpha_2$
under the modular functor $\cosnfrc$ classifies all world sheets with prescribed
ambient bordism $\surf$. 
These relationships fit well with the following observation \Cite{Sect.\,9}{yangYa3},
which we will further explore elsewhere: First, the modular functors $\cosnfrc$ and
$\csnc$ can be promoted to symmetric monoidal double functors
  \be
  \dosnfrc\,,\dsnc \Colon \dbord \rarr~ \dprof
  \ee
between the symmetric monoidal double categories $\dbord$ of open-closed bordisms 
(which has orientation preserving embeddings as vertical 1-morphisms) and $\dprof$
of $\fk$-linear profunctors (having linear functors as vertical 1-morphisms). And
second, the universal correlators together with the field functors 
$\{\df_\alpha\colon \ocyl(\cfrc,\tu,\alpha)\Rarr~\cyl(\cc,\alpha)\}_{\alpha\in\bbord}$
fit into a monoidal vertical transformation $\dosnfrc\,{\xRightarrow{~~}}\,\dsnc$
between these symmetric monoidal double functors.
\end{rem}


\vskip 2.4em

\noindent
{\sc Acknowledgements:}\\[.3em]
We thank Christian Blanchet, Nils Carqueville, Lukas M\"uller and Lukas Woike 
for helpful discussions. We are 
grateful to the Erwin Schr\"odinger International Institute for Mathematics and 
Physics (ESI) for the hospitality during the time part of this work was done.
 \\
JF is supported by VR under projects no.\ 2017-03836 and 2022-02931. CS and YY are 
supported by the Deutsche Forschungsgemeinschaft (DFG, German Research Foundation) 
under SCHW1162/6-1. CS is also supported by DFG under Germany's Excellence Strategy -
EXC 2121 ``Quantum Universe'' - 390833306, and YY by an Oberwolfach Leibniz 
Fellowship of the Mathematical Research Institute at Oberwolfach.

\newpage

\newcommand\wb{\,\linebreak[0]} \def\wB {$\,$\wb}
\newcommand\Arep[2]  {{\em #2}, available at {\tt #1}}
  \newcommand\Bach[2]{{\em #2}, Bachelor thesis (#1)}
\newcommand\Bi[2]    {\bibitem[#2]{#1}}
\newcommand\inBO[9]  {{\em #9}, in:\ {\em #1}, {#2}\ ({#3}, {#4} {#5}), p.\ {#6--#7} {\tt [#8]}}
\newcommand\J[7]     {{\em #7}, {#1} {#2} ({#3}) {#4--#5} {{\tt [#6]}}}
\newcommand\JO[6]    {{\em #6}, {#1} {#2} ({#3}) {#4--#5} }
\newcommand\JP[7]    {{\em #7}, {#1} (in press) {{\tt [#6]}}}
\newcommand\BOOK[4]  {{\em #1\/} ({#2}, {#3} {#4})}
\newcommand\PhD[3]   {{\em #3}, Ph.D.\ thesis #1 ({\tt #2})}
\newcommand\PhDo[2]  {{\em #2}, Ph.D.\ thesis #1}
\newcommand\Prep[2]  {{\em #2}, preprint {\tt #1}}
\newcommand\uPrep[2] {{\em #2}, unpublished preprint {\tt #1}}

\def\aagt  {Alg.\wB\&\wB Geom.\wb Topol.}     
\def\adma  {Adv.\wb Math.}
\def\anma  {Ann.\wb Math.}
\def\apcs  {Applied\wB Categorical\wB Structures}
\def\aste  {Ast\'erisque}
\def\coma  {Contemp.\wb Math.}
\def\comp  {Commun.\wb Math.\wb Phys.}
\def\crap  {C.\wb R.\wb Acad.\wb Sci.\wB Paris (S\'erie I -- Math\'ematique)}
\def\ctgc  {Cah.\wb Topol.\wb G\'eom.\wb Diff\'er.\wB Cat\'egoriques}
\def\geat  {Geom.\wB and\wB Topol.}
\def\imrn  {Int.\wb Math.\wb Res.\wb Notices}
\def\inma  {Invent.\wb math.}
\def\jhrs  {J.\wB Homotopy\wB Relat.\wB Struct.}
\def\jktr  {J.\wB Knot\wB Theory\wB and\wB its\wB Ramif.}
\def\joto  {Journal of Topology}
\def\kyjm  {Kyoto J.\ Math.}
\def\mams  {Memoirs\wB Amer.\wb Math.\wb Soc.}
\def\nupb  {Nucl.\wb Phys.\ B}
\def\nyjm  {New\wB York\wB J.\wb Math}
\def\phrb  {Phys.\wb Rev.\ B}
\def\slnc  {Springer\wB Lecture\wB Notes\wB in\wB Comp.\wb Science}
\def\quto  {Quantum Topology}
\def\sigm  {SIGMA}
\def\taac  {Theory\wB and\wB Appl.\wb Categories}
\def\topo  {Topology}
\def\trgr  {Transformation\wB Groups}


\end{document}